\documentclass[12pt,reqno]{amsbook}

\usepackage[top=3cm, bottom=2.5cm, left=2cm, right=2cm]{geometry}
\usepackage{amsmath,amssymb,amsthm,mathrsfs}
\usepackage[pdftex]{graphicx}
\usepackage{epsfig}
\usepackage[english]{babel}
\usepackage[usenames,dvipsnames]{color}
\usepackage{xcolor}
\usepackage{enumerate}
\usepackage{float}
\usepackage{slashed}
\usepackage[color]{showkeys}  
\usepackage[pdftex,pdfpagelabels,bookmarks,hyperindex,hyperfigures,pagebackref]{hyperref} 
\usepackage{setspace}

\usepackage[OT2,OT1]{fontenc} \newcommand\cyr {
  \renewcommand\rmdefault{wncyr} \renewcommand\sfdefault{wncyss}
  \renewcommand\encodingdefault{OT2} \normalfont \selectfont }
\DeclareTextFontCommand{\textcyr}{\cyr}

\newcommand{\rb}{\mbox{\kern -1.3em {\cyr b}}}

\theoremstyle{plain}
\newtheorem{theorem}{Theorem}[section]
\newtheorem{proposition}{Proposition}[section]
\newtheorem{definition}{Definition}[section]
\newtheorem{remark}{Remark}[section]
\newtheorem{lemma}{Lemma}[section]
\newtheorem{corollary}{Corollary}[theorem]
\theoremstyle{definition}
\newtheorem{example}{Example}[section]
\newtheorem{exercise}{Exercise}[section]

\newcommand{\inner}[1]{\left\langle  #1 \right\rangle }

\newcommand{\Z}{{\mathbb Z}}

\newcommand{\e}{{\mathbf{e}}}

\newcommand{\calC}{{\mathcal C}}

\newcommand{\supp}[1]{{\operatorname{supp}#1}}

\numberwithin{equation}{section}
\numberwithin{figure}{section}
\numberwithin{section}{chapter}

\begin{document}

\frontmatter

\title[A Primer on Script Geometry]{A Primer on Script Geometry}

 \author[P. Cerejeiras]{Paula Cerejeiras}

\author{Uwe K\"ahler}%

\author{Teppo Mertens}%

\author{Frank Sommen}%

\author{Adrian Vajiac}%

\author{Mihaela Vajiac}%


\date{\today}

\maketitle
\setcounter{tocdepth}{3}
\tableofcontents

\mainmatter

\tableofcontents

\chapter{Introduction}

In the last two decades one can observe an increasing interest in the analysis of discrete structures. On the one hand this increasing interest is based on the fact that increased computational power is nowadays available to everybody and that computers can essentially work only with discrete values. This means that one requires discrete structures which are inspired by the usual continuous structures. On the other hand, the increased computational power also means that problems in physics which are traditionally modeled by means of continuous analysis are more and more directly studied on the discrete level, the principal example being the Ising model from statistical physics as opposed to the continuous Heisenberg model. Another outstanding example can be seen in the change of the philosophy of the Finite Element Method. The classical point of view of the Finite Element Method is to start from the variational formulation of a partial differential equation and to apply a Galerkin-Petrov or a Galerkin-Bubnov method via a neste sequence of finite-dimensional subspaces. These are created by discretizing the continuous domain by a mesh and to construct the basis functions of the finite-dimensional spaces as functions over the mesh. The modern approach lifts the problem and its finite element modelation directly on to the mesh resulting in the so-called Finite Element Exterior Calculus. The basic idea behind this discrete exterior calculus is that large classes of mixed finite element methods can be formulated on Hilbert complexes where one solves the variational problem on finite-dimensional subcomplexes. This not only represents a more elegant way of looking at finite element methods, but it also has two practical advantages. First of all it allows a better characterization of stable discretizations by requiring two hypotheses: they can be written as a subcomplex of a Hilbert complex and there exists a bounded cochain projection from that complex to the subcomplex~\cite{Arnold2}.  This was later on extended to abstract Hilbert complexes~\cite{Stern}.  Secondly, it mimics the engineer's approach of directly performing finite element modeling on the mesh.

The principal example of this approach is the Hodge-deRham complex for approximating manifolds. 
Maybe it is worthwile to point out that the underlying ideas are much older.  Whitney introduced his complex of Whitney forms in 1957~\cite{Whitney}. Among other things, he used them to identity the de Rham cohomology with simplicial cohomology. While these was done with purely geometric applications in mind later in it was shown that Whitney forms are finite elements on the deRham complex. Nevertheless, it is interesting to note that the original idea was purely geometric in nature. 

Another example of this approach can be found in computational modeling ~\cite{Desbrun}. There, a discrete exterior calculus based on simplicial co-chains is introduced. One of the advantages is that it avoids the need for interpolation of forms and many important tools could be obtain like discrete exterior derivative, discrete boundary and co-boundary operator. An important step consisted also in the establishment of a discrete Poinca\'e lemma. It states that given a closed $k$-cochain $\omega$ on a (logically) star-shaped complex, i.e. $d\omega=0$ there exists a $(k-1)$-chochain $\alpha$ such that $\omega=d\alpha$. 

While standard Whitney forms are linked to barycentric coordinates and, therefore, can be easily adapted to more general meshes, a large part of the above mentioned applications of Hodge theory to discrete structures are linked to simplicial complexes which are not that easily adapted to more general meshes. 

To overcome this problem we are going to present a new type of algebraic topology based on the concept of scripts. A priori scripts are based on complexes, but more general than simplicial complexes. It is based on more geometrical constructions which also makes this concept rather intuitive. To make that clear and to make understanding easy we provide many concrete (classic) examples, including the torus, Klein bottle, and projective plane. Also, newly introduced  notions and operations will always be accompanied by concrete examples so as to make understanding easier for the reader. As will be seen many of these notions and operations are rather intuitive while at the same time provide a more geometric understanding than classic approaches. 

One of the key points in this theory is the concept of tightness which replaces the need for the establishment of a Poincar\'e lemma. Hereby, tighness imposes cells and chains to be minimal which is in fact what the geometric meaning of the Poincar\'e lemma represents. This can easily be seen if one notices that tightness means that the local homology at the level of cells is trivial which corresponds to the Poincar\'e lemma for manifolds which says that each point has a neighborhood with trivial homology. 

In Chapter 2 we introduce the basic concepts, including the geometrical offprint of a script, equivalent and unitary scripts. The geometric offprint or skeleton of a script as a the support of a boundary chain will provide us with all the necessary geometrical information so as to represent the geometric boundary of a chain. 

In Chapter 3 we are going to discuss the geometrical properties of scripts. This is closely linked to minimising and uniqueness properties of scripts. In particular the question of the skeleton being a unique minimal script will lead us to the central notion of tightness. A variety of examples will show that tightness is indeed a geometric and intuitive notion. 

One of the essential parts in possible applications is the possibility to manipulate scripts. In Chapter 4 we present and discuss basic operations, such as creation and cleaning (removing) operations as well as identification operations. Again a variety of examples will be given. 

In Chapter 5 we introduce the necessary concepts of metrics on scripts, dual scripts, and the corresponding Dirac and Laplace operators. This will be the groundwork for a function theory of monogenic and harmonic functions on our discrete structures.  

Finally, Chapter 6 will be dedicated to the question of Cartesian products on scripts. We will give two types of Cartesian products on scripts and discuss tighness in this context. As always examples will be provided. Additionally, the introduction of a Cartesian product also allows us to give a notion of discrete curvature in the two-dimensional case which is much more intuitive than the standard notion.


\chapter{Scripts in general}

\section{Complexes}
\begin{definition}
A {\em complex} is a 
(finite or infinite) sequence of modules together with boundary maps $\partial_i:\mathcal{M}_i\longrightarrow\mathcal{M}_{i+1}$ such that $\partial_{i+1}\circ\partial_{i}=0$.
\end{definition}
The starting point for scripts is the idea of a complex of free
modules over $\Z$ together with boundary maps, with certain properties.
A script is a special sequence of modules:
\begin{equation}
  \label{def:script}
  \mathcal{M}_{-2} \longleftarrow \mathcal{M}_{-1} \overset{\partial}{\longleftarrow} \mathcal{M}_{0}
  \overset{\partial}{\longleftarrow} \mathcal{M}_{1} \overset{\partial}{\longleftarrow}\mathcal{M}_{2}
  \overset{\partial} {\longleftarrow} \cdots 
\end{equation}
whereby $\partial : \mathcal{M}_{k} \rightarrow \mathcal{M}_{k-1}$ is
a linear map called {\em boundary map} satisfying to
$\partial \circ \partial = \partial^2 =0$.  We have the following
terminology:
\begin{itemize}
\item[(i)] $\mathcal{M}_{-2}=\{0\}$;
\item[(ii)] $\mathcal{M}_{-1}=\mathbb{Z}$ is the {\em accumulator
    module}, generated by $1$, which is called {\em accumulator};
\item[(iii)] $\mathcal{M}_k$ is {\em the module of $k$--chains},
  defined as a free $\mathbb{Z}$--module generated over a set
  $\mathcal{C}_k = \{ C^k_j \}_{j\in J}$ of so-called {\em $k$--cells}. An
  element of $\mathcal{M}_k$ is called a {\em $k$--chain}, thus we
  write:
  \begin{equation}
    \label{module_M_k}
    \mathcal{M}_{k} =  \left\{ C^k=\sum_{j\in J} \lambda_j C^k_j : \lambda_j \in \mathbb{Z}, C^k_j \in \mathcal{C}_k \right\}.
  \end{equation}
\end{itemize}
The lower index spaces of $k$--cells have special terminology:
\begin{itemize}
\item[(a)] $\mathcal{C}_0 = \{p_j=C^0_j \}_{j\in J}$ is called the set of
  {\em points};
\item[(b)] $\mathcal{C}_1 = \{ \ell_j=C^1_j \}_{j\in J}$ is called the
  set of {\em lines};
\item[(c)] $\mathcal{C}_2 = \{v_j=C^2_j \}_{j\in J}$ is called the set
  of {\em planes}.
\end{itemize}
Using this notation, we write, for example:
\begin{equation}
  \label{module_M_0}
  \mathcal{M}_{0} = \left\{ \sum_{j\in J} \lambda_j p_j : \lambda_j \in \mathbb{Z}, p_j \in \mathcal{C}_0 \right\}.
\end{equation}
In conclusion, we define a script as follows.
\begin{definition}
  A {\em script} is a complex of free modules $\mathcal{M}_k$  over $\Z$ of type: 
\begin{align}
  \label{def:script1}
  0 \longleftarrow \mathbb{Z} \overset{\partial}{\longleftarrow} \mathcal{M}_{0}
  \overset{\partial}{\longleftarrow} \mathcal{M}_{1} \overset{\partial}{\longleftarrow}\mathcal{M}_{2}
  \overset{\partial} {\longleftarrow} \cdots 
\end{align}
generated by the spaces of $k$--chains $\mathcal{C}_k$ together with
the boundary map $\partial$ at each level.  The {\em dimension of the
  script} is the largest $n$ for which $\mathcal{M}_n\neq \emptyset$. If $\mathcal{M}_n\neq \emptyset$ for all $n$, then the script is said to be an {\em infinite script}.
\end{definition}

\section{Immediate examples}

\begin{example}[Addition on $\mathbb{Z}$]
  \label{example:addition_Z}
  Consider the $0$--dimensional script
  \begin{align}
    \label{addition_Z_complex}
    0 \longleftarrow \mathbb{Z} \overset{\partial}{\longleftarrow} \mathcal{M}_{0}
  \end{align}
  with $\partial p_j =1$ for all $p_j \in \mathcal{C}_0$. In this case we have
  $\displaystyle\partial\left( \sum_{j\in J} \lambda_j p_j \right) =
  \sum_{j\in J} \lambda_j,$ so it represents the usual addition in
  $\mathbb{Z}$.
\end{example}

\begin{example}[An interval]
  \label{example:interval}
  The following $1$--dimensional script
  \begin{align}
    \label{interval_complex}
    0 \longleftarrow \mathbb{Z} \overset{\partial}{\longleftarrow} \mathcal{M}_{0} \overset{\partial}{\longleftarrow} \mathcal{M}_{1}
  \end{align}
  where $\mathcal{C}_0 = \{ p, q\}, \mathcal{C}_1 = \{ \ell \},$
  $\partial p = \partial q =1,$ and $\partial \ell = p-q$, represents
  an interval.
\end{example}

\begin{example}[Circles, spheres, etc.]
  \label{example:circles_spheres}
  The $1$--dimensional script
  \begin{align}
    \label{circle_complex}
    0 \longleftarrow \mathbb{Z} \overset{\partial}{\longleftarrow} \mathcal{M}_{0}
    \overset{\partial}{\longleftarrow} \mathcal{M}_{1}
  \end{align}
  with $\mathcal{C}_0 = \{ p_1, p_2 \}$,
  $\mathcal{C}_1 = \{ \ell_1, \ell_2 \}$, $\partial p_j =1$, and
  $\partial \ell_j = p_1-p_2, ~j=1,2$, represents a circle.  The
  extension of this script to
  \begin{align}
    \label{sphere_complex}
    0 \longleftarrow \mathbb{Z} \overset{\partial}{\longleftarrow} \mathcal{M}_{0}
    \overset{\partial}{\longleftarrow} \mathcal{M}_{1} \overset{\partial}{\longleftarrow} \mathcal{M}_{2}
    \end{align}
    with $\mathcal{C}_0, \mathcal{C}_1$ as before,
    $\mathcal{M}_2 = \{ v_1, v_2 \},$ and extra relations
    $\partial v_j = \ell_1- \ell_2, ~j=1,2,$ represents a $2$--sphere
    in an elementary form.

    In general, the extension
    \begin{align}
      \label{m_sphere_complex}
      0 \longleftarrow \mathbb{Z} \overset{\partial}{\longleftarrow} \mathcal{M}_{0}
      \overset{\partial}{\longleftarrow} \mathcal{M}_{1} \overset{\partial}{\longleftarrow}
      \cdots \overset{\partial}{\longleftarrow} \mathcal{M}_{k} \overset{\partial}{\longleftarrow}
      \cdots \overset{\partial}{\longleftarrow} \mathcal{M}_{m}
    \end{align}
    where $\mathcal{C}_k = \{ C^k_1, C^k_2 \}$ and
    $\partial C^k_j = C^{k-1}_2 - C^{k-1}_1$ represents a $m$--sphere.
\end{example}

\begin{example}[Simplexes]
  \label{example_simplexes}
  Regular simplexes are a special case of scripts: consider the sets
  of $k$--cells (points, lines, etc.) as follows:
 \begin{align*}
 \mathcal{C}_0&= \{ [0], [1], \cdots, [m] \}, \\
  \mathcal{C}_1&= \{ [i,j], i,j =0, \cdots, m \}, \\
  \vdots \\
  \mathcal{C}_k&= \{ [\alpha_0, \cdots, \alpha_k]: 0\leq \alpha_0 <  \cdots < \alpha_k \leq m\}, \\
    \vdots \\
  \mathcal{C}_m &= \{ [0, 1, \cdots, m] \}
  \end{align*}
 and define the boundary map by:
 \begin{align}
   \label{simplex_boundary}
   \partial [\alpha_0, \cdots, \alpha_k] = \sum_{j=0}^k (-1)^{j} [\alpha_0, \cdots, \alpha_k]_j^{\hat ~}
   = \sum_{j=0}^k (-1)^{j} [\alpha_0, \cdots, \alpha_{j-1},  \alpha_{j+1}, \cdots, \alpha_k].
 \end{align}
 It is easy to see that the script defined this way represents a
 regular simplex.
\end{example}

\section{The geometrical offprint of a script}

\begin{definition}
  Let $C^k = \displaystyle\sum_{j\in J} \lambda_j C^k_j$ be a general
  $k$--chain.  The {\em support} of a single $k$--cell is
  itself and, in general, it is 
  denoted by
  \begin{align*}
    \supp{C^k} = \{C_j^k \,\big|\, \lambda_j \not=0 \}.
  \end{align*}
  Moreover, we denote by $\rb C^k =$ \mbox{supp} $\partial C^k$ the
  so--called {\em geometrical boundary} of the chain $C^k.$
\end{definition}

Therefore there are natural maps
\begin{align}
  \label{geo_offprint}
  \{ 1 \} \overset{\rb}{\longleftarrow}  {\mathcal{P}}(\mathcal{C}_0) \overset{\rb}{\longleftarrow} {\mathcal{P}}(\mathcal{C}_1)
  \overset{\rb}{\longleftarrow} {\mathcal{P}}(\mathcal{C}_2) \overset{\rb}{\longleftarrow}
  \cdots \overset{\rb}{\longleftarrow} {\mathcal{P}}(\mathcal{C}_k)  \overset{\rb}{\longleftarrow} \cdots
\end{align}
which represent so--to--speak the {\em geometrical offprint of the
  script}.

\section{Subscripts}
\begin{definition}
  Consider a script~\eqref{def:script} and let
  \begin{equation}
    \label{Eq:1.011}
    0 \longleftarrow \Z \overset{\partial^\prime}{\longleftarrow} \mathcal{M}^\prime_{0}
    \overset{\partial^\prime}{\longleftarrow} \mathcal{M}_{1}^\prime \overset{\partial^\prime} {\longleftarrow}
    \cdots 
  \end{equation}
  be another script for which
  \begin{equation}
    \label{Eq:1.012}
    \mathcal{C}_k' \subset \mathcal{C}_k, \quad k=0, 1, \ldots 
  \end{equation}
  and such that if $C^{'k}_j \in \mathcal{C}_k'$ then
  $\rb C^{'k}_j \subset \mathcal{C}_{k-1}'$ and
  $\partial' C^{'k}_j =\partial C^{'k}_j$, for all $k=0, 1, \ldots$
  Then we call this new script a {\em subscript} of the original
  script.
\end{definition}

In particular for a $k$--cell $C^k_j$ we may consider
\begin{equation}
  \label{Eq:1.013}
  \mathcal{C}_k' = \{ C^{k}_j \}, \quad \mathcal{C}_{k-1}' =  \rb  C^{k}_j, \quad \dots, \quad\mathcal{C}_{k-l}' =  \rb^l  C^{k}_j,\quad \dots
\end{equation}
then the corresponding script
\begin{equation}\label{Eq:1.014}
  0 \overset{\partial}{\longleftarrow} \mathbb{Z} \overset{\partial}{\longleftarrow} \mathcal{M}_{0}(\mathcal{C}^\prime_0)
  \overset{\partial}{\longleftarrow} \mathcal{M}_{1}(\mathcal{C}^\prime_1)\cdots
  \overset{\partial}{\longleftarrow} \{ \lambda C^k_j\,\big|\,\lambda \in \mathbb{Z}\}
\end{equation}
is called the \textit{subscript generated by} $\{ C^k_j \}.$

More general, for a subset $A \subset \mathcal{C}^k$ we may consider
the subscript for which
\begin{equation}\label{Eq:1.015}
  \mathcal{C}^\prime_k = A, \quad \mathcal{C}^\prime_{k-1} = \rb  A
  := \bigcup_j \rb C^k_j, \, C^k_j \in A, \dots, \mathcal{C}^\prime_{k-l} = \rb^l  A.
\end{equation}
This is called the {\em subscript generated by} $A.$

\begin{example}
  \label{Ex:1.0005}
  A subscript of a symplex is called a simplicial complex. It can be
  generated by a subset $A$ of a $k$--dimensional subsimplexes of the
  overall $m$--symplex $[0,\dots, m].$
\end{example}

\section{Equivalent scripts}

Consider a cell $C^k_j \in \mathcal{C}_k$ and replace $\mathcal{C}_k$
by
$\mathcal{C}_k^{\prime} = (\mathcal{C}_k \setminus \{ C^k_j \}) \cup
\{ C'^{\,k}_j \}, $ where we set $C'^{\,k}_j = \pm C^{k}_j$ and
$\partial C'^{\,k}_j = \pm \partial C^{k}_j,$ and whenever
$C^{k}_j \in \rb C^{k+1}_l$ and
$$
\partial C^{k+1}_l = \lambda_j C^k_j + \sum_{i\neq j} \lambda_i C^k_i
$$
we replace $\partial C^{k+1}_l$ by
$\pm \lambda_j C'^{\,k}_j +\displaystyle\sum_{i\neq j} \lambda_i C^k_i$.

Then the newly obtained script is called an {\em equivalent script}.
Note that, by iteration, this definition includes permutation of
indices as well since it just corresponds to changing the names of objects.  Clearly, the $\rb$--maps for equivalent scripts are essentially the same and they have the same geometrical offprint. The converse is usually not true.

\section{Unitary scripts}

\begin{definition}
  A script is called {\em unitary} if for every $k$--cell $C_j^k,$ the
  boundary map
  $\partial C^{k}_j = \displaystyle\sum_i \lambda_i C^{k-1}_i$ only
  involves the values $\lambda_i = \pm 1$ (that is, whenever
  $\lambda_i \not=0$ ).
\end{definition}
Given a candidate for the geometrical offprint~\eqref{geo_offprint}
of a unitary script
\begin{equation}
  \label{Eq:1.017}
  0 \overset{\partial}{\longleftarrow} \mathbb{Z} \overset{\partial}{\longleftarrow} \mathcal{M}_{0}
  \overset{\partial}{\longleftarrow} \mathcal{M}_{1} \overset{\partial}{\longleftarrow} \cdots
\end{equation}
it may happen that any other unitary script with the same geometrical offprint is equivalent to this script.  In those cases the script is determined by its geometrical offprint up to equivalence.  This property motivates the previous definition of equivalence of scripts. 

In fact a unitary cell $C^k_j$ has a boundary $\partial C^k_j$ that
can be seen as a surface $\rb C^k_j$ with an orientation on it. More
general, we may also consider {\em unitary chains}
$C^k = \displaystyle\sum_j \lambda_j C^k_j$, with $\lambda_j = \pm 1.$

\section{Cycles, boundary, and homology}

\begin{definition}
  A $k$--chain $C^k \in \mathcal{M}_k(\mathcal{C}_k)$ that is {\em
    closed}, i.e. $\partial C^k=0$, is called a $k$--{\em cycle}. By
  $\mathcal{Z}_k({\mathcal{C}_k})$ we denote the module of all
  $k$--cycles.

  A $k$--cycle $C^k \in \mathcal{Z}_k(\mathcal{C}_k)$ is called a
  $k$--{\em boundary} if for some
  $C^{k+1} \in \mathcal{M}_{k+1}(\mathcal{C}_{k+1})$ we have
  $C^k = \partial C^{k+1}$. By $\mathcal{B}_k(\mathcal{C}_k)$ we
  denote the module of $k$--boundaries.  The $k$--th homology space of
  the script is given by
  \begin{align}
    \label{k_homology_space}
    \mathcal{H}_k(\mathcal{C}_k) := \mathcal{Z}_k(\mathcal{C}_k) /
    \mathcal{B}_k(\mathcal{C}_k).
  \end{align}
\end{definition}

One can also define local modules: let
$\mathcal{U} \subset \mathcal{C}_k$, then
$\mathcal{M}_k(\mathcal{U}), \mathcal{Z}_k(\mathcal{U}),
\mathcal{B}_k(\mathcal{U}),$ and $\mathcal{H}_k(\mathcal{U})$ denote
the modules of $k$--chains, $k$--cycles, $k$--boundaries, and
$k$--homology of $\mathcal{U}$, respectively.

One can also define relative homology. For that we extend the boundary
$\rb$ (which up to now is only defined for $k$--cells and sets of
$k$--cells) to $k$--chains.  For
$C^k \in \mathcal{M}_k(\mathcal{C}_k)$ we set
\begin{equation}
  \label{Eq:1.018}
  \rb(C^k) := \supp(\partial C^k) \subset \mathcal{C}_{k-1}.
\end{equation}
Next, let $\mathcal{U} \subset \calC_k, \mathcal{V} \subset \calC_{k-1},$ then
by $\mathcal{Z}_k(\mathcal{U}, \mathcal{V})$ we denote the module of
$k$--chains $C^k \in \mathcal{M}_k(\mathcal{U})$ for which
$\rb(C^k) \subset \mathcal{V}$ or also $\partial C^k \in \mathcal{B}_{k-1}(\mathcal{V}).$

By $\mathcal{U}_{\mathcal{V}}$ we denote the subset of $k$--cells
$C^k_j \in \mathcal{U}$ for which $\rb(C^k_j) \subset \mathcal{V}$ and
we denote by $\mathcal{B}_k(\mathcal{U}, \mathcal{V})$ the module of
$k$--chains $C^k \in \mathcal{Z}_k(\mathcal{U}, \mathcal{V})$ of the
form
\begin{equation}
  \label{Eq:1.019}
  C^k = C'^k+C''^k, \quad  C'^k \in \mathcal{B}_k(\mathcal{U}),
  C''^k \in \mathcal{M}_k(\mathcal{U}_{\mathcal{V}}).
\end{equation}
Clearly, also
$\mathcal{M}_k(\mathcal{U}_{\mathcal{V}}) \subset
\mathcal{Z}_k(\mathcal{U}, \mathcal{V})$.  By
$\mathcal{H}_k(\mathcal{U}, \mathcal{V}) = \mathcal{Z}_k(\mathcal{U},
\mathcal{V}) / \mathcal{B}_k(\mathcal{U}, \mathcal{V})$ we denote the
homology module of $\mathcal{U}$ relative to $\mathcal{V}.$ In this
way everything is naturally defined and above all, crystal clear.

\section{Other rings}

We presented the theory of scripts over the ring $\mathbb{Z}$ of
integers. Sometimes it will be useful to allow more values like the
field of rational numbers $\mathbb{Q}$ (e.g. to study invertible
morphisms). Moreover, one can also consider the scripts over other
rings like $\mathbb{Z}/n\Z$ ($n \in \mathbb{Z}$), or polynomials.

For any script over $\mathbb{Z}$
$$
0 \longleftarrow \mathbb{Z} \overset{\partial}{\longleftarrow}
\mathcal{M}_{0} \overset{\partial}{\longleftarrow} \mathcal{M}_{1}
\overset{\partial}{\longleftarrow} \cdots
$$
we can consider the script over $\mathbb{Z}/n\Z$:
\begin{equation}
  \label{Eq:1.020}
  0 \longleftarrow \mathbb{Z}/n\Z
  \overset{\partial_n}{\longleftarrow} \Pi_n(\mathcal{M}_{0}) \overset{\partial_n}{\longleftarrow}
  \Pi_n(\mathcal{M}_{1}) \overset{\partial_n}{\longleftarrow} \cdots
\end{equation}
whereby $\Pi_n : \mathbb{Z} \rightarrow \mathbb{Z}/n\Z$ is the natural
projection and $\partial_n = \Pi_n \circ \partial.$

In case $n=3$, $\Pi_3$ leaves unitary scripts invariant because then
$\Pi_3(\partial C^k_j) = \partial C^k_j$ for every $k$--cell $C^k_j$.
Moreover, every script over $\mathbb{Z}/3\Z$ is by definition unitary
and $n=3$ is the lowest case for which every cell has 2 states of
orientation.

We note that not all scripts over $\mathbb{Z}/n\Z$ ($n\ge 4$) are
unitary. We will see later the script for the Klein bottle (\ref{fig:klein}) is
not unitary.

Yet one can also consider $\mathbb{Z}/2\Z$ and the projection $\Pi_2$.
In this case orientability is no longer an issue and in fact every
$k$--chain has the form $C^k = \displaystyle\sum_{j \in A} C^k_j,$
$A \subset \mathcal{C}_k,$ so that the map
$C^k \rightarrow \supp {C^k}$ from $\mathcal{M}_k(\mathcal{C}_k)$ to
$\mathcal{P}(\mathcal{C}_k)$ is bijective. In particular, for every
cell $C^k_j$, its boundary $\partial C^k_j$ is mapped bijectively on
$\rb (C^k_j)$ and so there is a one--to--one correspondence between a
$\mathbb{Z}/2\Z$--script and its geometrical offprint.

Also, for every $C^k \in \mathcal{M}_k$, the boundary $\partial C^k$
may be identified with $\rb C^k = \supp{\partial C^k}$. Moreover,
$$
\rb (\supp{C^k}) = \bigcup_j \rb (C^k_j),\quad  C^k_j \in \supp{C^k}
$$
and one obtains:
$$
\rb (\supp{C^k}) = \bigcup_j \supp(\partial C^k_j),
$$
(with $C^k_j \in \mbox{supp } C^k$) which is bigger than:
$$
\rb (C^k)=\supp(\partial C^k)=\supp{}\left(\partial\sum_j\lambda_j
  C^k_j\right) =\supp{}\left(\sum_j\lambda_j \partial C^k_j\right).
$$
Thus, we have the following:

\begin{lemma}
  $\rb(C^k)\subset \rb(\mbox{supp } C^k)$. In general, the inclusion
  is strict.
\end{lemma}


Whereby, it is best to not fully identify $C^k$ with $\supp{C^k}$.
However, for a unitary script over $\mathbb{Z},$ the operator $\Pi_2$
may be identified with the projection on the geometrical offprint, or
skeleton; $\Pi_2$ is a kind of R\"ontgen image.

\section{Clifford algebra}

A way to encode simplexes is given by a Clifford algebra of the
appropriate dimension as follows.  Consider $m+1$ points and attach to
them the $m+1$ basis elements $\e_0, \cdots, \e_m$ generating the
Clifford algebra $\mathbb{R}_{m+1,0}$ with relations
\begin{equation}
  \label{Eq:1.021}
  \e_j \e_k + \e_k \e_j = 2\delta_{j,k}, \quad j,k \in \{ 0, \dots, m\}.
\end{equation}
Then every basis element of $\mathbb{R}_{m+1,0}$ has the form
\begin{equation}
  \label{Eq:1.022}
  \e_A = \e_{j_0} \cdots \e_{j_k}, \quad A = \{ j_0, \dots,  j_k \}, ~s.t. ~ 1\leq j_0 < \cdots <  j_k \leq m.
\end{equation}
We now identify basis elements with symplexes
\begin{equation}
  \label{Eq:1.023}
  \e_A = \e_{j_0} \cdots \e_{j_k} \quad \rightarrow \quad [A]= [ j_0, \dots,  j_k ].
\end{equation}
Then any $k$--vector $\displaystyle\sum_{|A|=k} \lambda_A \e_A$ is
mapped isomorphically on the $k$--chain
$\displaystyle\sum_{|A|=k} \lambda_A [A] \in \mathcal{M}_k
(\mathcal{C}_k).$

Next, let $\e = \e_0+\e_1 + \cdots +\e_m;$ then
\begin{equation}
  \label{Eq:1.024}
  \e \cdot \e_A := \left(\sum_{j=0}^m \e_{j} \right) \cdot \e_{A}
  = [\e ~ \e_A]_{k-1} \longrightarrow \partial [ A] = \partial [ j_0, \dots,  j_k ].
\end{equation}
This is called the Clifford algebra representation of simplicial
complexes which it turns out to be very useful.  Note that here
$\cdot$ denotes the inner product, not the regular Clifford product and
the boundary operator is well defined.


\chapter{Geometrical properties of scripts}

\section{Minimization}
General complexes are too general for the sake of their intrinsic
geometries and there are a number of elementary properties one may
assume.  We say that a cell $C^k_j$ is in {\em minimal state} if
$$
\partial C^k_j = \sum_{l} \lambda_j^l C^{k-1}_l
$$
with $\gcd_l (\lambda_j^l) =1$ (where $\gcd_l$ is the usual greatest
common divisor w.r.t. the index $l$).

\begin{proposition}\label{Lm:2.001}
  Every complex has a canonical minimization.
\end{proposition}
\begin{proof}
  Assume we already minimized
  $\mathcal{C}_0, \dots,\mathcal{C}_{k-1},$ and let
  $C^k_j \in \mathcal{C}_k.$ If $\partial C^k_j =0$ we remove $C^k_j$
  from $\mathcal{C}_k$ and also from any $\partial C^{k+1}_s$ in which
  it occurs. Let $g =\gcd_l (\lambda_j^l) >1$ then replace $C^k_j$ by
  $C^{\prime k}_j =\displaystyle \frac{1}{g}C^k_j$ in $\mathcal{C}_k$
  and by $g C^{\prime k}_j$ in any $\partial C^{k+1}_s$ where it
  occurs. Note that this may change $C^{k+1}_s$ from minimal to
  non--minimal.
\end{proof}

Without too much loss of generality one may hence assume scripts to be
minimal. Unitary scripts \textit{are already} minimal.

\section{The skeleton problem}

Let
$$
0 \longleftarrow \mathbb{Z}
\overset{\partial}{\longleftarrow} \mathcal{M}_{0}
\overset{\partial}{\longleftarrow} \mathcal{M}_1
\overset{\partial}{\longleftarrow} \cdots
$$
be a minimal script and let
\begin{equation}\label{Eq:1.025}
  \cdots  \overset{\rb}{\longleftarrow} \mathcal{P}(\mathcal{C}_{k-1})
  \overset{\rb}{\longleftarrow} \mathcal{P}(\mathcal{C}_{k}) \overset{\rb}{\longleftarrow} \cdots
\end{equation}
be its skeleton (or geometrical offprint). In general there may exist other scripts with the
same skeleton. This leads to the following:

\textbf{Problem:} When does it happen that a skeleton \eqref{Eq:1.025}
corresponds to a unique {\em minimal} script (up to equivalence)?

\begin{remark}
  In what follows we may assume that $\mathcal{C}_k$ has no redundant
  cells, i.e. cells $C^k_j$ that do not appear in any
  $\rb(C^{k+1}_j)$.  In this case the skeleton has the form
  \begin{equation}
    \label{Eq:1.026}
    \mathcal{P}(\mathcal{C}_m)  \overset{\rb}{\longrightarrow}  \mathcal{P}(\rb(\mathcal{C}_{m}))
    =  \mathcal{P}(\mathcal{C}_{m-1}) \overset{\rb}{\longrightarrow}   \mathcal{P}(\rb^2(\mathcal{C}_{m}))
    =  \mathcal{P}(\mathcal{C}_{m-2}){\longrightarrow} \cdots
  \end{equation}
\end{remark}

The notion of tightness provides an answer to this problem.
\section{Tight scripts: definitions}

\begin{definition}
  Let $\mathcal{U} \subset \mathcal{C}_k$; then $\mathcal{U}$ is
  called {\em set tight} and we write {\em $s$--tight} if
  $\mathcal{Z}_k(\mathcal{U})$ is generated by a single cycle $C^k$.
  By definition such a cycle will be minimal.

  A cycle $C^k$ is called {\em cycle tight} and we write {\em
    $c$--tight }if $\supp{C^k}$ is $s$--tight and $C^k$ also generates
  $\mathcal{Z}_k(\supp{C^k})$. In this case $C^k$ is minimal and
  $\supp{C^k}$ is $s$--tight.

  A single cell $C^k_j$ is called {\em tight} if
  $\mathcal{Z}_k(\rb C^k_j)$ is $s$--tight and generated by
  $\partial C^k_j$. The interested reader will see that this means
  that $C^k_j$ is minimal.
\end{definition}

\begin{definition}
  A script is {\em tight} if and only if each of its cells is tight.
\end{definition}

\begin{definition}
  Let $\mathcal{U} \subset \mathcal{C}_k$ and $\mathcal{V}  \subset \mathcal{C}_{k-1}$ then
  $\mathcal{U}$ is {\em $s$--tight relative} to $\mathcal{V}$ if
  $\mathcal{Z}_k(\mathcal{U}, \mathcal{V})$ is generated by a single
  chain $C^k.$

  A chain $C^k$ is called {\em $c$--tight } if $C^k$ generates
  $\mathcal{Z}_k(\supp{C^k}, \rb (C^k))$ i.e. $\supp{C^k}$ is
  $s$--tight relative to $\rb(C^k)$ and $C^k$ is minimal.
 \end{definition}

 \begin{remark}
   We use the same notation in the two definitions since in the case
   where a chain $C^k$ is a cycle we have that $\rb(C^k)=\emptyset$
   and the definitions agree.
 \end{remark}

\section{Elementary properties of tight scripts}

First note that in a minimal script we may assume that for every point
$p \in \mathcal{C}_0, \partial p = 1$ so that
$\partial : \mathcal{M}_0 \rightarrow \mathbb{Z}$ corresponds to
integration (summation).

We prove the following structure theorem for cells of dimension $1$ in tight scripts:

\begin{theorem}\label{Lm:2.002}
  In a tight script every line $\ell \in \mathcal{C}_1$ may be
  interpreted as an oriented line from a point $p$ to another point
  $q$, i.e. $\partial \ell = q-p, ~p, q \in \mathcal{C}_0.$
\end{theorem}

\begin{proof} The case $\partial \ell =0$ is pathological and the case
  $\rb \ell = \{ p \}$ does not occur since
  $\mathcal{Z}_0(\{ p \}) =0.$ Also in case
  $\rb \ell = \{p, q, r, \ldots \},$ $\mathcal{Z}_0(\rb \ell)$ has at
  least 2 generators $r-q$ and $q-p.$ Therefore we must have that
  $\rb \ell = \{p, q \},$ where $p, q \in \mathcal{C}_0, p\not=q,$ and
  $\mathcal{Z}_0(\{p, q \})$ is obviously generated by $q-p.$
\end{proof}

In the case of $2-$cells, we first define the notion of {\em polygon}:

\begin{definition}\label{def:2.001}
  Let $\ell_1, \cdots, \ell_n$ be $n$ distinct lines for which
  $\partial \ell_j = p_j-p_{j-1}, j=1, \dots, n-1, \partial \ell_n =
  p_0-p_{n-1},$ for some set
  $\{p_0, p_1, \dots, p_{n-1} \} \subset \mathcal{P}_0$ of distinct
  points. Then the cycle $\ell_1+\ell_2+\cdots+\ell_n$ is called an
  {\em $n-$polygon}, $n\geq 2.$
\end{definition}

The following structure theorem for $2-$cells in tight scripts
follows:

\begin{theorem}\label{Th:2.001}
  Let $v \in \mathcal{C}_2$ be  tight $2-$cell; then there
  exists $n\ge 2$ such that $\partial v = \ell_1+\ell_2+\cdots+\ell_n$
  is an $n-$polygon.
\end{theorem}

\begin{proof}
  Pick $\ell_1 \in \rb(v),$ with $\partial \ell_1 = p_1-p_0.$ Then
  there exist a point $p_2\neq p_1$ and a line $\ell_2 \not= \ell_1$
  for which $\partial \ell_2 = p_2-p_1$ (otherwise, we would have
  $p_1 \in \mbox{supp }(\partial\partial v) ).$

  If $p_2=p_0$ we have a $2-$gon inside $\rb (v).$ If $p_2\neq p_0$
  then $p_2 \notin \{ p_0, p_1 \}$ and there exist $p_3\neq p_2$ and
  $\ell_3$ for which $\partial \ell_3 = p_3-p_2$, $\ell_3\neq
  \ell_2$. We also have that $\ell_3\neq\ell_1$ since
  $p_2\notin\{p_0,p_1\}$.  Now, if $p_3 \in \{ p_0, p_1 \}$ the
  tightness condition requires that $p_3=p_0$ (otherwise $\ell_2$ and
  $\ell_3$ will form a $2-$gon). In this case
  $\{ \ell_1, \ell_2, \ell_3 \}$ is a $3-$gon.

  If $p_3 \notin \{ p_0, p_1, p_2 \}$ we repeat the process and there
  exist $p_4\neq p_3$ and $\ell_4 \notin \{ \ell_1, \ell_2, \ell_3 \}$
  with $\partial \ell_4 = p_4-p_3,$ and the proof follows inductively.

  After finitely many steps we create a polygon
  $\ell'_1+\ell'_2+\cdots+\ell'_n$ inside $\rb (v).$ Due to tightness,
  $\ell'_1+\ell'_2+\cdots+\ell'_n$ generates $\mathcal{Z}_1(\rb (v))$
  or $=\pm \partial v.$
\end{proof}  

\begin{corollary}\label{Cor:2.001}
  A tight $2-$dimensional script is always unitary.
\end{corollary}
\begin{proof}
  Following the previous theorem, any two cell will have an
  $n-$polygon as boundary. Therefore the script is unitary.
\end{proof}

Tight scripts provide a solution to the skeleton problem (\ref{Eq:1.025}):

\begin{theorem}\label{Th:2.002} Let
  \begin{equation}\label{Eq:1.0027}
    0 \overset{\partial}{\longleftarrow} \mathbb{Z} \overset{\partial}{\longleftarrow} \mathcal{M}_{0} \overset{\partial}{\longleftarrow} \cdots
  \end{equation}
  be a tight script with skeleton:
  \begin{equation}\label{Eq:1.0028}
    \cdots  \overset{\rb}{\longleftarrow} \mathcal{P}(\mathcal{C}_{k-1})
    \overset{\rb}{\longleftarrow} \mathcal{P}(\mathcal{C}_{k}) \overset{\rb}{\longleftarrow} \cdots .
  \end{equation}
  Then any minimal script with the same skeleton is equivalent to the
  original script.
\end{theorem}

\begin{proof}
  This clearly holds for
  $0 \overset{\partial}{\longleftarrow} \mathbb{Z}
  \overset{\partial}{\longleftarrow} \mathcal{M}_{0}.$ Assume the
  property for
  $$
  0 \overset{\partial}{\longleftarrow} \mathbb{Z}
  \overset{\partial}{\longleftarrow}
  \mathcal{M}_{0}\overset{\partial}{\longleftarrow} \cdots
  \overset{\partial}{\longleftarrow} \mathcal{M}_{k-1}
  $$
  and let $C^k_j \in \mathcal{C}_k.$ Since the script is tight, we
  have that $\mathcal{Z}_{k-1}(\rb C^k_j)$ has one generator (up to
  sign). Let us call this generator $\sum_l \lambda_l C^{k-1}_l,$ then
  we can choose $\partial C^k_j = \lambda \sum_l \lambda_l C^{k-1}_l$
  and that fixes
  $\mathcal{M}_{k-1} \overset{\partial}{\longleftarrow}
  \mathcal{M}_{k}$ because $\lambda=\pm1$ when $\partial C^k_j$ is
  minimal.
\end{proof}

We expect the converse to be true as well, we leave the proof to the
interested student.

Tight scripts also solve the {\em assignment problem}. Suppose given
$$
0 \overset{\partial}{\longleftarrow} \mathbb{Z}
\overset{\partial}{\longleftarrow}
\mathcal{M}_{0}\overset{\partial}{\longleftarrow} \cdots
\overset{\partial}{\longleftarrow} \mathcal{M}_{k-1}
$$
and for $C^k_j$ we also know
$\rb C^k_j = \{ C^{k-1}_l, \, \text{some} \,\, l's \}.$ How to
actually find the coefficients $\lambda_l$ for which $\partial C^k_j$
eventually equals $\sum_l \lambda_l C^{k-1}_l$?

First of all, one $\lambda_l$ may be freely chosen. Next, one has the
equation $0 = \partial^2 C^k_j = \sum_l \lambda_l \partial C^{k-1}_l$
which for a tight cell has a unique solution up to a constant. As we
need a cell to be minimal, the constant is $\pm 1.$ Therefore the
solution is unique up to sign hence the script obtained is unique up
to equivalence.

\begin{definition}\label{Def:2.002}
  Let $\ell_1, \cdots, \ell_n$ be $n$ distinct lines for which
  $\partial \ell_1 = p_1-p, \partial \ell_j = p_j-p_{j-1}, j=2,
  \cdots, n-1, \partial \ell_n = q-p_{n-1},$ whereby
  $\{p, p_0, p_1, \cdots, p_{n-1}, q \} \subset \mathcal{P}_0$ are
  distinct points. Then the chain $\ell_1+\ell_2+\dots+\ell_n$ is
  called a {\em simple curve} from $p$ to $q$ with length $n.$
\end{definition}

\begin{theorem}\label{Th:2.003}
  Let $\ell$ be a tight one-chain of length $n$ inside a tight script;
  then $\ell$ is either a cycle or a simple curve of same length $n$
  between two points.
\end{theorem}

\begin{proof}
  The chain $\ell$ together with $\mbox{supp } \ell$ and
  $\rb(\mbox{supp } \ell)$ defines a $1-$dimensional graph with lines
  in $\mbox{supp } \ell$ and points in $\rb(\mbox{supp } \ell).$ Every
  line connects 2 points. Since the script is tight, this graph must
  be connected or else $\mathcal{Z}_1(\mbox{supp } \ell, \rb (\ell))$
  wold have more than one generator.

  If $\ell$ is not a cycle, then the graph contains no loops, so it is
  actually a tree.

  Finally, should $\rb (\ell)$ have 3 points $p, q, r$ or more, then $p$
  can be connected to $q$ by a (simple) curve $\ell_1,$ and $q$ to $r$
  by another curve $\ell_2.$ Then the curves
  $\ell_1, \ell_2 \in \mathcal{Z}_1(\mbox{supp } \ell, \rb (\ell))$ are
  different, contradicting the tightness of $\ell$. Hence $\rb (\ell)$
  can have at most two points $\rb (\ell) = \{p, q \} (p\not=q)$ and
  since $\rb (\ell)$ is connected there exists a simple curve from $p$
  to $q$ inside the tree.  This would be the single generator of
  $\mathcal{Z}_1(\mbox{supp } \ell, \rb (\ell)),$ therefore the tree is
  this simple curve.
\end{proof}

\section{CW-complexes as Scripts}

A CW-complex is a Hausdorff space $X$ together with a partition of $X$
into open cells (of varying dimension) that satisfies two properties:
\begin{enumerate}[(i)]
\item for each $n-$ dimensional open cell $C$ there is a continuous
  map $f$ from the closed ball $\mathbb{B}\subset {\mathbb{R}^n}$ to
  $X$ such that
  \begin{enumerate}
  \item[(i.1)] the restriction of $f$ to $\overset{\circ}{\mathbb{B}}$
    is a homeomorphism onto cell $C;$
  \item[(i.2)] the image of the sphere $\partial \mathbb{B}$ is equal
    to the union of finitely many cells of dimension less than $n.$
  \end{enumerate}
\item A CW-complex is {\em regular} if the map $f$ is a homeomorphism
  on the closed ball.
\end{enumerate}

Next let $\mathcal{C}_k$ be the set of all $k-$dimensional cells in a
regular CW-complex; then for each $C^k_j \in \mathcal{C}_k$ we put
$\rb (C^k_j) = \{ C^{k-1}_l : C^{k-1}_l \subset f(\partial \mathbb{B})
\};$ we must have that
$$f(\partial \mathbb{B}) \subset \rb (C^k_j) \cup \rb^2 (C^k_j) \cup
\cdots \cup \rb^k (C^k_j).$$

In this way we obtain a skeleton in which every cell is basically a
$k-$dimensional polyhedron. Now, a polyhedron is always the skeleton
of a tight unitary script therefore the skeleton of a CW-complex is
the skeleton of a tight unitary script. This means that this tight
script is the only minimal script attached to this skeleton. So we
have proven:

\begin{theorem}\label{Th:2.004}
  To a given regular CW-complex corresponds a unique tight unitary
  script.
\end{theorem}

The converse is not true as there exist tight unitary scripts that do
not correspond to a CW-complex. Clearly, every 2D-tight script is
CW-complex.

\section{A $2-$torus}

We have
 \begin{gather*}
   \mathcal{C}_{0}= \{ p_0, p_1, p_2, p_3 \}, \quad \partial p_j =1, j=0, 1, 2, 3.  \\
   \mathcal{C}_{1} = \{ \ell_1, \ell_2, \ell_3, \ell_4, \ell_5, \ell_6, \ell_7, \ell_8  \},   \\
   \partial \ell_1 = p_1-p_0, \quad \partial \ell_2 = p_0 -p_1, \quad \partial \ell_3 = p_2 -p_0, \quad \partial \ell_4 = p_0 -p_2   \\
   \partial \ell_5 = p_3-p_2, \quad \partial \ell_6 = p_2 -p_3, \quad \partial \ell_7 = p_3 -p_1, \quad \partial \ell_8 = p_1 -p_3   \\
   \mathcal{C}_{2} = \{ v_1, v_2, v_3, v_4  \},    \\
   \partial v_1 = \ell_5+\ell_8-\ell_1-\ell_4, \quad \partial v_2 = \ell_6+\ell_4-\ell_2-\ell_8,  \\
   \partial v_3 = \ell_1+\ell_7-\ell_5-\ell_3, \quad \partial v_4 = \ell_2+\ell_3-\ell_6-\ell_7, \\
   \mathcal{C}_{3}= \{ C \}, \quad \partial C =v_1+v_2+v_3+v_4.
\end{gather*}

Hence $\partial C =v_1+v_2+v_3+v_4$ is a unitary and tight script but is not the image of a sphere so its no CW-complex.


\begin{figure}[hb]
  \begin{center}
    \includegraphics[scale=0.4]{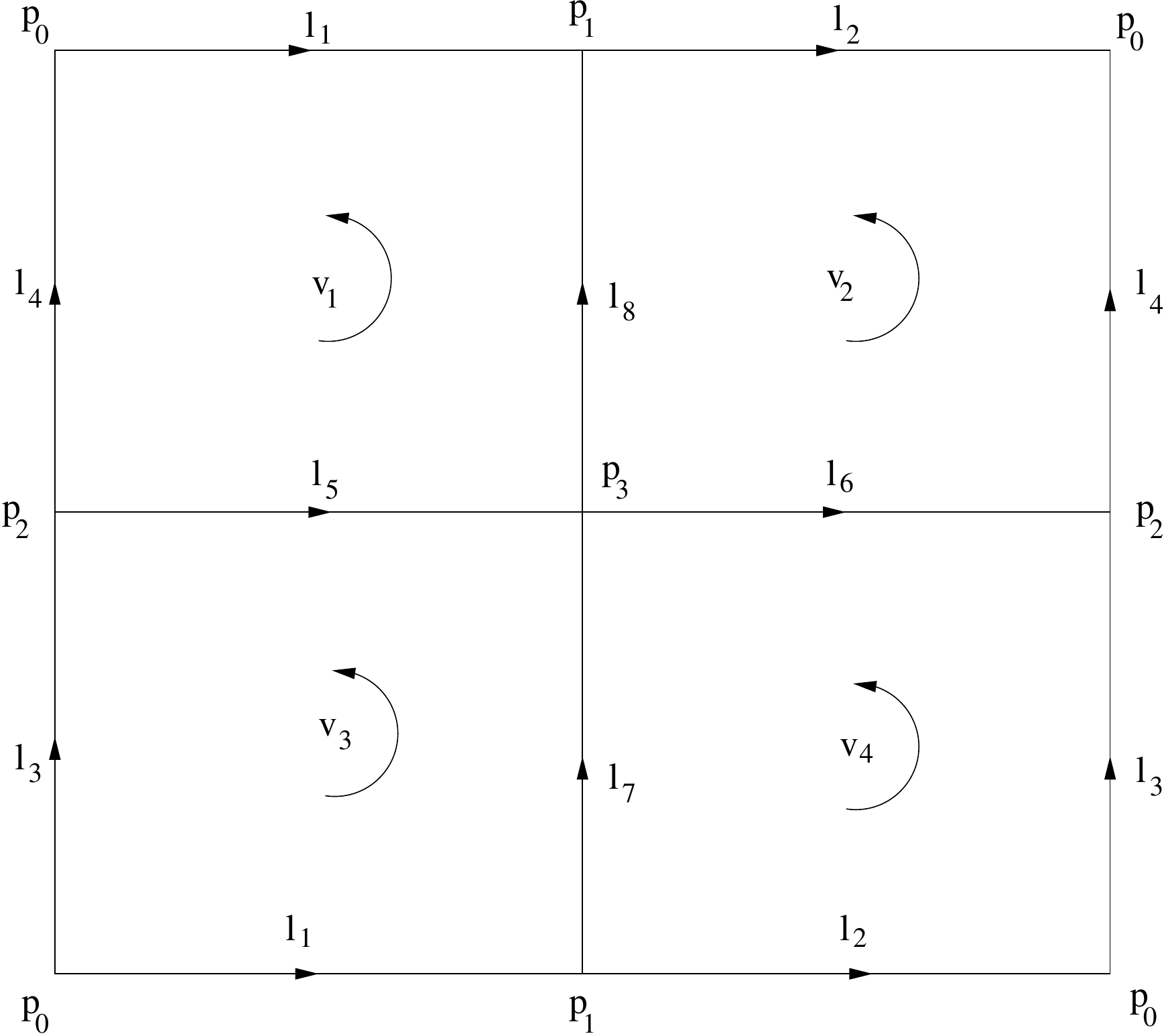}
    \caption{Script for the $2-$torus}
  \end{center} \label{fig:torus}
\end{figure}

\section{A Klein bottle}

We take
 \begin{gather*}
   \mathcal{C}_{0}= \{ p_0, p_1, p_2, p_3 \}, \quad \partial p_j =1, j=0, 1, 2, 3.  \\
   \mathcal{C}_{1} = \{ \ell_1, \ell_2, \ell_3, \ell_4, \ell_5, \ell_6, \ell_7, \ell_8  \},   \\
   \partial \ell_1 = p_1-p_0, \quad \partial \ell_2 = p_0 -p_1, \quad \partial \ell_3 = p_2 -p_0, \quad \partial \ell_4 = p_0 -p_2   \\
   \partial \ell_5 = p_3-p_1, \quad \partial \ell_6 = p_1 -p_3, \quad \partial \ell_7 = p_3 -p_2, \quad \partial \ell_8 = p_2 -p_3   \\
   \mathcal{C}_{2} = \{ v_1, v_2, v_3, v_4  \},    \\
   \partial v_1 = \ell_5+\ell_8-\ell_2-\ell_3, \quad \partial v_2 = \ell_6-\ell_1-\ell_4-\ell_8,  \\
   \partial v_3 = -\ell_1+\ell_3+\ell_7-\ell_5, \quad \partial v_4 =
   \ell_4-\ell_2-\ell_6-\ell_7.
\end{gather*}

\begin{figure}[H]
  \begin{center}\includegraphics[scale=0.4]{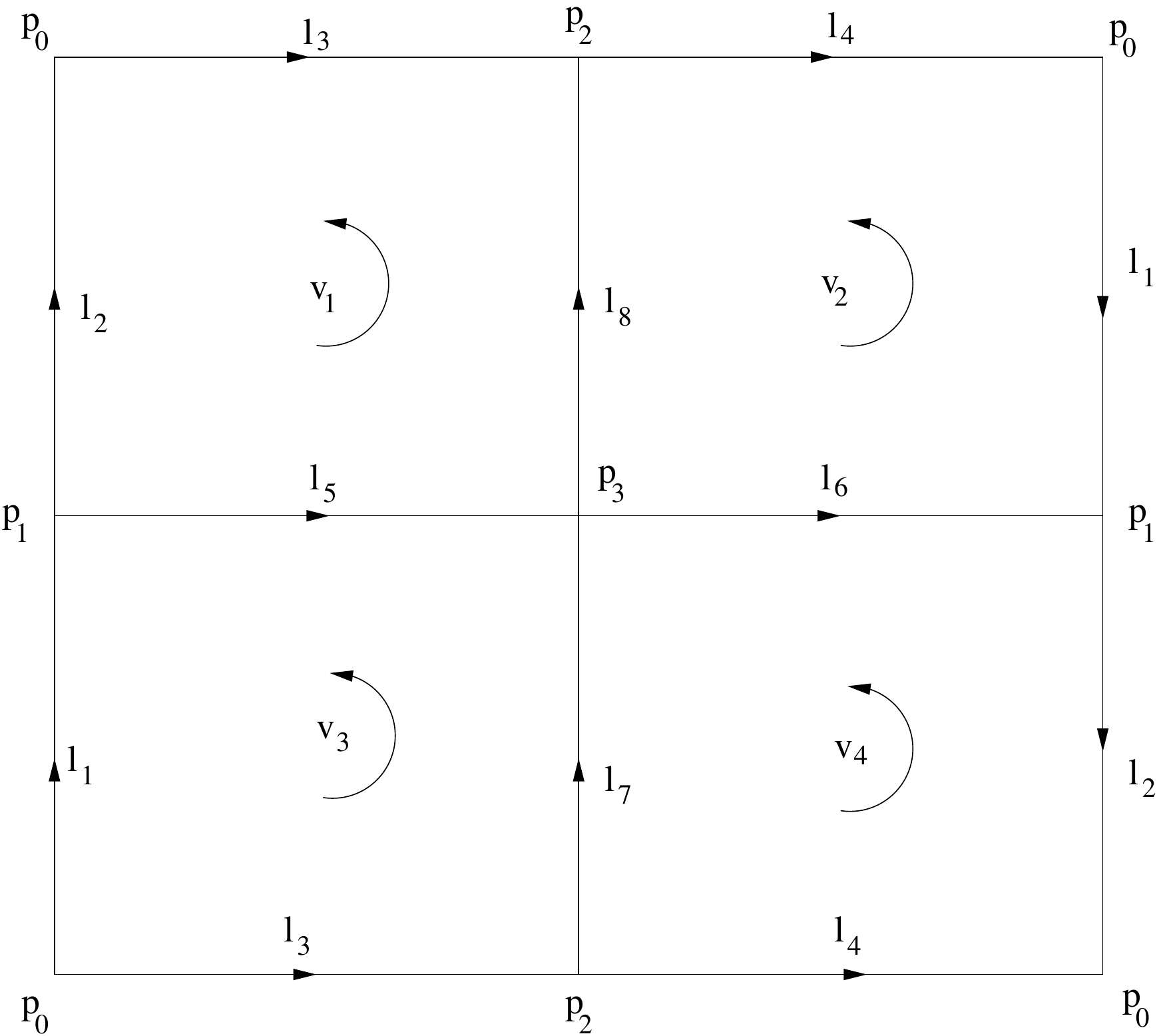}
    \caption{The Klein Bottle}
  \end{center} \label{fig:klein}
\end{figure}

Notice that
$$\partial v_1  + \partial v_2 + \partial v_3 + \partial v_4 = -2\ell_1 -2\ell_2.$$
So if we add a new cell $v_5, ~\partial v_5 = \ell_1 + \ell_2,$ to the
script we obtain a cycle for the "extended Klein bottle":
$$\partial v_1  + \partial v_2 + \partial v_3 + \partial v_4 + 2 \partial v_5 = 0.$$
Hence, $v_1  + v_2 +  v_3 + v_4 + 2  v_5 $ is a tight cycle but it is not unitary.

\section{A projective plane}

We take
\begin{gather*}
  \mathcal{C}_{0}= \{ p_1, p_2, p_3 \}, \quad \partial p_j =1, j=1, 2, 3.  \\
  \mathcal{C}_{1} = \{ \ell_1, \ell_2, \ell_3, \ell_4, \ell_5, \ell_6   \},   \\
  \partial \ell_1 = p_2-p_1, \quad \partial \ell_2 = p_1 -p_2, \quad \partial \ell_3 = p_3 -p_1, \quad \partial \ell_4 = p_1 -p_3   \\
  \partial \ell_5 = p_3-p_2, \quad \partial \ell_6 = p_2 -p_3,   \\
  \mathcal{C}_{2} = \{ v_1, v_2, v_3, v_4  \},    \\
  \partial v_1 = -\ell_2+\ell_5-\ell_3, \quad \partial v_2 = -\ell_1-\ell_4-\ell_5,  \\
  \partial v_3 = -\ell_1+\ell_3+\ell_6, \quad \partial v_4 =
  \ell_4-\ell_2-\ell_6.
\end{gather*}

\begin{figure}[H]
  \begin{center}\includegraphics[scale=0.4]{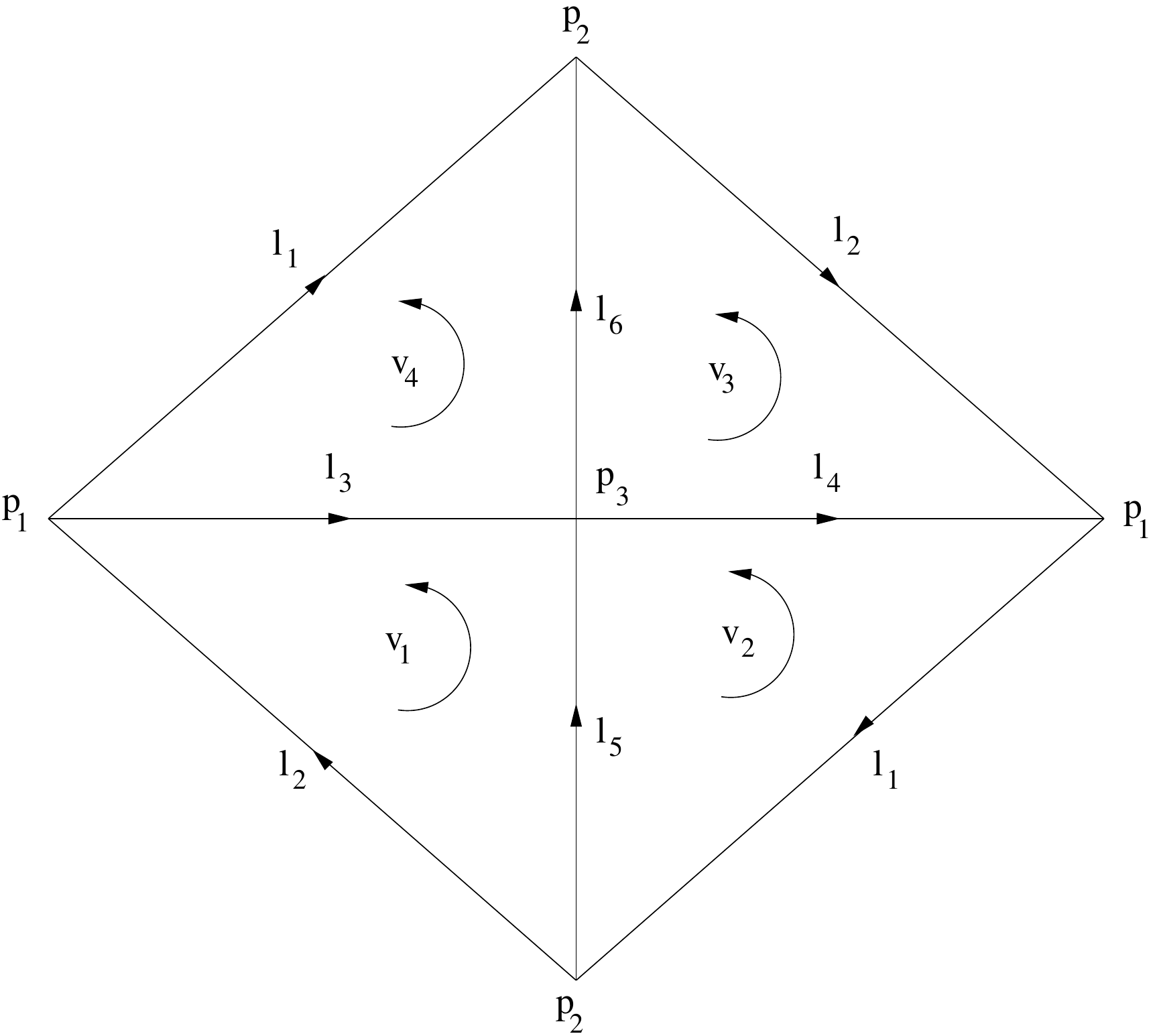}
    \caption{The Projective Plane}
  \end{center} \label{fig:proj}
\end{figure}

Also here
$\partial v_1 + \partial v_2 + \partial v_3 + \partial v_4 = -2\ell_1
-2\ell_2$ and we can add a fifth cell
$v_5, ~\partial v_5 = \ell_1 + \ell_2.$ Then we have a tight $2-$cycle which is not unitary: $v_1  + v_2 +  v_3 + v_4 + 2  v_5 $ with
$$\partial v_1  + \partial v_2 + \partial v_3 + \partial v_4 + 2 \partial v_5 = 0.$$

\textbf{Remarks:}\begin{enumerate}[i)]
\item Scripts can in fact always be made unitary by attaching extra
  cells. In the above example we can make an extra cell
  $v_6, ~\partial v_6 = \ell_1 + \ell_2.$ Then we get a unitary cycle
  $\partial v_1 + \partial v_2 + \partial v_3 + \partial v_4 +
  \partial v_5 + \partial v_6 = 0$ but that cycle is no longer
  tight. Also, using a Gauss method one may extend cells to make
  things tight, but that would generally not be unitary except on the
  level of the 2-cycles (boundaries of 3-cells).

  Hence, tight and unitary scripts are very special and deserve a new
  name. We define a tight and unitary script to be a {\em geometrical
    script} or \textit{geoscript}.

\item This section is not nearly complete and there are many
  interesting problems such as

  \textbf{Problem:} let $C$ be a tight chain; can one prove that
  $\partial C$ is also tight? In particular, if $C$ is a tight
  $2-$chain, can we prove that $\partial C$ is a polygon?

  The answer to both questions is an obvious no, as the union
  of two disjoint circles realized as the boundary of a script forming
  a cylinder seen as a single $2-$dimensional tight chain is a
  counterexample to both.

  This kind of problems will turn out quite important later on, in
  particular when we study extensions. But first we need more
  constructive methods.
\end{enumerate}

\textbf{Remark:} A cell $C^k_j$ is tight iff $\rb (C^k_j)$ has trivial
homology. Tightness hence means that the local homology at the level
of cells is trivial.  For manifolds this would correspond to a version
of Poincar\'e's lemma which means that each point has a neighborhood
with trivial homology.



\chapter{Basic Operations on Scripts}

An essential part of working with scripts is the possibility to modify them. Here we discuss three different types of operations on scripts, cleaning or removing operations, operations of creation, and operations of identification.

\section{Cleaning Operations}
There are two types of cleaning operators:
\begin{enumerate}
\item[1)] Floating Cells
\item[2)] Free arcs (or domes)
\end{enumerate}
Here are the definitions of the two:
\begin{definition}
  A {\em floating cell} is a cell for which $\partial C_j^k=0$. If it
  appears in $\partial C_l^{k+1}$ then $ C_l^{k+1}$ is usually not
  tight and one can in fact remove $C_j^k$ from the script by simple
  cancellation and replace it by $0$ if it appears in $C_l^{k+1}$.
\end{definition}

\begin{definition}
  A {\em free arc} is a cell $C_j^k$ that does not appear in any
  boundary $\partial C_l^{k+1}$. It may thus be removed from the
  script by simple cancellation.
\end{definition}

We will see from Chapter 5 that these operations are each other's dual. However,
removing floating cells may be essential on the way to a tight script
while removing free arcs is not and it could make the situation worse,
therefore we will not go through this process automatically.

Removing a free arc does not affect tightness, as tightness works from
a higher dimensional cell to a lower dimensional cell, therefore, as a
free arc does not appear in any boundary, removing it will not affect
tightness.

\section{Operations of Creation}
\begin{enumerate}
\item[1)] Creating New Cells
\item[2)] Pulling cells together
\end{enumerate}

One can {\em create new cells} through the following method. Let
$\mathcal{C}^{k-1}$ be a $k-1$ cycle, then we may add a new cell
$C_j^k$ to the script for which $\partial C_j^k=\mathcal{C}^{k-1}$. It
is as if one creates an arc or a dome above a cycle.  Creating new
cells is the inverse of removing free arcs.

The action of {\em pulling cells together} is the dual of the above
and it corresponds to introducing a new cell $C_j^{k-1}$ with  $\partial C_j^{k-1}=0$ that appears
in a number of $\partial C_l^k$ whereby $C_l^k$ form a cocycle (see Chapter 5).  As
$\partial C_j^{k-1}=0$ this operation worsens tightness and it is
usually not done.

\section{Operations of Identification}
\begin{enumerate}
\item[1)] Glueing cells together
\item[2)] Melting cells together
\item[3)] Cutting cells open
\item[4)] Expanding cells
 \end{enumerate}

 For {\em glueing cells together} let $C_1^k$ and $C_2^k$ be cells for which
 $\partial C_1^k= \partial C_2^k$ then we may add the constraint
 $ C_1^k= C_2^k$ to the script. This equation can be solved by
 replacing $C_2^k$ by $C_1^k$ wherever it occurs in
 $\partial C_l^{k+1}$ for some $C_l^{k+1}$ and to remove $C_2^k$ from
 the script.

 This operation may turn $C_l^{k+1}$ into a floating cell, so a
 cleaning may follow a glueing operation.

\begin{example}
  Notice that points always satisfy $\partial p_1=1=\partial p_2$, so
  glueing together points is "free". Once this is done one can glue
  together lines, then planes, etc., as one likes.  Also, if one has
  the relations $\partial \lambda C_1^k=\partial \mu C_2^k$, one may
  glue $\lambda C_1^k=\mu C_2^k$.
\end{example}

The process of {\em melting cells together} is the dual of
glueing. For example, let $C_1^k$ and $C_2^k$ be $k-$cells such that,
if they do appear in the script, they only appear inside
$\partial C_l^{k+1}$ as a sum $C_1^k+C_2^k$ or as a linear combination
$\lambda C_1^k+\mu C_2^k$.  Then we create a new cell $C_0^k$ and add
the equations $C_0^k=C_1^k+C_2^k$ (or
$C_0^k=\lambda C_1^k+\mu C_2^k$).

This equation can be solved by replacing $C_1^k+C_2^k$ by $C_0^k$ in
every $\partial C_l^{k+1}$ whenever it occurs. After this operation
$C_1^k$ and $C_2^k$ become free cells and they can be removed.

 As a warning, even when completing this operation there may appear free
cells in $\partial C_1^k$ or $\partial C_2^k$ that one may consider to
remove them as well.

Notice that for the highest dimension $k=m$, the operation of melting
cells is free (and optional); after completing this operation one may
melt lower dimensional cells.

The third operation of identification is {\em cutting cells open}
which is the inverse of glueing cells together and it involves the
creation of a new cell $C_2^k$ with $\partial C_2^k=\partial C_1^k$,
together with a possible replacement of $C_1^k$ by $C_2^k$ inside
$\partial C_l^{k+1}$.  Here one starts with the higher dimensions and
then works down through the dimensions of the cells.

The last operation of identification, the expanding of cells is
defined as the opposite of melting.

In conclusion, hereby one replaces a cell $C_0^k$ by a chain
$$
C_1^k+\dots + C_l^k
$$
such that $\partial C_0^k=\partial C_1^k+\dots + \partial
C_l^k$. Whenever $C_0^k$ appears in $\partial C_1^{k+1}$, one replaces
it as well.

To enable an expansion one may have to create extra lower dimensional
objects that may be needed to build the boundaries
$\partial C_1^k,\dots,\partial C_l^k$, starting with extra points to
create zero lines, etc...  An expansion is called free if
$\cup \rb(C_j^k)\setminus \rb(C_0^k)$,
$\cup \rb^2(C_j^k)\setminus \rb^2(C_0^k)$, etc. consist entirely of
new cells.

Further details we leave as an exercise to the reader.

\section{Examples of basic operations on scripts}

\begin{example}[Moebius Strip]
  We start from the rectangle in Figure~\ref{fig:moebius} and we have:
  \begin{figure}[ht]
    \begin{center}\includegraphics[scale=0.5]{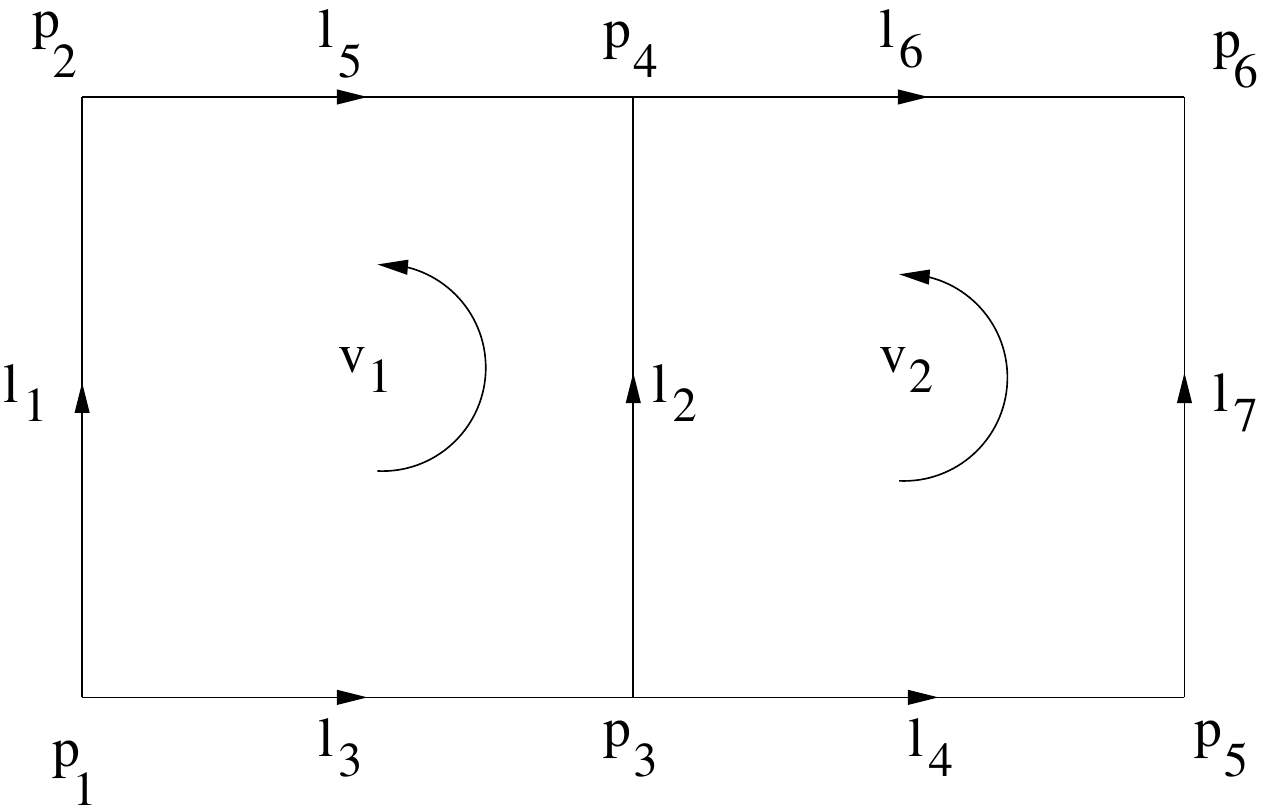}
      \caption{Starting rectangle for the Moebius Strip}
    \end{center} \label{fig:moebius}
  \end{figure}
  $$
  \mathcal{C}_0=\{p_1,p_2, p_3, p_4, p_5, p_6\},
  $$
  with $\partial p_j=1$,
  $$
  \mathcal{C}_1=\{l_1,l_2, l_3, l_4, l_5, l_6, l_7\},
  $$
  with $\partial l_1=p_2-p_1$, $\partial l_2=p_4-p_3$,
  $\partial l_3=p_3-p_1$, $\partial l_4=p_5-p_3$,
  $\partial l_5=p_4-p_2$, $\partial l_6=p_6-p_4$, and
  $\partial l_7=p_6-p_5$,
  $$
  \mathcal{C}_2=\{v_1,v_2\},
  $$
  with $\partial v_1=l_3+l_2-l_5-l_1$ and
  $\partial v_2=l_4+l_7-l_6-l_2$.

  Next we glue points $p_5=p_2$ and $p_6=p_4$ thus removing $p_6$ and
  $p_5$. This effectuates the following changes in the script:
  $\partial l_4=p_2-p_3$, $\partial l_6=p_1-p_4$,
  $\partial l_7=p_1-p_2=-\partial l_1$.

  Next we glue line $l_7=-l_1$, which works since
  $\partial l_7=-\partial l_1$ and this makes the following changes in
  the script:
  $$
  \partial v_2=l_4-l_1-l_6-l_2,
  $$
  leading to Moebius strip.
\end{example}

\begin{example}[Projective Plane]
  Reconsider the Moebius Strip with: $\partial l_1=p_2-p_1$,
  $\partial l_2=p_4-p_3$, $\partial l_3=p_3-p_1$,
  $\partial l_4=p_2-p_3$, $\partial l_5=p_4-p_2$,
  $\partial l_6=p_1-p_4$, and with: $\partial v_1=l_3+l_2-l_5-l_1$ and
  $\partial v_2=l_4-l_1-l_6-l_2$.
  \begin{figure}[H]
    \begin{center}\includegraphics[scale=0.4]{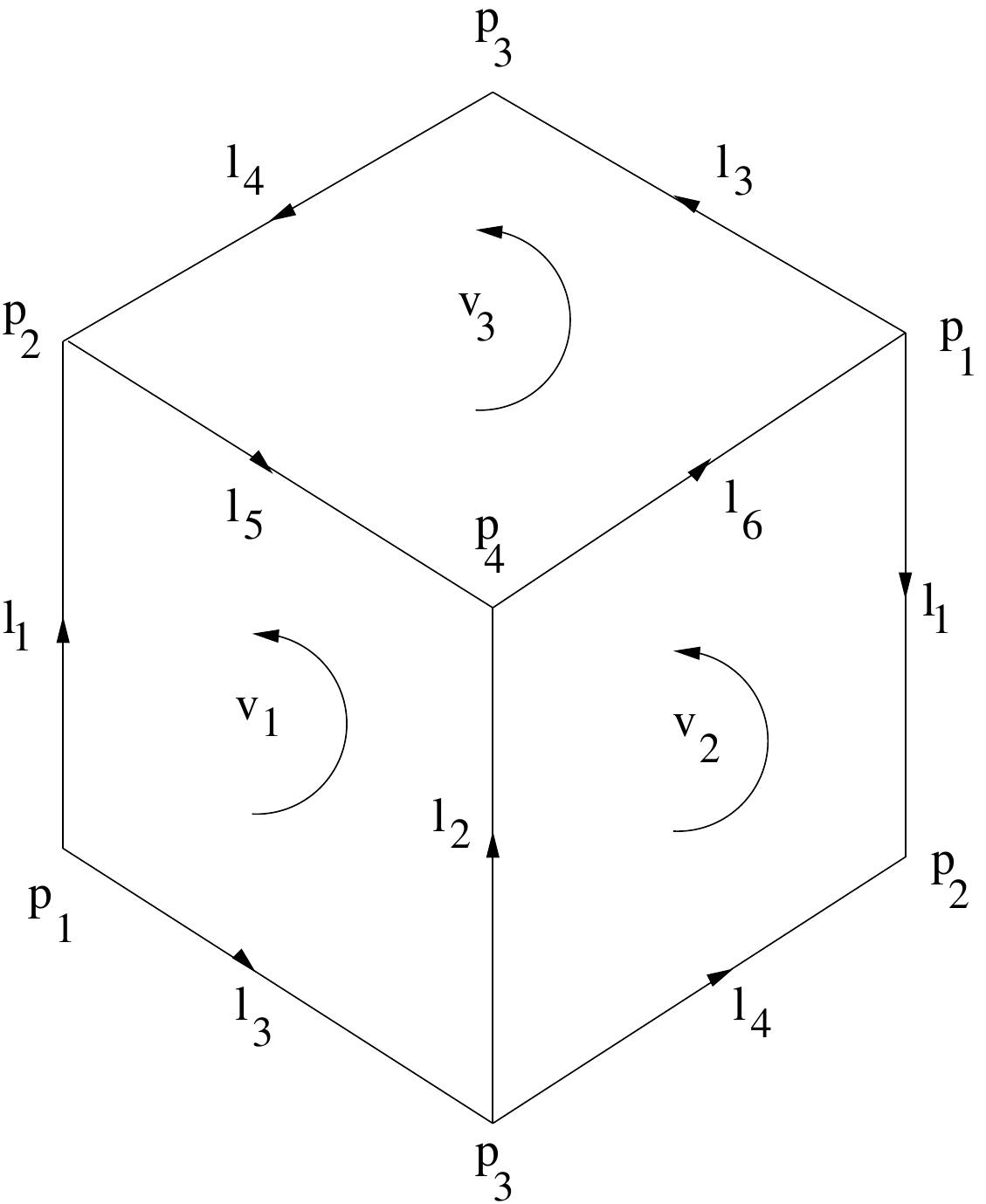}
      \caption{The Projective Plane}
    \end{center} \label{fig:projM}
  \end{figure}
  Attach a new cell $v_3$ to the boundary of Moebius Script
  $\rb v_3=\{l_3, l_4, l_5, l_6\}$. Tightness leads to:
  $$
  \partial v_3=l_5+l_6+l_3+l_4
  $$
  and
  $$
  \partial v_1+\partial v_2+ \partial v_3=2(l_3+l_4-l_1)
  $$
  so we obtain a projective plane.
\end{example}

\begin{example}[Another Klein bottle]
  We start with two Moebius strips (see Figure~\ref{fig:2moebius}).
  \begin{figure}[H]
    \begin{center}\includegraphics[scale=0.4]{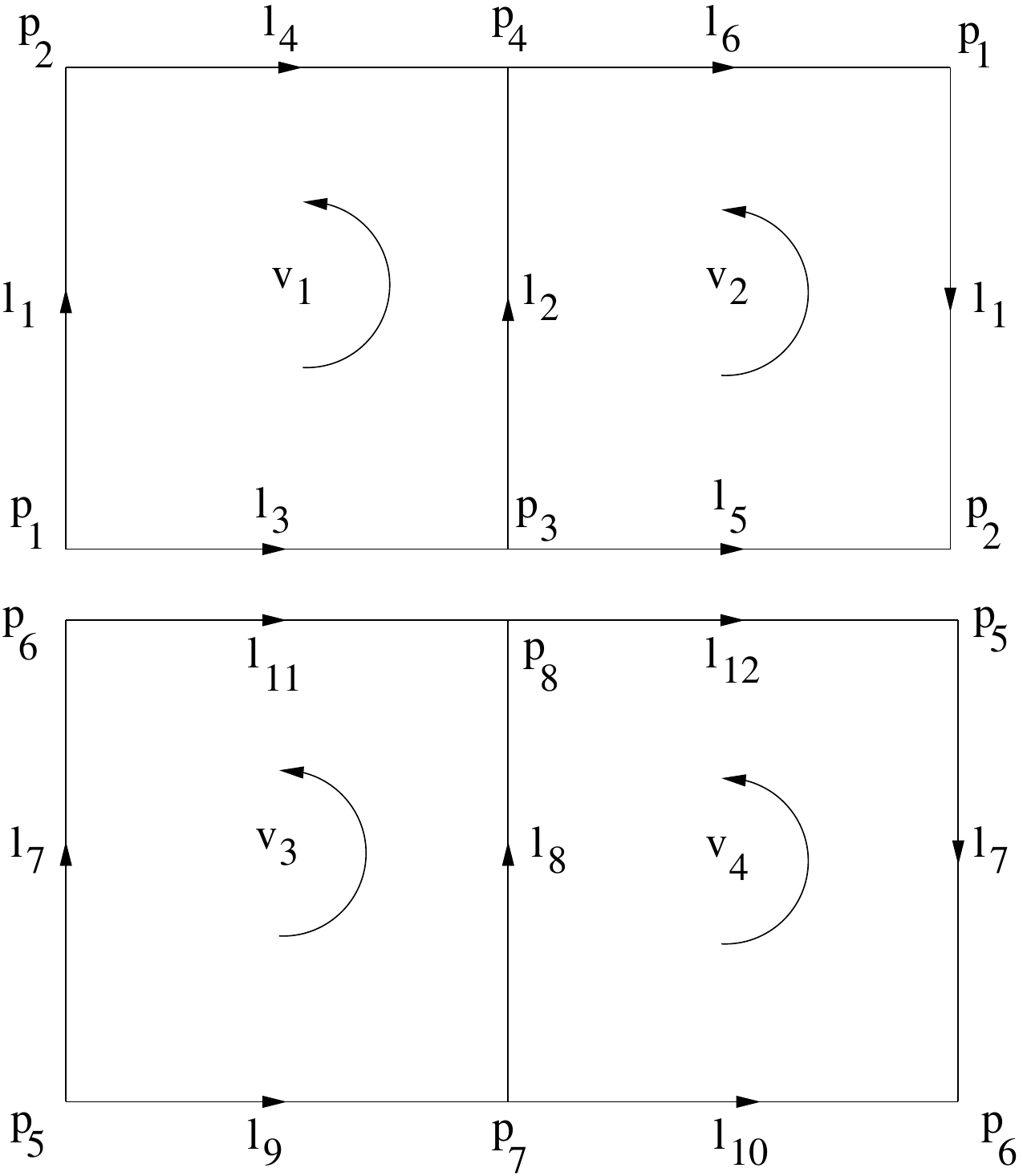}
      \caption{Two Moebius strips}
    \end{center} \label{fig:2moebius}
  \end{figure}
  Equations: exercise.

  Next we glue points $p_6=p_1$, $p_5=p_2$, $p_7=p_4$, and
  $p_8=p_3$. This gives rise to the following new relations:
  $$
  \partial l_{11}=p_3-p_1, \qquad \partial l_{12}=p_2-p_3,\qquad
  \partial l_9=p_4-p_2,\qquad \partial l_{10}=p_1-p_4.
  $$
  This allows us to glue further:
  $$
  l_{11}=l_3,\qquad l_{12}=l_5,\qquad l_9=l_4,\qquad l_{10}=l_6.
  $$
  This leads to the following script in Figure~\ref{fig:klein2}.
  \begin{figure}[H]\label{fig:klein2}
    \begin{center}\includegraphics[scale=0.4]{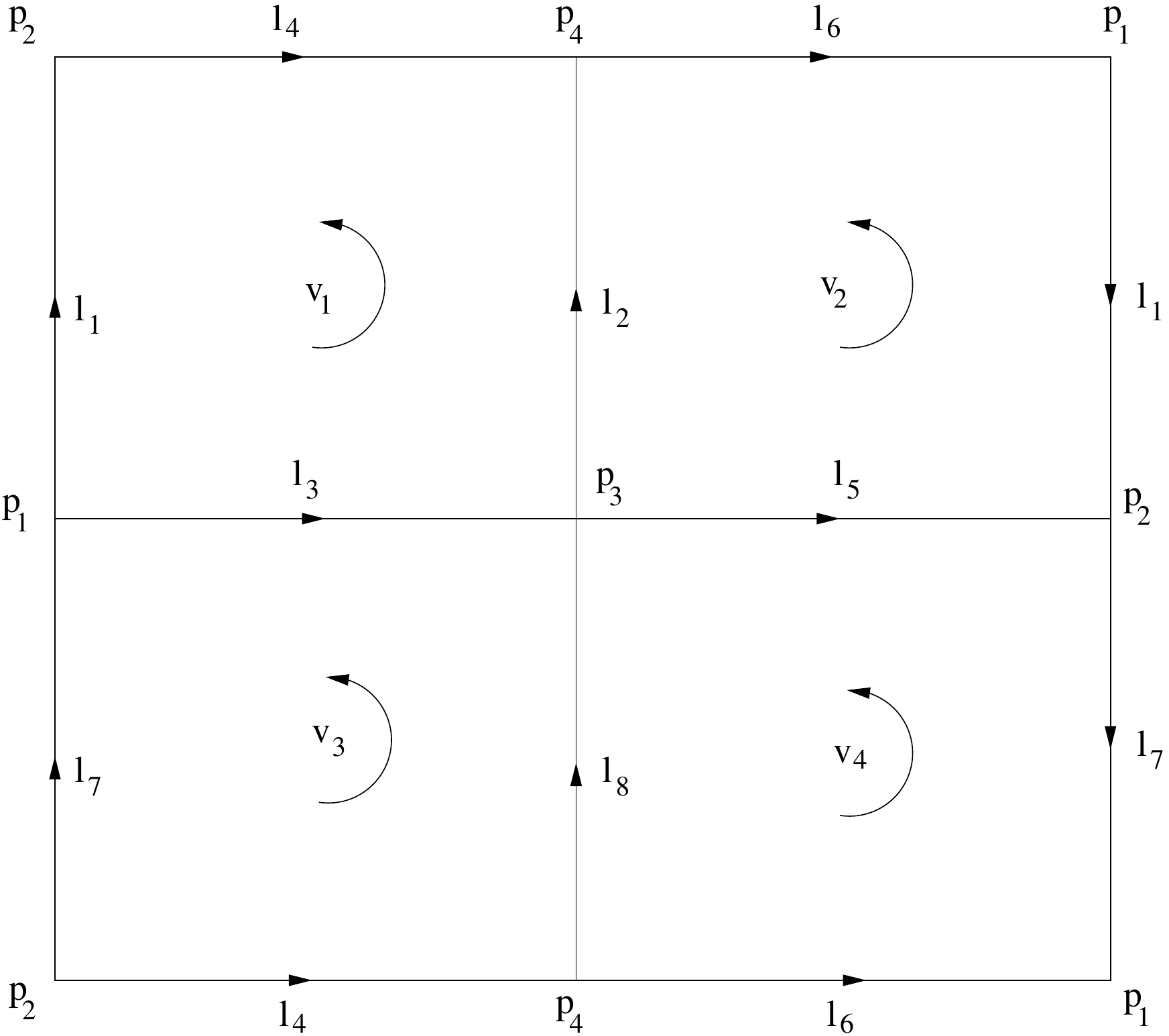}
      \caption{Two glued Moebius strips}
    \end{center}
  \end{figure}

  Next if we glue points $p_6=p_1$, $p_5=p_2$, $p_7=p_4$, and
  $p_8=p_3$. This also gives rise to the following new relations:
  $$
  \partial l_{1}=p_2-p_1, \qquad \partial l_{2}=p_4-p_3,\qquad
  \partial l_7=p_1-p_2,\qquad \partial l_{8}=p_3-p_4,
  $$
  $$
  \partial l_{3}=p_3-p_1, \qquad \partial l_{4}=p_4-p_2,\qquad
  \partial l_5=p_2-p_3,\qquad \partial l_{6}=p_1-p_4,
  $$

  $$
  \partial v_{1}=l_3+l_2-l_4-l_1, \qquad \partial
  v_{2}=l_5-l_1-l_6-l_2,\qquad \partial v_3=l_4+l_8-l_3-l_7,\qquad
  \partial v_{4}=l_6-l_7-l_5-l_8.
  $$
  This script is not equivalent to the Klein bottle script introduced
  earlier, but they do have a common refinement (expansion). To that
  end we first introduce new lines $l_4'$ and $l_3'$ with:
  $$
  \partial l_4'=p_3-p_2,\qquad \partial l_3'=p_4-p_1.
  $$
  \begin{figure}[ht]
    \label{fig:klein_refinement}
    \begin{center}\includegraphics[scale=0.4]{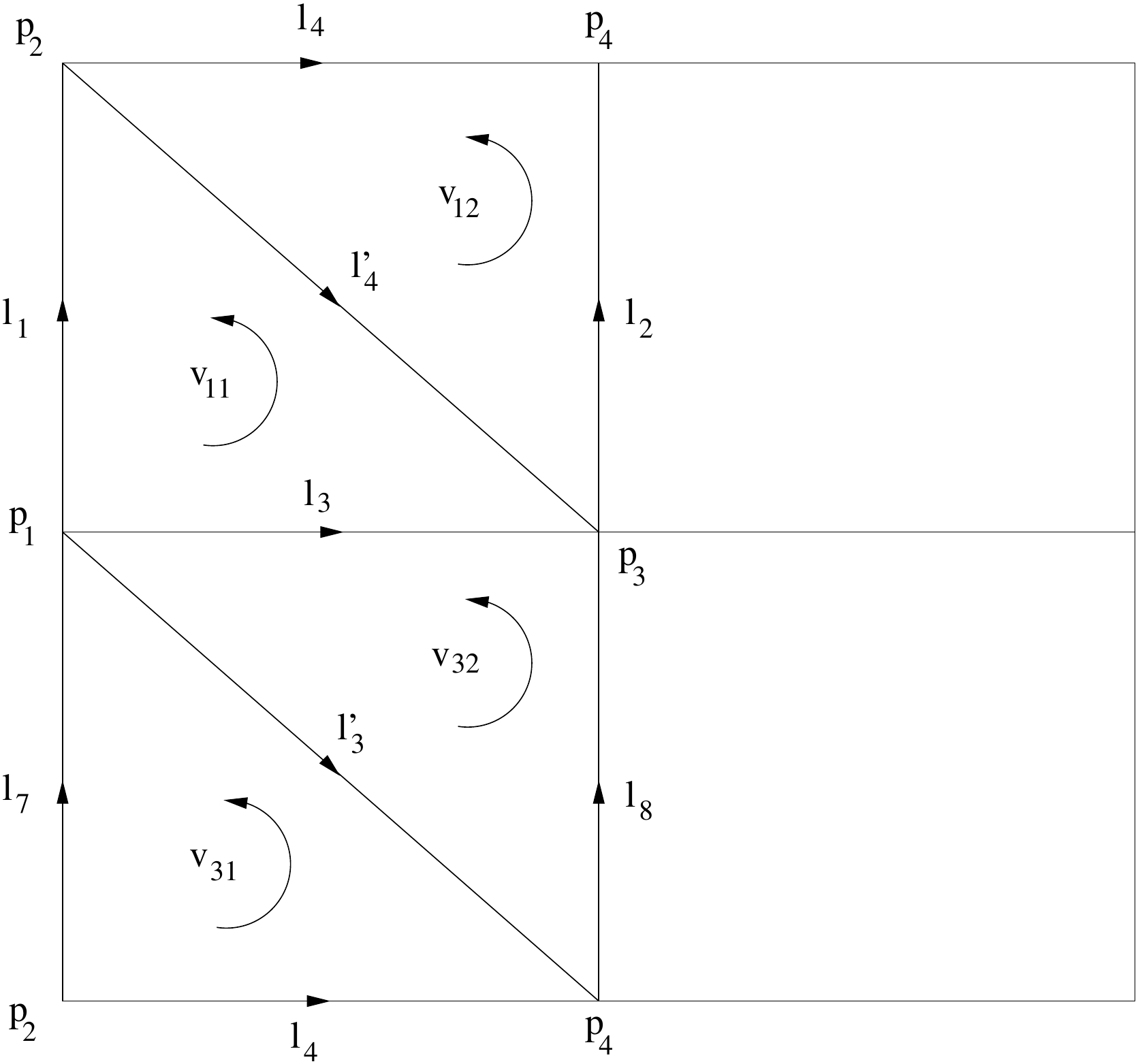}
      \caption{A refinement}
    \end{center}
  \end{figure}
  Then we consider the expansions:
  $$
  v_1=v_{11}+v_{12},\qquad v_3=v_{31}+v_{32},
  $$
  $$
  \partial v_{11}=l_3-l_4'-l_1,\qquad \partial v_{12}=l_4'+l_2-l_4,
  $$
  $$
  \partial v_{31}=l_4-l_3'-l_7,\qquad \partial v_{32}=l_3'+l_8-l_3.
  $$
  Now we melt $v_{11}$ to $v_{32}$ and $v_{31}$ to $v_{12}$, i.e.:
  $$
  v_1'=v_{11}+v_{32}, \qquad v_3'=v_{31}+v_{12}.
  $$
  We then get rid of $l_3$ and $l_4$:
  $$
  \partial v_1'=l_3'+l_8-l_4-l_1,\qquad \partial v_3'=l_4'+l_2-l_3'-l_7.
  $$
  This leads to the new script:
  \begin{figure}[H]
    \label{fig:klein3}
    \begin{center}\includegraphics[scale=0.4]{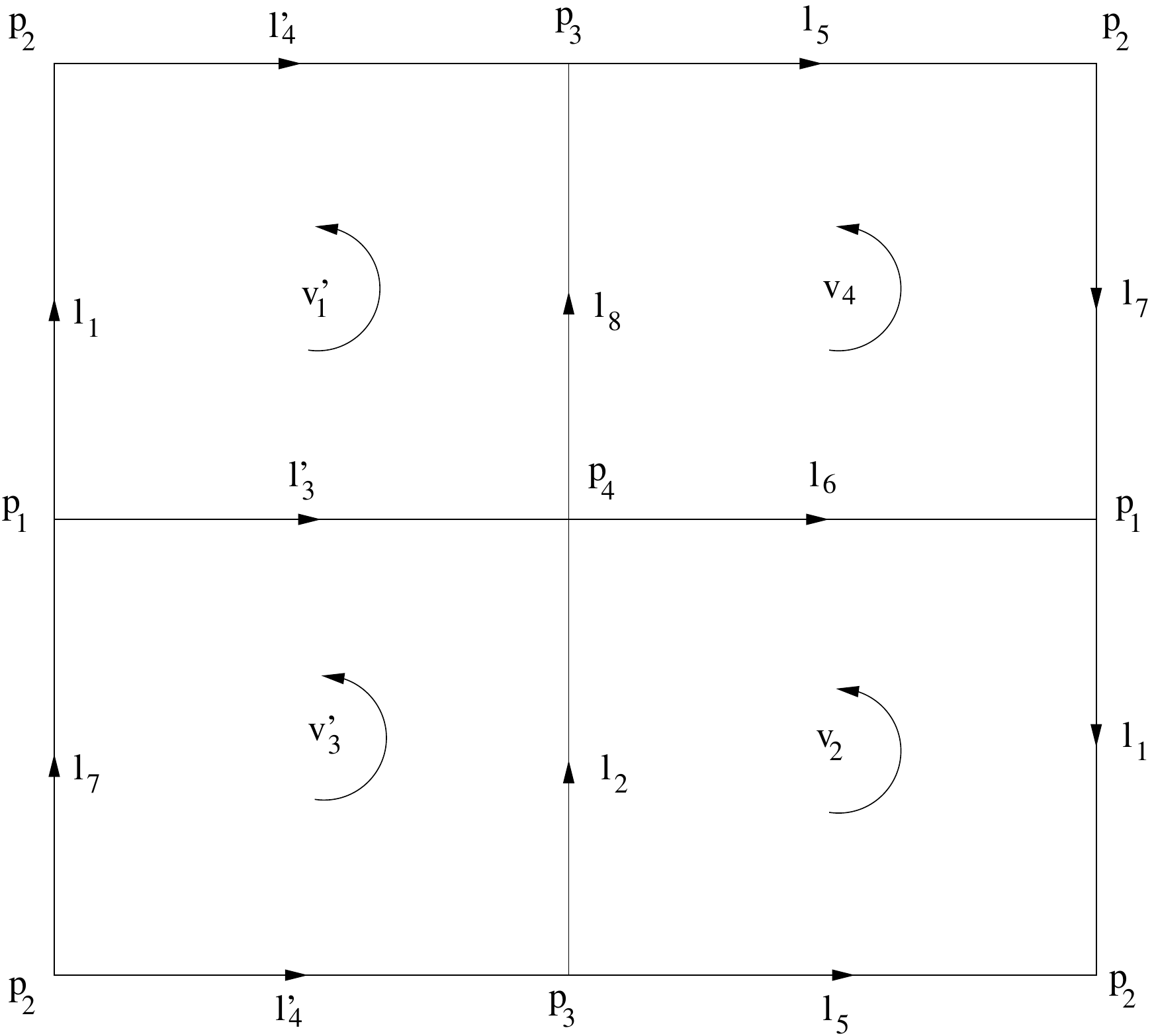}
      \caption{A new script for the Klein bottle}
    \end{center}
  \end{figure}
  This is equivalent (up to the names, of course) to the first Klein
  script introduced earlier. For CW-complexes topology is available
  and the two Klein bottles are topologically equivalent, but not as
  scripts.

  The Thomson addition of two projective planes is known to be a Klein
  bottle. It is obtained by deleting a disk from each projective plane
  and glueing the edges together. Now, removing a disk from a
  projective plane gives a Moebius strip.  So to carry out the Thomson
  addition of projective planes we, in fact, have to glue together two
  Moebius strips.
\end{example}

\section{3D World with 2D--disk portal}

The main idea here is to take two copies (top and bottom) of the
$\mathbb{R}^3\setminus \mathbb{D}$ where $\mathbb{D}$ is the
two--dimensional disk. These two copies are glued together by two
disks at the portal and two hemispheres at infinity.
  \begin{figure}[H]
    \label{fig:2Dportal}
    \begin{center}
      \includegraphics[scale=0.4]{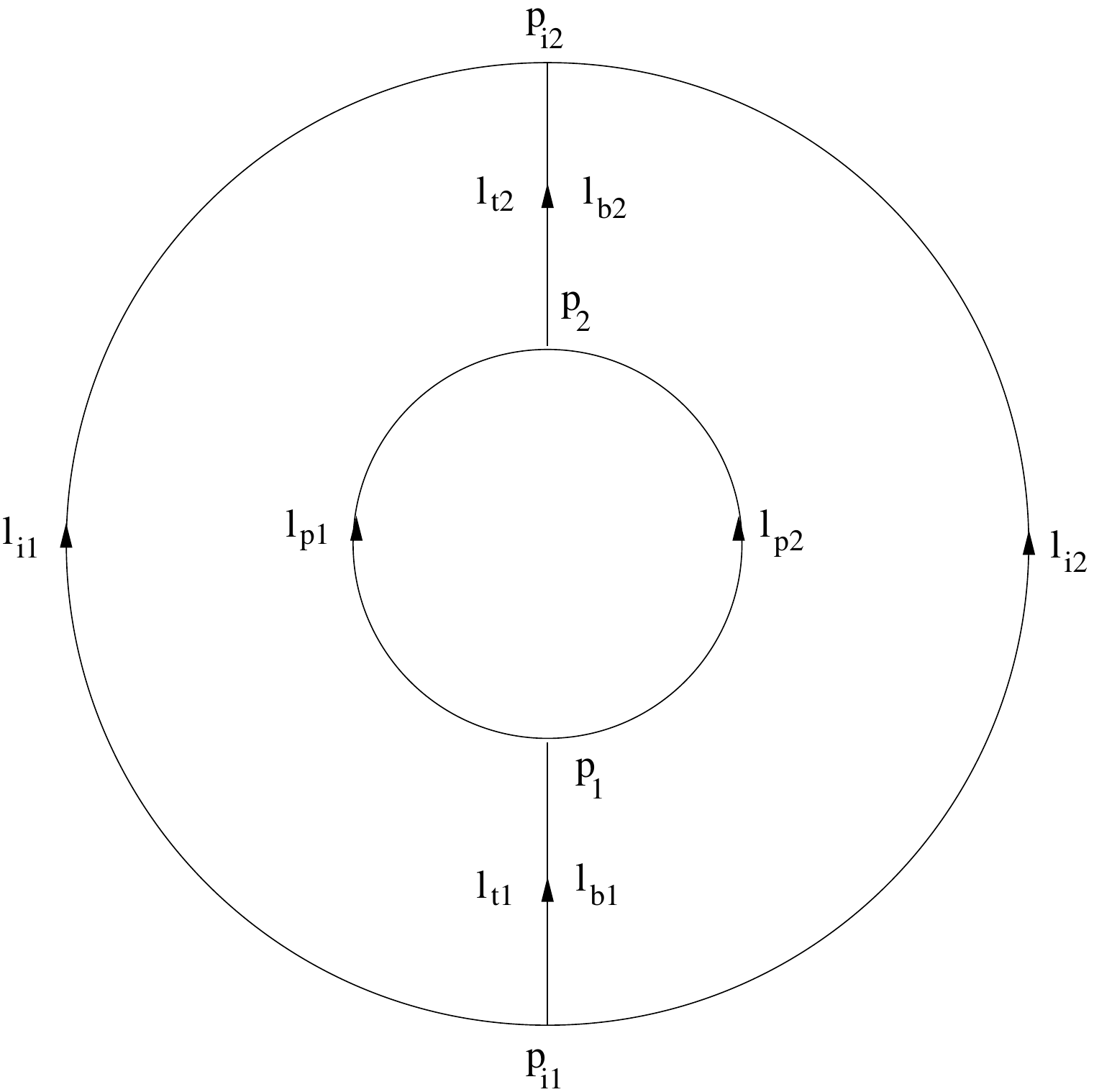}
      \caption{3D world with 2D--disk portal}
    \end{center}
  \end{figure}
The elements are:
\begin{itemize}
\item The points at infinity $pi_1, pi_2$ and the points at the portal
  $p_1, p_2$
\item The lines at infinity $li_1, li_2$ and the lines at the portal
  $lp_1, lp_2$
\item The lines for the top-space $lt_1, lt_2$ and for the bottom
  space $lb_1, lb_2$
\item The planes at infinity $vi_1, vi_2$ and the planes at the portal
  $vp_1, vp_2$
\item The planes for the top-space $vt_1, vt_2$ and for the bottom
  space $vb_1, vb_2$
\item $3-$D worlds for the top space $wt_1, wt_2$ and for the bottom
  space $wb_1, wb_2$.
\end{itemize}

Here is the script (world equations) for the lines:
$$
\partial li_1=pi_2-pi_1, \qquad \partial li_2=pi_2-pi_1, \qquad
\partial lp_1=p_2-p_1,\partial lp_2=p_2-p_1,
$$

$$
\partial lt_1=p_1-pi_1, \qquad \partial lt_2=pi_2-p_2, \qquad \partial
lb_1=\partial lt_1,\partial lb_2=\partial lt_2.
$$

Here is the script for the planes:
$$
\partial vi_1=li_2-li_1, \qquad \partial vi_2=li_2-li_1,
$$

$$
\partial vp_1=lp_2-lp_1=\partial vp_2,
$$

$$
\partial vt_1=lt_1+lp_1+lt_2-li_1,\qquad \partial vt_2=li_2-lt_2-lp_2-lt_1,
$$

$$
\partial vb_1=lb_1+lp_1+lb_2-li_1,\qquad \partial vb_2=li_2-lb_2-lp_2-lb_1.
$$

Here are the scripts for the $3-$D worlds in this case:
$$
\partial wt_1=vp_1+vt_1+vt_2-vi_1,\qquad \partial wt_2=vi_2-vp_2-vt_1-vt_2,
$$

$$
\partial wb_1=vp_2+vb_1+vb_2-vi_1,\qquad \partial wb_2=vi_2-vp_1-vb_1-vb_2.
$$

Closed portal equations (open at $\infty$, closed portal and open at infinity):
$$
\partial wt_1+\partial wt_2-\partial wb_1-\partial wb_2=2(vp_1-vp_2)
$$

Open portal equations:
$$
\partial wt_1+\partial wt_2+\partial wb_1+\partial wb_2=2(vi_2-vi_1).
$$

"Infinity" is like a second portal.

{\bf PORTAL DISCONNECTION:}

$$
\partial vp_3=\partial vp_4=\partial vp_1=\partial vp_2
$$

$$
\partial wt_2\longrightarrow vi_2-vp_3-vt_1-vt_2
$$

$$
\partial wb_2\longrightarrow vi_2-vp_4-vb_1-vb_2
$$

{\bf THIS IS CUTTING ALONG THE PORTAL:}

Re-glueing $vp_1=vp_3,\qquad vp_4=vp_2$, then no more portal.

The portal is an example of a $3-$D geometry that is a non-orientable
$3-$manifold that generalizes the Klein bottle.

%
%
%
%
%
%
%
%
%
%
%
%
%
%
%
%
%


\chapter{Metrics, Duals of Scripts, Dirac Operators}

In this section we will define the notion of the dual of a script,
which may not be a script in general. This will also allow us to introduce a Dirac operator and monogenic scripts.

\section{Metrics}

\begin{definition}\label{definition metric}
  For any script we can define the following metric, called the {\em Kronecker metric}:
  $$
  \inner{C^k_i, C^l_j} = \delta_{k,l} \delta_{i,j};
  $$
  where $C^k_i\in\mathcal{C}_k$ and $C^l_i\in\mathcal{C}_l$.

  This extends through linearity to an inner product on the modules of
  chains. Chains in modules of different dimensions are orthogonal.
\end{definition}

This represents the most canonical example of an inner product. More
general ones can be defined when needed, however this metric allows an
easy and straight-forward way of realising the notion of script
duality which follows.

\section{Duality}

Here we will describe the notion of a dual of a script and start with the definition of the dual of the $\partial$ operator with respect to the metric above.
This dual of the $\partial$ operator will be denoted by $d$. Just as in the
classical case and, via the Stokes formula, the dual boundary operator becomes:
\begin{definition}
  The {\em dual boundary operator}
  $d:\mathcal{C}_k\to\mathcal{C}_{k+1}$ is given by:
  $$
  \inner{d C^k_i, C^{k+1}_j}=\inner{C^k_i,\partial C^{k+1}_j}.
  $$
\end{definition}

\begin{remark}
  Since $\partial^2=0$ it is easy to see that $d^2=0$ as well, where
  $d^2:\mathcal{C}_k\to\mathcal{C}_{k+2}.$
\end{remark}
\begin{example}
  As expected, the dual boundary operator will take the accumulator
  into a sum of points, a point into a sum of lines, and so forth.
\end{example}

For a script:
\begin{align*}
  0 \longleftarrow \mathbb{Z} \overset{\partial}{\longleftarrow} \mathcal{M}_{0}
  \overset{\partial}{\longleftarrow} \mathcal{M}_{1} \overset{\partial}{\longleftarrow}\mathcal{M}_{2}
  \overset{\partial} {\longleftarrow} \cdots  \overset{\partial} {\longleftarrow} \mathcal{M}_{k}
\end{align*}

we now have, using the $d$ boundary operator above:
\begin{align*}
0 \longrightarrow \mathbb{Z} \overset{d}{\longrightarrow} \mathcal{M}_{0}
  \overset{d}{\longrightarrow} \mathcal{M}_{1} \overset{d}{\longrightarrow}\mathcal{M}_{2}
  \overset{d} {\longrightarrow} \cdots \overset{d} {\longrightarrow} \mathcal{M}_{k}=\mathbb{Z}^n,
\end{align*}
where n is the number of connected components of $\mathcal{M}_{k}$.

\begin{definition}
  The {\em dual of the script} is given by :
  \begin{align}
  \label{script-dual}
  \mathbb{Z}^n \overset{\partial'}{\longleftarrow} \mathcal{M'}_{0}
  \overset{\partial'}{\longleftarrow} \mathcal{M'}_{1} \overset{\partial'}{\longleftarrow}\mathcal{M'}_{2}
  \overset{\partial'} {\longleftarrow} \cdots   \overset{\partial'} {\longleftarrow}\mathcal{M'}_{k}=\mathbb{Z}.
\end{align}
where $\partial'=d$ and $\mathcal{M'}_{l}=\mathcal{M}_{k-l-1}$.
\end{definition}

\begin{remark}
  The dual of a script will not, in general, be a script by itself. For example, one must remove free arcs as they are in the kernel of the $d$ operator and $\mathcal{M}_{k}$ shuld be generated by a single cell so that it becomes the accumulator of the dual script $\mathcal{M'}_{0}$.
\end{remark}
\begin{remark}

  We will show that the dual of a script is also a script if and only
  if the original script is orientable.  In this case we have that the
  resulting dual script will be orientable as well, since the dual of
  the accumulator for each connected component will be the equivalent
  of a "volume cell" in the dual.
\end{remark}%

\begin{lemma}\label{lemma unitary script}
The dual of a unitary script is unitary.
\end{lemma}
\begin{proof}
For each $C_i^l\in \mathcal{C}_l$ with $l\neq m $, where $m$ is the dimension of the script, we can determine $d C_i^l$ by calculating the inner product for each $C_j^{l+1}$

\begin{align*}
\langle d C_i^l, C_j^{l+1} \rangle &= \langle C_i^l, \partial C_j^{l+1} \rangle\\
&= \sum_{m} \lambda_m^{(j)} \langle C_i^l, C_m^{l} \rangle\\
&= \lambda_i^{(j)}
\end{align*}

Hence $d C_i^l = \sum_{j} \lambda_{i}^{(j)} C_{j}^{l+1}$.
\end{proof}

\begin{example}[addition in $\mathbb{Z}$]

We have

\begin{displaymath}
0 \leftarrow \mathbb{Z} \overset{\partial}{\leftarrow} \mathcal{M}_0
\end{displaymath}

with $\partial p_j = 1$ for all $p_j\in \mathcal{C}_0$.\\

The dual of this script is given by:

\begin{displaymath}
0 \leftarrow \mathcal{M}_0 \overset{d}{\leftarrow} \mathbb{Z}
\end{displaymath}

where $d 1 = \sum_{j} p_j$ and $dp_j = 0$. The dual is a script if and only if $\mathcal{C}_0 = \{p\}$ and thus we have $d 1  = p$, $dp = 0$.

\end{example}

\begin{example}[An interval]

The following script

\begin{displaymath}
0 \leftarrow \mathbb{Z} \overset{\partial}{\leftarrow} \mathcal{M}_0 \overset{\partial}{\leftarrow} \mathcal{M}_1
\end{displaymath}

where $\mathcal{C}_0 = \{p,q\}$, $\mathcal{C}_1 = \{\ell\}$, $\partial p = \partial q = 1$ and $\partial \ell = p-q$, represents an interval.\\

The dual of this script is given by

\begin{displaymath}
0 \leftarrow \mathcal{M}_1 \overset{d}{\leftarrow} \mathcal{M}_0 \overset{d}{\leftarrow} \mathbb{Z}
\end{displaymath}

where, by following the proof of Lemma \ref{lemma unitary script}, we have $d p = l$, $d q = -l$, $d1 = p+q$ and $d\ell = 0$.

\end{example}

\begin{example}[$m$-sphere and $m$-dimensional ball]

The script

\begin{displaymath}
0 \leftarrow \mathbb{Z} \overset{\partial}{\leftarrow} \mathcal{M}_0 \overset{\partial}{\leftarrow} \mathcal{M}_1 \overset{\partial}{\leftarrow} \ldots \overset{\partial}{\leftarrow} \mathcal{M}_k \overset{\partial}{\leftarrow} \ldots \overset{\partial}{\leftarrow} \mathcal{M}_m
\end{displaymath}

where $\mathcal{C}_l = \{C_1^l,C_2^l\}$, $\partial C_1^0 = \partial C_2^0 = 1$ and $\partial C_j^{k} = C_{1}^{k-1}-C_{2}^{k-1}$ for $k=1,\ldots,m$, $l=0,\ldots,m$, represents an $m$-sphere.

The dual of this script won't be a script itself. But we can still calculate its dual:

\begin{displaymath}
0 \leftarrow \mathcal{M}_{m} \overset{d}{\leftarrow} \ldots \overset{d}{\leftarrow} \mathcal{M}_2 \overset{d}{\leftarrow} \mathcal{M}_1 \overset{d}{\leftarrow} \mathcal{M}_0 \overset{d}{\leftarrow} \mathbb{Z}
\end{displaymath}

with $d 1 = C_1^0+C_2^0$, $dC_{1}^k=C_1^{k+1}+C_2^{k+1}$, $dC_{2}^k=-C_1^{k+1}-C_2^{k+1}$ where $k = 0, \ldots,m-1$ and $dC_1^m = dC_2^m = 0$.

If we want the dual to be a script we add the volume $\mathcal{C}_{m+1} = \{C_{1}^{m+1}\}$, with $\partial C_j^{m+1} = C_{1}^{m}-C_{2}^{m}$, of the $m$-sphere so that we have the script of the $m$-dimensional ball.

Its dual script is then given by

\begin{displaymath}
0 \leftarrow \mathcal{M}_{m+1} \overset{d}{\leftarrow} \ldots \overset{d}{\leftarrow} \mathcal{M}_2 \overset{d}{\leftarrow} \mathcal{M}_1 \overset{d}{\leftarrow} \mathcal{M}_0 \overset{d}{\leftarrow} \mathbb{Z}
\end{displaymath}

with $d 1 = C_1^0+C_2^0$, $dC_{1}^k=C_1^{k+1}+C_2^{k+1}$, $dC_{2}^k=-C_1^{k+1}-C_2^{k+1}$ where $k = 0, \ldots,m-1$, $dC_1^m = C_1^{m+1}$, $dC_2^m = -C_1^{m+1}$ and $dC_1^{m+1} = 0$.

\end{example}

\begin{example}[Simplexes]

We have the following script

\begin{align*}
\mathcal{C}_0 &= \{[0],[1],\ldots,[m]\}\\
\mathcal{C}_1 &= \{[i,j]\vert 0\leq i<j\leq m\}\\
\vdots\\
\mathcal{C}_k &= \{[\alpha_0,\ldots,\alpha_k]\vert 0\leq \alpha_0<\alpha_1<\ldots<\alpha_k\leq m\}\\
\vdots\\
\mathcal{C}_m &= \{[0,1,\ldots,m]\}\\
\end{align*}

where the boundary map is defined as:

\begin{displaymath}
\partial [\alpha_0,\ldots,\alpha_k] = \sum_{j=0}^k (-1)^j [\alpha_0,\ldots,\alpha_k]_j = \sum_{j=0}^k (-1)^j [\alpha_0,\ldots, \alpha_{j-1}, \alpha_{j+1},\ldots,\alpha_k]
\end{displaymath}

The dual boundary operator is

\begin{align*}
d [\alpha_0,\ldots,\alpha_k] =& \sum_{j=0}^{\alpha_0-1} [j,\alpha_0,\ldots,\alpha_k] - \sum_{j=\alpha_0+1}^{\alpha_1-1} [\alpha_0,j,\alpha_1,\ldots,\alpha_k] \\
&+ \ldots + (-1)^k \sum_{j=\alpha_{k-1}+1}^{\alpha_k-1} [\alpha_0,\ldots,\alpha_{k-1},j,\alpha_k]\\
&+ (-1)^{k+1}\sum_{j=\alpha_k+1}^{m} [\alpha_0,\ldots,\alpha_k, j]
\end{align*}

If $\alpha_l = \alpha_{l+1}-1$ then we leave out the sum $\sum_{j=\alpha_{l}+1}^{\alpha_{l+1}-1}$.

\end{example}

\begin{example}[A $2-$torus]

We have the following script for a $2-$torus:

\begin{gather*}
\mathcal{C}_0 = \{p_0, p_1, p_2, p_3\}, \quad \partial p_j = 1, j = 0, 1, 2, 3.\\
\mathcal{C}_1 = \{\ell_1, \ell_2, \ell_3, \ell_4, \ell_5, \ell_6, \ell_7, \ell_8\},\\
\partial \ell_1 = p_1 - p_0, \quad \partial \ell_2 = p_0 - p_1, \quad \partial \ell_3 = p_2 - p_0, \quad \partial \ell_4 = p_0 - p_2\\
\partial \ell_5 = p_3 - p_2, \quad \partial \ell_6 = p_2 - p_3, \quad \partial \ell_7 = p_3 - p_1, \quad \partial \ell_8 = p_1 - p_3\\
\mathcal{C}_2 = \{v_1, v_2, v_3, v_4\},\\
\partial v_1 = \ell_5 + \ell_8 - \ell_1 - \ell_4, \quad \partial v_2 = \ell_6 + \ell_4 - \ell_2 - \ell_8,\\
\partial v_3 = \ell_1 + \ell_7 - \ell_5 - \ell_3, \quad \partial v_4 = \ell_2 + \ell_3 - \ell_6 - \ell_7,\\
\mathcal{C}_3 = \{C\}, \quad \partial C = v_1 + v_2 + v_3 + v_4
\end{gather*}

Here we compute the dual of the $2-$torus described in section 3.6. We have that:

 \begin{gather*}
   d 1 = p_0+p_1+p_2+p_3, \\
 dp_0=  \ell_2 -   \ell_1+ \ell_4- \ell_3 ;\quad dp_1=  \ell_1 -   \ell_2+ \ell_8- \ell_7 ;\quad dp_2=  \ell_3 -   \ell_4+ \ell_6- \ell_5 ;\quad dp_3=  \ell_5 -   \ell_6+ \ell_7- \ell_8  \\
   d\ell_1=v_3-v_1; \quad d\ell_2=v_4-v_2; \quad  d\ell_3=v_4-v_3; \quad d\ell_4=v_2-v_1;  \\
   d\ell_5=v_1-v_3; \quad d\ell_6=v_2-v_4; \quad  d\ell_7=v_3-v_4; \quad d\ell_8=v_1-v_2;  \\
   d v_1 = d v_2 =  d v_3 = d v_4 = C ; \\
   dC=0.
\end{gather*}
\end{example}

\begin{remark}
With the change $\partial'=d$ and $\mathcal{C}'_k=\mathcal{C}_{2-k-1}$ and $v'=p$, $1'=C$, $p'=l$ and $l'=v$ we see that the dual of the torus is represented by the torus itself, with a suitable change of indices.
\end{remark}

\begin{example}[Klein bottle]

We look at the following script of a Klein bottle described in section 3.7.

\begin{gather*}
\mathcal{C}_0 = \{p_0, p_1, p_2, p_3\}, \quad \partial p_j = 1, j = 0, 1, 2, 3.\\
\mathcal{C}_1 = \{\ell_1, \ell_2, \ell_3, \ell_4, \ell_5, \ell_6, \ell_7, \ell_8\},\\
\partial \ell_1 = p_1 - p_0, \quad \partial \ell_2 = p_0 - p_1, \quad \partial \ell_3 = p_2 - p_0, \quad \partial \ell_4 = p_0 - p_2\\
\partial \ell_5 = p_3 - p_1, \quad \partial \ell_6 = p_1 - p_3, \quad \partial \ell_7 = p_3 - p_2, \quad \partial \ell_8 = p_2 - p_3\\
\mathcal{C}_2 = \{v_1, v_2, v_3, v_4\},\\
\partial v_1 = \ell_5 + \ell_8 - \ell_2 - \ell_3, \quad \partial v_2 = \ell_6 - \ell_1 - \ell_4 - \ell_8,\\
\partial v_3 = -\ell_1 + \ell_3 + \ell_7 - \ell_5 , \quad \partial v_4 = \ell_4 - \ell_2 - \ell_6 - \ell_7,
\end{gather*}

Here we compute the dual of the Klein bottle, which will not be a script. We have that:

 \begin{gather*}
   d 1 = p_0+p_1+p_2+p_3, \\
 dp_0=  \ell_2 -   \ell_1+ \ell_4- \ell_3 ;\quad dp_1=  \ell_1 -   \ell_2+ \ell_6- \ell_5 ;\quad dp_2=  \ell_3 -   \ell_4+ \ell_8- \ell_7 ;\quad dp_3=  \ell_5 -   \ell_6+ \ell_7- \ell_8  \\
   d\ell_1=-v_2-v_3; \quad d\ell_2=-v_1-v_4; \quad  d\ell_3=v_3-v_1; \quad d\ell_4=v_4-v_2;  \\
   d\ell_5=v_1-v_3; \quad d\ell_6=v_2-v_4; \quad  d\ell_7=v_3-v_4; \quad d\ell_8=v_1-v_2;  \\
   d v_1 = d v_2 =  d v_3 = d v_4 = 0  .
\end{gather*}

\end{example}

\begin{example}[Extended Klein bottle]

As mentioned in section 3.7, we can add a new cell $v_5$, with $\partial v_5 = \ell_1 + \ell_2$ to the script we obtain a script for the ``extended Klein bottle", which is tight, but not unitary. Just adding this cell won't make the dual into a script:

\begin{gather*}
d1 = p_0+p_1+p_2+p_3\\
d p_0 = -\ell_1 + \ell_2 - \ell_3 +\ell_4, \quad d p_1 = \ell_1 -\ell_2 - \ell_5 + \ell_6,\\
d p_2 = \ell_3 - \ell_4 - \ell_7 + \ell_8, \quad d p_3 = \ell_5 - \ell_6 + \ell_7 - \ell_8\\
d \ell_1 = -v_2 - v_3 + v_5, \quad d \ell_2 = -v_1 - v_4 + v_5, \quad d \ell_3 = -v_1 + v_3, d \ell_4 = -v_2 + v_4\\
d \ell_5 = v_1 - v_3, \quad d \ell_6 = v_2 - v_4, \quad d \ell_7 = v_3 - v_4, \quad d \ell_8 = v_1 - v_2\\
d v_1 = dv_2 = dv_3 = dv_4 = dv_5 = 0
\end{gather*}

But if we add a cell $C$, with $\partial C = v_1+v_2+v_3+v_4+2v_5$, then its dual will be a script, but the script won't be unitary any more. The dual operator will be the same for the points and lines, but it changes for the planes:

\begin{gather*}
d v_1 = dv_2 = dv_3 = dv_4 = C, \quad d v_5 = 2C\\
dC=0
\end{gather*}

As $dv_5 = 2C$, we have that $v_5$ is not a tight cell in the dual script. Hence the dual is not tight and not unitary.

\end{example}

\begin{example}[Projective plane]

As in section 3.8 we start from the following script:

\begin{gather*}
\mathcal{C}_0 = \{p_1,p_2,p_3\}, \quad \partial p_j = 1, j=1,2,3\\
\mathcal{C}_1 = \{ \ell_1,\ell_2,\ell_3,\ell_4,\ell_5,\ell_6\},\\
\partial \ell_1 = p_2 - p_1, \quad \partial \ell_2 = p_1 - p_2, \quad \partial \ell_3 = p_3 - p_1, \quad \partial \ell_4 = p_1 - p_2\\
\partial \ell_5 = p_3 - p_2, \quad \partial \ell_6 = p_2 - p_3,\\
\mathcal{C}_2 = \{v_1,v_2,v_3,v_4\}\\
\partial v_1 = -\ell_2 - \ell_3 + \ell_5, \quad \partial v_2 = -\ell_1 - \ell_4 - \ell_5\\
\partial v_3 = -\ell_1 + \ell_3 + \ell_6, \quad \partial v_4 = -\ell_2 + \ell_4 - \ell_6
\end{gather*}

Here we compute the dual of the projective plane, which won't be a script. We have that:

\begin{gather*}
   d 1 = p_1+p_2+p_3, \\
  dp_1=  \ell_2 -   \ell_1+ \ell_4- \ell_3 ;\quad dp_2=  \ell_1 -   \ell_2+ \ell_6- \ell_5 ;\quad dp_3=  \ell_3 -   \ell_4+ \ell_5- \ell_6  \\
   d\ell_1=-v_2-v_3; \quad d\ell_2=-v_1-v_4; \quad  d\ell_3=v_3-v_1; \quad d\ell_4=v_4-v_2;  \\
   d\ell_5=v_1-v_2; \quad d\ell_6=v_3-v_4;   \\
   d v_1 = d v_2 =  d v_3 = d v_4 = 0.
\end{gather*}
\end{example}

\begin{example}[Extended projective plane]

We can do the same trick as with the Klein bottle and add a $v_5$, with $\partial v_5 = \ell_1 + \ell_2$ to the script we obtain a script for the ``extended projective plane". Calculating its dual yields

\begin{gather*}
d1 = p_1 + p_2 + p_3\\
dp_1 = \ell_2 - \ell_1 + \ell_4 - \ell_3, \quad dp_2 = \ell_1 - \ell_2 - \ell_5 + \ell_6, \quad dp_3 = \ell_3 - \ell_4 + \ell_5 - \ell_6\\
d\ell_1 = -v_2-v_3+v_5, \quad d \ell_2 = -v_1 - v_4 + v_5, \quad d \ell_3 = -v_1 + v_3\\
d\ell_4 = -v_2 + v_4, \quad d\ell_5 = v_1-v_2, \quad d\ell_6 = v_3-v_4\\
dv_1 = dv_2 = dv_3 = dv_4 = dv_5 = 0
\end{gather*}

which still isn't a script. So once again adding a new cell $C\in\mathcal{C}_3$, $\partial C = v_1+v_2+v_3+v_4+2v_5$, then the dual will be a script. The dual operator acts the same on points and lines but for planes and for $C$ we have

\begin{gather*}
dv_1 = dv_2 = dv_3 = dv_4 = C, \quad dv_5 = 2C\\
dC=0
\end{gather*}
\end{example}

\begin{example}[Moebius strip]

After the glueing of the cells done in section 4.4, we get the following script:

\begin{gather*}
\mathcal{C}_0 = \{p_1,p_2,p_3,p_4\}, \quad \partial p_j = 1, j=1,2,3,4\\
\mathcal{C}_1 = \{ \ell_1,\ell_2,\ell_3,\ell_4,\ell_5,\ell_6\},\\
\partial \ell_1 = p_2 - p_1, \quad \partial \ell_2 = p_4 - p_3, \quad \partial \ell_3 = p_3 - p_1 \\
\partial \ell_4 = p_2 - p_3, \quad \partial \ell_5 = p_4 - p_2, \quad \partial \ell_6 = p_1 - p_4,\\
\mathcal{C}_2 = \{v_1,v_2\}\\
\partial v_1 = -\ell_1 + \ell_2 + \ell_3 - \ell_5, \quad \partial v_2 = -\ell_1 - \ell_2 + \ell_4 - \ell_6
\end{gather*}

Note that the dual won't be a script, nevertheless it is given by

\begin{gather*}
d1 = p_1 + p_2 + p_3 + p_4\\
dp_1 = - \ell_1 - \ell_3 + \ell_6, \quad dp_2 = \ell_1 + \ell_4 - \ell_5,\\
dp_3 = -\ell_2 + \ell_3 - \ell_4, \quad dp_4 = \ell_2 + \ell_5 - \ell_6\\
d\ell_1 = -v_1 - v_2, \quad d \ell_2 = v_1 - v_2, \quad d \ell_3 = v_1\\
d\ell_4 = v_2, \quad d\ell_5 = -v_1, \quad d\ell_6 = -v_2\\
dv_1 = dv_2 = 0
\end{gather*}

\end{example}

\begin{example}[Another projective plane model]

Reconsider the Moebius strip with an extra cell $v_3$:

\begin{gather*}
\mathcal{C}_0 = \{p_1,p_2,p_3,p_4\}, \quad \partial p_j = 1, j=1,2,3,4\\
\mathcal{C}_1 = \{ \ell_1,\ell_2,\ell_3,\ell_4,\ell_5,\ell_6\},\\
\partial \ell_1 = p_2 - p_1, \quad \partial \ell_2 = p_4 - p_3, \quad \partial \ell_3 = p_3 - p_1 \\
\partial \ell_4 = p_2 - p_3, \quad \partial \ell_5 = p_4 - p_2, \quad \partial \ell_6 = p_1 - p_4,\\
\mathcal{C}_2 = \{v_1,v_2,v_3\}\\
\partial v_1 = -\ell_1 + \ell_2 + \ell_3 - \ell_5, \quad \partial v_2 = -\ell_1 - \ell_2 + \ell_4 - \ell_6\\
\partial v_3 = \ell_5+\ell_6+\ell_3+\ell_4
\end{gather*}

Here we compute the dual of the other projective plane model described in section 4.4... {\bf insert picture}.
 \begin{gather*}
   d 1 = p_1+p_2+p_3+p_4, \\
 dp_1= - \ell_1 -   \ell_3+ \ell_6 ;\quad dp_2=  \ell_1 +   \ell_4- \ell_5 ;\quad dp_3= - \ell_2 +   \ell_3- \ell_4;\quad dp_4=  \ell_2 +  \ell_5 - \ell_6  \\
   d\ell_1=-v_1-v_2; \quad d\ell_2= v_1-v_2; \quad  d\ell_3=v_1+v_3; \quad d\ell_4=v_2+v_3;  \\
   d\ell_5=-v_1+v_3; \quad d\ell_6=-v_2+v_3;   \\
   d v_1 = d v_2 =  d v_3 = 0.
\end{gather*}

\end{example}

\begin{example}[Another Klein bottle]

We will compute the dual of the Klein bottle, described in section 4.4. Its script is

\begin{gather*}
\mathcal{C}_0 = \{p_1,p_2,p_3,p_4\}, \quad \partial p_j = 1, j=1,2,3,4\\
\mathcal{C}_1 = \{ \ell_1,\ell_2,\ell_3,\ell_4,\ell_5,\ell_6,\ell_7,\ell_8\},\\
\partial \ell_1 = p_2 - p_1, \quad \partial \ell_2 = p_4 - p_3, \quad \partial \ell_3 = p_3 - p_1, \quad \partial \ell_4 = p_4 - p_2,\\
\partial \ell_5 = p_2 - p_3, \quad \partial \ell_6 = p_1 - p_4, \quad \partial \ell_7 = p_1 - p_2, \quad \partial \ell_8 = p_3 - p_4\\
\mathcal{C}_2 = \{v_1,v_2,v_3,v_4\}\\
\partial v_1 = -\ell_1 + \ell_2 + \ell_3 - \ell_4, \quad \partial v_2 = -\ell_1 - \ell_2 + \ell_5  - \ell_6\\
\partial v_3 = -\ell_3 + \ell_4 - \ell_7 + \ell_8, \quad \partial v_4 = -\ell_5 + \ell_6 - \ell_7 - \ell_8
\end{gather*}

Note that the dual won't be a script, nevertheless it is given by

\begin{gather*}
d1 = p_1 + p_2 + p_3 + p_4\\
dp_1 = - \ell_1 - \ell_3 + \ell_6 + \ell_7, \quad dp_2 = \ell_1 - \ell_4 + \ell_5 - \ell_7,\\
dp_3 = -\ell_2 + \ell_3 - \ell_5 + \ell_8, \quad dp_4 = \ell_2 + \ell_4 - \ell_6 - \ell_8\\
d\ell_1 = -v_1 - v_2, \quad d \ell_2 = v_1 - v_2, \quad d \ell_3 = v_1 - v_3, \quad d\ell_4 = -v_1 + v_3,\\
d\ell_5 = v_2 - v_4, \quad d\ell_6 = -v_2 + v_4, \quad d\ell_7 = -v_3 - v_4, \quad d\ell_8 = v_3 - v_4\\
dv_1 = dv_2 = dv_3 = dv_4 = 0
\end{gather*}

\end{example}

\begin{example}[3D-world without 2D-disk portal]\label{example without portal}

Here is the script without the 2D-disk portal:

\begin{gather*}
\mathcal{C}_0 = \{pi_1,pi_2,p_1,p_2\}\\
\partial pi_1 = \partial pi_2 = \partial p_1 = \partial p_2 = 1\\
\hline
\mathcal{C}_1 = \{li_1,li_2,lt_1,lb_1,lt_2,lb_2\}\\
\partial li_1 = pi_2-pi_1 = \partial li_2, \quad
\partial lp_1 = p_2-p_1 = \partial lp_2,\\
\partial lt_1=p_1-pi_1 = \partial
lb_1, \quad \partial lt_2 = pi_2-p_2 = \partial lb_2\\
\hline
\mathcal{C}_2 = \{vi_1,vi_2,vp_1,vp_2,vt_1,vt_2,vb_1,vb_2\}\\
\partial vi_1=li_2-li_1 = \partial vi_2, \quad \partial vp_1=lp_2-lp_1=\partial vp_2\\
\partial vt_1=lt_1+lp_1+lt_2-li_1,\quad \partial vt_2=li_2-lt_2-lp_2-lt_1\\
\partial vb_1=lb_1+lp_1+lb_2-li_1,\quad \partial vb_2=li_2-lb_2-lp_2-lb_1\\
\hline
\mathcal{C}_3 = \{wt_1,wt_2,wb_1,wb_2\}\\
\partial wt_1=vp_1+vt_1+vt_2-vi_1,\quad \partial wt_2=vi_2-vp_2-vt_1-vt_2\\
\partial wb_1=vp_2+vb_1+vb_2-vi_1,\quad \partial wb_2=vi_2-vp_1-vb_1-vb_2
\end{gather*}

Its dual won't be a script, but it is given by

\begin{gather*}
d1 = pi_1+pi_2+p_1+p_2,\\
\hline
dpi_1 = -li_1 - li_2 - lt_1 - lb_1, \quad dpi_2 = li_1 + li_2 + lt_2 + lb_2,\\
dp_1 = -lp_1 - lp_2 + lt_1 + lb_1, \quad dp_2 = lp_1 + lp_2 - lt_2 - lb_2,\\
\hline
dli_1 = -vi_1 - vi_2 - vt_1 - vb_1, \quad dli_2 = vi_1 + vi_2 + vt_2 + vb_2,\\
dlp_1 = -vp_1 - vp_2 + vt_1 + vb_1, \quad dlp_2 = vp_1 + vp_2 - vt_2 - vb_2,\\
dlt_1 = vt_1 - vt_2, \quad dlb_1 = vb_1 - vb_2, \quad dlt_2 = vt_1 - vt_2, \quad dlb_2 = vb_1 - vb_2,\\
\hline
dvi_1 = -wt_1 - wb_1, \quad dvi_2 = wt_2 + wb_2, \quad dvp_1 = wt_1 - wb_2, \qquad dvp_2 = -wt_2 + wb_1\\
dvt_1 = wt_1 - wt_2, \quad dvt_2 = wt_1 - wt_2, \quad dvb_2 = wb_1 - wb_2, \quad dvb_1 = wb_1 - wb_2\\
\hline
dwt_1 = dwt_2 = dwb_1 = dwb_2 = 0
\end{gather*}

\end{example}

\begin{example}[3D-world with 2D-disk portal]

We add two cells $vp_3$ and $vp_4$ to the script of Example \ref{example without portal} such that $\partial vp_3 = \partial vp_4= \partial vp_1 = \partial vp_2$ and change the boundary of $wt_1$ and $wb_2$:

\begin{align*}
\partial wt_2 &= vi_2-vp_3-vt_1-vt_2\\
\partial wb_2 &= vi_2-vp_4-vb_1-vb_2
\end{align*}

Its dual still won't be a script, but it is given by

\begin{gather*}
d1 = pi_1+pi_2+p_1+p_2,\\
\hline
dpi_1 = -li_1 - li_2 - lt_1 - lb_1, \quad dpi_2 = li_1 + li_2 + lt_2 + lb_2,\\
dp_1 = -lp_1 - lp_2 + lt_1 + lb_1, \quad dp_2 = lp_1 + lp_2 - lt_2 - lb_2,\\
\hline
dli_1 = -vi_1 - vi_2 - vt_1 - vb_1, \quad dli_2 = vi_1 + vi_2 + vt_2 + vb_2,\\
dlp_1 = -vp_1 - vp_2 - vp_3 - vp_4 + vt_1 + vb_1, \\
dlp_2 = vp_1 + vp_2 + vp_3 + vp_4 - vt_2 - vb_2,\\
dlt_1 = vt_1 - vt_2, \quad dlb_1 = vb_1 - vb_2, \quad dlt_2 = vt_1 - vt_2, \quad dlb_2 = vb_1 - vb_2,\\
\hline
dvi_1 = -wt_1 - wb_1, \quad dvi_2 = wt_2 + wb_2, \quad dvp_1 = wt_1, \qquad dvp_2 = wb_1\\
dvt_1 = wt_1 - wt_2, \quad dvt_2 = wt_1 - wt_2, \quad dvb_2 = wb_1 - wb_2, \quad dvb_1 = wb_1 - wb_2\\
dvp_3 = -wt_2, \quad dvp_4 = -wb_2\\
\hline
dwt_1 = dwt_2 = dwb_1 = dwb_2 = 0
\end{gather*}
\end{example}

\begin{example}[The pentagon model of the projective plane]

Let $\mathbb{Z}_5$ contain the elements $\{1,2,3,4,5\}$. We start from the following script:

\begin{gather*}
\mathcal{C}_0 = \{p_j\mid j\in \mathbb{Z}_{5}\} \cup \{p_j\mid j\in \mathbb{Z}_{5}\}, \quad \partial p_j = \partial q_j = 1\\
\mathcal{C}_1 = \{k_j\mid j\in \mathbb{Z}_{5}\} \cup \{l_j\mid j\in \mathbb{Z}_{5}\} \cup \{m_j\mid j\in \mathbb{Z}_{5}\}\\
\partial k_j = p_{j+1} - p_j, \quad \partial l_j = p_j - q_j, \quad \partial m_j = q_{j+2} - q_j, \quad j \in\mathbb{Z}_5\\
\mathcal{C}_2 = \{v_j\mid j\in \mathbb{Z}_{5}\}\cup\{v\}\\
\partial v_j = l_j + k_j + k_{j+1} - l_{j+2} - m_j, \quad j\in\mathbb{Z}_5\\
\partial v = m_1 + m_2 + m_3 + m_4 + m_5
\end{gather*}

The dual won't be a script, nevertheless it is given by

\begin{gather*}
d1 = \sum_{j=1}^5 (p_j + q_j)\\
dp_j =  k_{j-1} - k_j + l_j, \quad dq_j = m_{j-2} - m_j - l_j, \quad j\in\mathbb{Z}_5\\
dk_j = v_j + v_{j-1}, \quad dl_j = v_j - v_{j-2}, \quad dm_j = v - v_j, \quad j\in \mathbb{Z}_5\\
dv_1 = dv_2 = dv_3 = dv_4 = dv_5 = dv = 0
\end{gather*}

\end{example}

\section{Duality and Orientability in general:}

\begin{enumerate}[1)]
\item If the script is generated by a single cell of dim $n+1$ then
  $\partial C^{n+1}$ will be an orientation on $C^n.$ The dual complex
  will map any $n-$cell on $\pm C^{n+1}$ which means that $C^{n+1}$
  acts as accumulator for the dual complex which is now a
  script. Necessary for this is that the script is orientable (counter-example:
  Klein bottle)
\end{enumerate}

Observations concerning metrics and duality in the two-dimensional case:

\begin{enumerate}
\item Every tight unitary 2D-script can be identified with a CW-complex
  and, therefore, can be given a unique topology.
\item Every line has two end points and every plane element is a
  polygon.
\item Suppose that a 2D-script is orientable. We then can create a
3-cell $C^3$ such that $\partial C^3 = \sum_j C^2_j.$ In this case we
know that this 3-cell will become the accumulator of the dual complex while also the accumulator of the original complex will become an extra 3-cell of the dual complex.
\end{enumerate}

We know come to a key theorem for three-dimensional scripts arising from two-dimensional scripts with an extra 3D-cell attached.

\begin{theorem}
  If a three-dimensional script consisting of a single 3-cell is unitary and tight and the dual script
  is also unitary and tight then the CW complex associated to the
  $2-$dimensional subscript constructed above becomes an orientable,
  connected, compact (topological) manifold of dimension $2$.
\end{theorem}

\begin{proof}
  Orientability and connectedness follow from $\partial C^3$ being tight
  (generated by a unitary cycle which unique up to the sign).

 From the assumption that $\partial C^3$ is tight follows the orientability and connectedness, indeed, in this case $\partial C^3$ is generated by a unitary cycle which is unique up to the sign.

 We also have that the 2D-subscript is a CW complex due to the fact that it is tight. Remains to be proven
  that every point in this CW-complex has a local neighborhood homeomorphic to the unit disk.

  Case 1. The point belongs to a $2-$dimensional face $C^2$ of the
  CW-complex. In this case since $C^2$ is tight its boundary is homeomorphic to a polygon so its cell is homeomorphic to a disk containing that point.

  Case 2.  The point belongs to one of the lines $C^1_j$. Due to the
  tightness of the dual boundary operator $d$ acting on $C^1_j$, $C^1_j$ connects two
  two-dimensional faces. Therefore,
   $$C^1_j \in \partial C^2_k \quad and \quad C^1_j \in \partial C^2_l, ~k\not= l. $$

  Clearly, the union of $C^1_j$ and these two faces is a neighborhood homeomorphic to the disk.

  Case 3.  The point is one of the points $C^0_j.$ In that case every
  point is the dual of a solid polygon. This means that there are a
  number of lines and $2-$cells issuing from that point that can be
  ordered in the following form: line $\ell_1,$ face $v_1,$ line
  $\ell_2,$ face $v_2,$ \ldots line $\ell_{k-1},$ face $v_{k-1},$ line
  $\ell_1.$

  They form a dual of a polygon which is also a disk. This concludes
  the proof.

\end{proof}

\textbf{Remark:} Every tightness condition has been used in the
proof.

Clearly, the Klein bottle and the projective plane are non-orientable and the dual of their complex is also not a tight script.


\section{Discrete Dirac Operators on Scripts}

\begin{definition}
  The {\em Dirac Operator} in this case is given by
  $\slashed\partial =d+\partial$.
\end{definition}

\begin{remark}
  Here the Dirac operator acts on sums of chains of different
  dimensions.
\end{remark}

\begin{definition}
  The {\em discrete Hodge-Laplace Operator} in this case is given by
  $\slashed\partial^2=d\partial +\partial d$.
\end{definition}

\begin{definition}
  In this case the {\em harmonic and monogenic functions respectively}
  are solutions of $\slashed\partial^2 F=0$ and $\slashed\partial F=0$
  respectively.  They will be linear combinations of all $C_j^k$'s or
  subsets of them.
\end{definition}

\begin{remark}
  In particular there could be harmonic and monogenic functions
  corresponding to chains of a given dimension $k$, for example a
  $k-$chain $F$ is monogenic iff it satisfies the Hodge system:
  $$
  dF=\partial F=0.
  $$
  However, note that not all monogenic chains are sums of solutions of
  the Hodge system.
\end{remark}

\begin{definition}
  The {\em sound} of a script is the sum all eigenvalues of the Hodge
  Laplacian.
\end{definition}

\begin{remark}
As both the boundary operator $\partial$ and its dual $d$ are linear operators, we can describe them using matrices. Doing so, we can determine each entry of these matrices using the metric defined in Definition \ref{definition metric}, which implies that the matrix representation for $d$ is the transpose of the matrix for $\partial$.\\

Hence if we are looking for monogenic or harmonic functions, we can calculate the eigenvalues and eigenvectors of the Dirac operator and the Laplace operator.
\end{remark}

\begin{remark}
We have the following:

\begin{align*}
\langle \slashed{\partial}^2 C_i^k,C_j^l\rangle &=\langle \slashed{\partial} C_i^k,\slashed{\partial} C_j^l\rangle\\
&=\langle \partial C_i^k,\partial C_j^l\rangle + \langle d C_i^k,d C_j^l\rangle\\
\end{align*}

This is only non-zero if $k = l$, this means that the Laplace operator sends $k$-chains to $k$-chains, i.e. the matrix representation of $\slashed{\partial}^2$ will be a block diagonal matrix.
\end{remark}

\begin{example}[Addition in $\Z$]

In the general case where the dual isn't a script, we had $d1 = \sum_{j}p_j$ and $\partial p_j = 1$. Thus we have:

\begin{align*}
\slashed{\partial}(\mu + \sum_{j=1}^n \lambda_j p_j) &= (d+\partial)(\mu + \sum_{j=1}^n \lambda_j p_j)\\
&= \underbrace{\left(\begin{matrix}
0 & 1 & 1 & \cdots & 1\\
1 & 0 & 0 & \cdots & 0\\
1 & 0 & 0 & \cdots & 0\\
\vdots & \vdots & \vdots & \ddots & \vdots\\
1 & 0 & 0 & \cdots & 0
\end{matrix}\right)}_{:=A}\left(\begin{matrix}
\mu\\
\lambda_1\\
\lambda_2\\
\vdots\\
\lambda_n\\
\end{matrix}\right)
\end{align*}

It has eigenvalues $0,\sqrt{n}, -\sqrt{n}$ with multiplicities $n-1, 1,1$ respectively. Thus all corresponding eigenvectors of eigenvalue 0 (Monogenic functions) are linear combinations of $p_1 - p_j$ for $j=2,\ldots,n$.\\
For harmonic functions we need to look at the eigenvalues of

\begin{displaymath}
A^2 = \left(\begin{matrix}
n & 0 & 0 & \cdots & 0\\
0 & 1 & 1 & \cdots & 1\\
0 & 1 & 1 & \cdots & 1\\
\vdots & \vdots & \vdots & \ddots & \vdots\\
0 & 1 & 1 & \cdots & 1
\end{matrix}\right)
\end{displaymath}
which are $0,n$ with multiplicities $n-1, 2$ respectively and the same eigenvectors. The sound here is $2n$.

\end{example}

\begin{example}[An interval]

We have $dp = -dq = \ell$, $d1 = p+q$, $\partial \ell = p-q$ and $\partial p = \partial q = 1$. This yields

\begin{align*}
\slashed{\partial}(\mu + \lambda_1 p + \lambda_2 q + \nu \ell) &= (d+\partial)(\mu + \lambda_1 p + \lambda_2 q + \nu \ell)\\
&= \underbrace{\left(\begin{matrix}
0 & 1 & 1 & 0\\
1 & 0 & 0 & 1\\
1 & 0 & 0 & -1\\
0 & 1 & -1 & 0
\end{matrix}\right)}_{:=A}\left(\begin{matrix}
\mu\\
\lambda_1\\
\lambda_2\\
\nu
\end{matrix}\right)
\end{align*}

which has eigenvalues $\sqrt{2}, -\sqrt{2}$ both with multiplicity 2. Hence the only monogenic function is 0. Squaring this matrix yields

\begin{displaymath}
A^2 = \left(\begin{matrix}
2 & 0 & 0 & 0\\
0 & 2 & 0 & 0\\
0 & 0 & 2 & 0\\
0 & 0 & 0 & 2
\end{matrix}\right)
\end{displaymath}

It is now trivial to see that the sound is 8.

\end{example}

\begin{example}[$m-$sphere]

We have the following matrix representation of the Dirac operator for the $m$-sphere

\begin{align*}
A = \left(\begin{matrix}
 & 1 & 1\\
1 & & & 1 & 1 \\
1 & & & -1 & -1 \\
& 1 & -1 & & & \ddots\\
& 1 & -1 & & & & \ddots\\
& & & \ddots & & & & 1 & 1 &\\
& & & & \ddots & & & -1 & -1 &\\
& & & & & 1 & -1\\
& & & & & 1 & -1
\end{matrix}\right)
\end{align*}

where the empty spaces are zeros. Simple calculation yields

\begin{displaymath}
A^2 = \left(\begin{matrix}
2\\
& 3 & -1\\
& -1 & 3\\
& & & 4\\
& & & & \ddots\\
& & & & & 4\\
& & & & & & 2 & 2\\
& & & & & & 2 & 2
\end{matrix}\right)
\end{displaymath}

Hence it is easy to see that $A^2$ only has 0,4,2 as eigenvalues with multiplicity $1,2m, 2$. The sound is thus $8m+4$. Moreover the only monogenic and harmonic functions are multiples of $C_1^m-C_2^m$. These are the only functions where $\partial F = dF = 0$.\\

If we add the extra cell in order to make the dual a script, we get the following matrix for the Dirac operator on the $m$-dimensional ball:

\begin{align*}
A = \left(\begin{matrix}
 & 1 & 1\\
1 & & & 1 & 1 \\
1 & & & -1 & -1 \\
& 1 & -1 & & & \ddots\\
& 1 & -1 & & & & \ddots\\
& & & \ddots & & & & 1 & 1 &\\
& & & & \ddots & & & -1 & -1 &\\
& & & & & 1 & -1 & & & 1\\
& & & & & 1 & -1 & & & -1\\
& & & & & & & 1 & -1
\end{matrix}\right)
\end{align*}

The matrix for the Laplace operator becomes

\begin{displaymath}
A^2 = \left(\begin{matrix}
2\\
& 3 & -1\\
& -1 & 3\\
& & & 4\\
& & & & \ddots\\
& & & & & 4\\
& & & & & & 3 & 1\\
& & & & & & 1 & 3\\
& & & & & & & & 2
\end{matrix}\right)
\end{displaymath}

which has eigenvalues 2,4 with multiplicity $4,2m$ respectively. The sound is $8m+8$ and there are no non-zero monogenics or harmonics.

\end{example}

\begin{example}[Simplexes]

Note that for a $m$-simplex, we have that the matrix representation of $\slashed{\partial}$ will only contain $0,1,-1$. In order to calculate the eigenvalues of the Laplace operator we will look at

\begin{displaymath}
\alpha_{i,j}^k = \langle  C_i^k,\slashed{\partial}^2 C_j^k\rangle =\langle C_i^k,d\partial C_j^k\rangle + \langle C_i^k,\partial d C_j^k\rangle
\end{displaymath}

If $i=j$, then

\begin{displaymath}
\alpha_{i,j}^k = (k + 1) + (m-k) = m+1
\end{displaymath}

If $[\alpha_0,\ldots,\alpha_k] = C_i^k\neq C_j^k = [\beta_0,\ldots,\beta_k]$, then we have two possibilities:

\begin{enumerate}[(i)]
\item
$|\{\alpha_0,\ldots,\alpha_k\}\setminus\{\beta_0,\ldots,\beta_k\}| \geq 2$. Applying $d$ results in a linear combination of $C_j^k$ with 1 extra number and $\partial$ is a linear combination of $C_j^k$ with 1 number less, so $\partial d$ will change at most only 1 $\alpha_l$ and the same for $d \partial$, thus we have $\langle C_i^k,d\partial C_j^k\rangle = \langle C_i^k,\partial d C_j^k\rangle = 0$.
\item
$|\{\alpha_0,\ldots,\alpha_k\}\setminus\{\beta_0,\ldots,\beta_k\}|=1$, thus there is a unique $l,l'$ such that $\alpha_{l} \neq \beta_{l'}$. Using the same reasoning as in (ii), we only need to look at the term where we deleted $\beta_{l'}$ and replaced it with $\alpha_l$. Hence we have

\begin{align*}
\alpha_{i,j}^k =& \langle \partial C_i^k,\partial C_j^k\rangle + \langle C_i^k,\partial d C_j^k\rangle\\
=& (-1)^{l+l'} + \langle C_{i}^k, \partial d C_j^k \rangle\\
=& (-1)^{l+l'} + (-1)^{l+l'+1}\\
=& 0
\end{align*}

Hence the matrix representation of $\slashed{\partial}^2$ is a diagonal matrix with $m+1$ on the diagonal. Hence there are no harmonic or monogenic functions except for 0. The sound of an $m$-simplex is $2^{m+1}(m+1)$.
\end{enumerate}

\end{example}

\begin{example}[A $2-$torus]

Let $A_k$ denote the matrix representation of $\partial$ applied on $\mathcal{C}_k$, we have:

\begin{align*}
A_0&=\left(\begin{matrix}
1 & 1 & 1 & 1
\end{matrix}\right)\\
A_1&=\left(\begin{matrix}
-1 & 1 & -1 & 1 & 0 & 0 & 0 & 0\\
1 & -1 & 0 & 0 & 0 & 0 & -1 & 1\\
0 & 0 & 1 & -1 & -1 & 1 & 0 & 0\\
0 & 0 & 0 & 0 & 1 & -1 & 1 & -1
\end{matrix}\right)\\
A_2&=\left(\begin{matrix}
-1 & 0 & 1 & 0\\
0 & -1 & 0 & 1\\
0 & 0 & -1 & 1\\
-1 & 1 & 0 & 0\\
1 & 0 & -1 & 0\\
0 & 1 & 0 & -1\\
0 & 0 & 1 & -1\\
1 & -1 & 0 & 0
\end{matrix}\right)\\
A_3&=\left(\begin{matrix}
1\\
1\\
1\\
1
\end{matrix}\right)
\end{align*}

Hence for the matrix representation of $\slashed{\partial}$ we get

\begin{displaymath}
A:=\left(\begin{matrix}
0 & A_0 & 0 & 0 & 0\\
A_0^T & 0 & A_1 & 0 & 0\\
0 & A_1^T & 0 & A_2 & 0\\
0 & 0 & A_2^T & 0 & A_3\\
0 & 0 & 0 & A_3^T & 0
\end{matrix}\right)
\end{displaymath}

Which has eigenvalues $-2\sqrt{2},-2,0,2,2\sqrt{2}$ with multiplicity 2,6,2,6,2 respectively. The monogenic functions are linear combinations of $\ell_1+\ell_2+\ell_5+\ell_6$ and $\ell_3+\ell_4+\ell_7+\ell_8$.\\
The action of $\slashed{\partial}^2$ will have the same eigenvectors with the square of the corresponding eigenvalue. Hence the sound is equal to 80.

\end{example}

\begin{example}[Klein bottle]

Let $A_k$ denote the matrix representation of $\partial$ applied on $\mathcal{C}_k$, we have:

\begin{align*}
A_0&=\left(\begin{matrix}
1 & 1 & 1 & 1
\end{matrix}\right)\\
A_1&=\left(\begin{matrix}
-1 & 1 & -1 & 1 & 0 & 0 & 0 & 0\\
1 & -1 & 0 & 0 & -1 & 1 & 0 & 0\\
0 & 0 & 1 & -1 & 0 & 0 & -1 & 1\\
0 & 0 & 0 & 0 & 1 & -1 & 1 & -1
\end{matrix}\right)\\
A_2&=\left(\begin{matrix}
0 & -1 & -1 & 0\\
-1 & 0 & 0 & -1\\
-1 & 0 & 1 & 0\\
0 & -1 & 0 & 1\\
1 & 0 & -1 & 0\\
0 & 1 & 0 & -1\\
0 & 0 & 1 & -1\\
1 & -1 & 0 & 0
\end{matrix}\right)
\end{align*}

Hence for the matrix representation of $\slashed{\partial}$ we get

\begin{displaymath}
A:=\left(\begin{matrix}
0 & A_0 & 0 & 0\\
A_0^T & 0 & A_1 & 0\\
0 & A_1^T & 0 & A_2\\
0 & 0 & A_2^T & 0\\
\end{matrix}\right)
\end{displaymath}

where the zeros are zero matrices. It has eigenvalues $-2\sqrt{2},-2,-\sqrt{2},0,\sqrt{2},2,2\sqrt{2}$ with multiplicity $2,4,2,1,2,4,2$ respectively. The monogenic functions are multiples of $\ell_3+\ell_4+\ell_5+\ell_6$.\\
The action of $\slashed{\partial}^2$ will have the same eigenvectors with the square of the corresponding eigenvalue. Hence the sound is equal to 72.

\end{example}

\begin{example}[Extended Klein bottle]

If we just add the cell $v_5$ with $\partial v_5 = \ell_1+\ell_2$, then the only thing that changes with respect to the previous example is the matrix $A_2$:

\begin{displaymath}
A_2=\left(\begin{matrix}
0 & -1 & -1 & 0 & 1\\
-1 & 0 & 0 & -1 & 1\\
-1 & 0 & 1 & 0 & 0\\
0 & -1 & 0 & 1 & 0\\
1 & 0 & -1 & 0 & 0\\
0 & 1 & 0 & -1 & 0\\
0 & 0 & 1 & -1 & 0\\
1 & -1 & 0 & 0 & 0
\end{matrix}\right)
\end{displaymath}

Now calculating the eigenvalues of Dirac operator we get $-2\sqrt{2}, -2, -\sqrt{2},0,\sqrt{2},2,2\sqrt{2}$ with multiplicities $2,5,1,2,1,5,2$. The monogenic functions are linear combinations of $\ell_3+\ell_4+\ell_5+\ell_6$ and $v_1+v_2+v_3+v_4+2v_5$. An easy calculations shows that the sound is 76.\\

Adding an extra cell $C$ such that $\partial C = v_1 + v_2 + v_3 + v_4 + 2v_5$, yields an extra matrix:

\begin{displaymath}
A_3 = \left(\begin{matrix}
1\\
1\\
1\\
1\\
2
\end{matrix}\right)
\end{displaymath}

In this case, the total matrix representation of the Dirac operator becomes

\begin{displaymath}
A=\left(\begin{matrix}
0 & A_0 & 0 & 0 & 0\\
A_0^T & 0 & A_1 & 0 & 0\\
0 & A_1^T & 0 & A_2 & 0\\
0 & 0 & A_2^T & 0 & A3\\
0 & 0 & 0 & A_3^T & 0
\end{matrix}\right)
\end{displaymath}

Now we have Eigenvalues $-2\sqrt{2},-2,-\sqrt{2},0,\sqrt{2},2,2\sqrt{2}$ with multiplicities $3,5,1,1,1,5,3$. The sound here is equal to 92 and the monogenic and harmonic functions are multiples of $\ell_3+\ell_4+\ell_5+\ell_6$.

\end{example}

\begin{example}[Projective plane]

Let $A_k$ denote the matrix representation of $\partial$ applied on $\mathcal{C}_k$, we have:

\begin{align*}
A_0&=\left(\begin{matrix}
1 & 1 & 1
\end{matrix}\right)\\
A_1&=\left(\begin{matrix}
-1 & 1 & -1 & 1 & 0 & 0\\
1 & -1 & 0 & 0 & -1 & 1\\
0 & 0 & 1 & -1 & 1 & -1
\end{matrix}\right)\\
A_2&=\left(\begin{matrix}
0 & -1 & -1 & 0\\
-1 & 0 & 0 & -1\\
-1 & 0 & 1 & 0\\
0 & -1 & 0 & 1\\
1 & -1 & 0 & 0\\
0 & 0 & 1 & -1
\end{matrix}\right)
\end{align*}

Hence for the matrix representation of $\slashed{\partial}$ we get

\begin{displaymath}
A:=\left(\begin{matrix}
0 & A_0 & 0 & 0\\
A_0^T & 0 & A_1 & 0\\
0 & A_1^T & 0 & A_2\\
0 & 0 & A_2^T & 0\\
\end{matrix}\right)
\end{displaymath}

where the zeros are zero matrices. It has eigenvalues $-\sqrt{6},-\sqrt{3},-\sqrt{2},\sqrt{2},\sqrt{3},\sqrt{6}$ with multiplicity $3,1,3,3,1,3$ respectively. There are no non-zero monogenic functions.\\
The eigenvalues of $\slashed{\partial}^2$ will be the square of the eigenvalues of $\slashed\partial$. Hence the sound is equal to 54.

\end{example}

\begin{example}[Extended projective plane]
If we just add the cell $v_5$ with $\partial v_5 = \ell_1+\ell_2$, then the only thing that changes is the matrix $A_2$:

\begin{displaymath}
A_2=\left(\begin{matrix}
0 & -1 & -1 & 0 & 1\\
-1 & 0 & 0 & -1 & 1\\
-1 & 0 & 1 & 0 & 0\\
0 & -1 & 0 & 1 & 0\\
1 & -1 & 0 & 0 & 0\\
0 & 0 & 1 & -1 & 0
\end{matrix}\right)
\end{displaymath}

Now calculating the eigenvalues of Dirac operator we get $-\sqrt{6},-2,-\sqrt{3},-\sqrt{2},0,\sqrt{2},\sqrt{3},2,\sqrt{6}$ with multiplicities $3,1,1,2,1,2,1,1,3$. The monogenic functions are multiples of $v_1+v_2+v_3+v_4+2v_5$. An easy calculations shows that the sound is 58.

Adding an extra cell $C$ such that $\partial C = v_1 + v_2 + v_3 + v_4 + 2v_5$, yields an extra matrix:

\begin{displaymath}
A_3 = \left(\begin{matrix}
1\\
1\\
1\\
1\\
2
\end{matrix}\right)
\end{displaymath}

and the total matrix representation of the Dirac operator becomes

\begin{displaymath}
A=\left(\begin{matrix}
0 & A_0 & 0 & 0 & 0\\
A_0^T & 0 & A_1 & 0 & 0\\
0 & A_1^T & 0 & A_2 & 0\\
0 & 0 & A_2^T & 0 & A3\\
0 & 0 & 0 & A_3^T & 0
\end{matrix}\right)
\end{displaymath}

Now we have Eigenvalues $-2\sqrt{2},-\sqrt{6},-2,-\sqrt{3},-\sqrt{2},\sqrt{2},\sqrt{3},2,\sqrt{6},2\sqrt{2}$ with multiplicities $1,3,1,1,2,2,1,1,3,1$. The sound here is equal to 74.

\end{example}

\begin{example}[Moebius strip]

Let $A_k$ denote the matrix representation of $\partial$ applied on $\mathcal{C}_k$, we have:

\begin{align*}
A_0&=\left(\begin{matrix}
1 & 1 & 1 & 1
\end{matrix}\right)\\
A_1&=\left(\begin{matrix}
-1 & 0 & -1 & 0 & 0 & 1\\
1 & 0 & 0 & 1 & -1 & 0\\
0 & -1 & 1 & -1 & 0 & 0\\
0 & 1 & 0 & 0 & 1 & -1
\end{matrix}\right)\\
A_2&=\left(\begin{matrix}
-1 & -1\\
1 & -1\\
1 & 0\\
0 & 1\\
-1 & 0\\
0 & -1
\end{matrix}\right)
\end{align*}

Hence for the matrix representation of $\slashed{\partial}$ we get

\begin{displaymath}
A:=\left(\begin{matrix}
0 & A_0 & 0 & 0\\
A_0^T & 0 & A_1 & 0\\
0 & A_1^T & 0 & A_2\\
0 & 0 & A_2^T & 0\\
\end{matrix}\right)
\end{displaymath}

where the zeros are zero matrices. It has eigenvalues $-2,0,2$ with multiplicity $6,1,6$ respectively. The monogenic functions are multiples of $\ell_3 + \ell_4 + \ell_5 + \ell_6$.\\
The action of $\slashed{\partial}^2$ has eigenvalues $0,4$ with multiplicity $1,12$ respectively. The harmonic functions are also multiples of $\ell_3 + \ell_4 + \ell_5 + \ell_6$. The sound is equal to 48.

\end{example}

\begin{example}[Another projective plane model]

Let $A_k$ denote the matrix representation of $\partial$ applied on $\mathcal{C}_k$, we have:

\begin{align*}
A_0&=\left(\begin{matrix}
1 & 1 & 1 & 1
\end{matrix}\right)\\
A_1&=\left(\begin{matrix}
-1 & 0 & -1 & 0 & 0 & 1\\
1 & 0 & 0 & 1 & -1 & 0\\
0 & -1 & 1 & -1 & 0 & 0\\
0 & 1 & 0 & 0 & 1 & -1
\end{matrix}\right)\\
A_2&=\left(\begin{matrix}
-1 & -1 & 0\\
1 & -1 & 0\\
1 & 0 & 1\\
0 & 1 & 1\\
-1 & 0 & 1\\
0 & -1 & 1
\end{matrix}\right)
\end{align*}

Hence for the matrix representation of $\slashed{\partial}$ we get

\begin{displaymath}
A:=\left(\begin{matrix}
0 & A_0 & 0 & 0\\
A_0^T & 0 & A_1 & 0\\
0 & A_1^T & 0 & A_2\\
0 & 0 & A_2^T & 0\\
\end{matrix}\right)
\end{displaymath}

where the zeros are zero matrices. It has eigenvalues $-2,2$ with multiplicity $7,7$ respectively. There are no non-zero monogenic functions.\\
The action of $\slashed{\partial}^2$ has eigenvalue $4$ with multiplicity $14$. There are no non-zero harmonic functions. The sound is equal to 56.

\end{example}

\begin{example}[Another Klein bottle]

Let $A_k$ denote the matrix representation of $\partial$ applied on $\mathcal{C}_k$, we have:

\begin{align*}
A_0&=\left(\begin{matrix}
1 & 1 & 1 & 1
\end{matrix}\right)\\
A_1&=\left(\begin{matrix}
-1 & 0 & -1 & 0 & 0 & 1 & 1 & 0\\
1 & 0 & 0 & -1 & 1 & 0 & -1 & 0\\
0 & -1 & 1 & 0 & -1 & 0 & 0 & 1\\
0 & 1 & 0 & 1 & 0 & -1 & 0 & -1
\end{matrix}\right)\\
A_2&=\left(\begin{matrix}
-1 & -1 & 0 & 0\\
1 & -1 & 0 & 0\\
1 & 0 & -1 & 0\\
-1 & 0 & 1 & 0\\
0 & 1 & 0 & -1\\
0 & -1 & 0 & 1\\
0 & 0 & -1 & -1\\
0 & 0 & 1 & -1
\end{matrix}\right)
\end{align*}

Hence for the matrix representation of $\slashed{\partial}$ we get

\begin{displaymath}
A:=\left(\begin{matrix}
0 & A_0 & 0 & 0\\
A_0^T & 0 & A_1 & 0\\
0 & A_1^T & 0 & A_2\\
0 & 0 & A_2^T & 0\\
\end{matrix}\right)
\end{displaymath}

where the zeros are zero matrices. It has eigenvalues $-\sqrt{6},-2,-\sqrt{2},0,\sqrt{2},2,\sqrt{6}$ with multiplicity $4,2,2,1,2,2,4$ respectively. The monogenic functions are multiples of $\ell_3 + \ell_4 + \ell_5 + \ell_6$.\\
The action of $\slashed{\partial}^2$ has eigenvalues $0,2,4,6$ with multiplicity $1,4,4,8$ respectively. The harmonic functions are also multiples of $\ell_3 + \ell_4 + \ell_5 + \ell_6$. The sound is equal to 72.
\end{example}

\begin{example}[3D-world without 2D-disk portal]

Let $A_k$ denote the matrix representation of $\partial$ applied on $\mathcal{C}_k$, we have:

\begin{align*}
A_0&=\left(\begin{matrix}
1 & 1 & 1 & 1
\end{matrix}\right)\\
A_1&=\left(\begin{matrix}
-1 & -1 & 0 & 0 & -1 & -1 & 0 & 0\\
1 & 1 & 0 & 0 & 0 & 0 & 1 & 1\\
0 & 0 & -1 & -1 & 1 & 1 & 0 & 0\\
0 & 0 & 1 & 1 & 0 & 0 & -1 & -1
\end{matrix}\right)\\
A_2&=\left(\begin{matrix}
-1 & -1 & 0 & 0 & -1 & -1 & 0 & 0\\
1 & 1 & 0 & 0 & 0 & 0 & 1 & 1\\
0 & 0 & -1 & -1 & 1 & 0 & 1 & 0\\
0 & 0 & 1 & 1 & 0 & -1 & 0 & -1\\
0 & 0 & 0 & 0 & 1 & -1 & 0 & 0\\
0 & 0 & 0 & 0 & 0 & 0 & 1 & -1\\
0 & 0 & 0 & 0 & 1 & -1 & 0 & 0\\
0 & 0 & 0 & 0 & 0 & 0 & 1 & -1
\end{matrix}\right)\\
A_3&=\left(\begin{matrix}
-1 & 0 & -1 & 0\\
0 & 1 & 0 & 1\\
1 & 0 & 0 & -1\\
0 & -1 & 1 & 0\\
1 & -1 & 0 & 0\\
1 & -1 & 0 & 0\\
0 & 0 & 1 & -1\\
0 & 0 & 1 & -1
\end{matrix}\right)
\end{align*}

Hence for the matrix representation of $\slashed{\partial}$ we get

\begin{displaymath}
A:=\left(\begin{matrix}
0 & A_0 & 0 & 0 & 0\\
A_0^T & 0 & A_1 & 0 & 0\\
0 & A_1^T & 0 & A_2 & 0\\
0 & 0 & A_2^T & 0 & A_3\\
0 & 0 & 0 & A_3^T & 0
\end{matrix}\right)
\end{displaymath}

where the zeros are zero matrices. It has eigenvalues $-2\sqrt{2},-2,-\sqrt{2},0,\sqrt{2},2,2\sqrt{2}$ with multiplicity $4,6,2,1,2,6,4$ respectively. The monogenic functions are multiples of $lt_1 - lb_1 - lt_2 + lb_2$.\\
The action of $\slashed{\partial}^2$ has eigenvalues $0,2,4,8$ with multiplicity $1,4,12,8$ respectively. The harmonic functions are also multiples of $lt_1 - lb_1 - lt_2 + lb_2$. The sound is equal to 120.

\end{example}

\begin{example}[3D-world with 2D-disk portal]

In this case $A_2$ and $A_3$ change to

\begin{align*}
A_2&=\left(\begin{matrix}
-1 & -1 & 0 & 0 & 0 & 0 & -1 & -1 & 0 & 0\\
1 & 1 & 0 & 0 & 0 & 0 & 0 & 0 & 1 & 1\\
0 & 0 & -1 & -1 & -1 & -1 & 1 & 0 & 1 & 0\\
0 & 0 & 1 & 1 & 1 & 1 & 0 & -1 & 0 & -1\\
0 & 0 & 0 & 0 & 0 & 0 & 1 & -1 & 0 & 0\\
0 & 0 & 0 & 0 & 0 & 0 & 0 & 0 & 1 & -1\\
0 & 0 & 0 & 0 & 0 & 0 & 1 & -1 & 0 & 0\\
0 & 0 & 0 & 0 & 0 & 0 & 0 & 0 & 1 & -1
\end{matrix}\right)\\
A_3&=\left(\begin{matrix}
-1 & 0 & -1 & 0\\
0 & 1 & 0 & 1\\
1 & 0 & 0 & 0\\
0 & 0 & 1 & 0\\
0 & -1 & 0 & 0\\
0 & 0 & 0 & -1\\
1 & -1 & 0 & 0\\
1 & -1 & 0 & 0\\
0 & 0 & 1 & -1\\
0 & 0 & 1 & -1
\end{matrix}\right)
\end{align*}

So now the Dirac operator has eigenvalues

\begin{align*}
\begin{array}{l|l}
\text{Eigenvalues} & \text{multiplicity}\\
\hline
-\sqrt{8+2\sqrt{2}} & 1\\
-\sqrt{7} & 1\\
-\sqrt{8-2\sqrt{2}}  & 1\\
-2\sqrt{2} & 2\\
-\sqrt{5} & 1\\
-2 & 4\\
-\sqrt{3} & 1\\
-1 & 1\\
0 & 3\\
1 & 1\\
\sqrt{3} & 1\\
2 & 4\\
\sqrt{5} & 1\\
2\sqrt{2} & 2\\
\sqrt{8-2\sqrt{2}}  & 1\\
\sqrt{7} & 1\\
\sqrt{8+2\sqrt{2}} & 1
\end{array}
\end{align*}

The monogenic functions are linear combinations of $lt_1 - lb_1 - lt_2 + lb_2$, $-vi_1 + vi_2-vp_1-vp_2+vp_3+vp_4$ and $-2vi_1+2vi_2-4vp_2+4vp_3-vt_1-vt_2+vb_1+vb_2$.\\
The eigenvalues of the Laplace operator are the square of the eigenvalues of the Dirac operator. The harmonic functions are also linear combinations of $lt_1 - lb_1 - lt_2 + lb_2$, $-vi_1 + vi_2-vp_1-vp_2+vp_3+vp_4$ and $-2vi_1+2vi_2-4vp_2+4vp_3-vt_1-vt_2+vb_1+vb_2$. The sound is equal to 128.

\end{example}

\begin{example}[The pentagon model of the projective plane]

Let $A_k$ denote the matrix representation of $\partial$ applied on $\mathcal{C}_k$, we have:

\setcounter{MaxMatrixCols}{15}
{\footnotesize\begin{align*}
A_0&=\left(\begin{matrix}
1 & 1 & 1 & 1 & 1 & 1 & 1 & 1 & 1 & 1
\end{matrix}\right)\\
A_1&=\left(\begin{matrix}
-1 & 0 & 0 & 0 & 1 & 1 & 0 & 0 & 0 & 0 & 0 & 0 & 0 & 0 & 0\\
1 & -1 & 0 & 0 & 0 & 0 & 1 & 0 & 0 & 0 & 0 & 0 & 0 & 0 & 0\\
0 & 1 & -1 & 0 & 0 & 0 & 0 & 1 & 0 & 0 & 0 & 0 & 0 & 0 & 0\\
0 & 0 & 1 & -1 & 0 & 0 & 0 & 0 & 1 & 0 & 0 & 0 & 0 & 0 & 0\\
0 & 0 & 0 & 1 & -1 & 0 & 0 & 0 & 0 & 1 & 0 & 0 & 0 & 0 & 0\\
0 & 0 & 0 & 0 & 0 & -1 & 0 & 0 & 0 & 0 & -1 & 0 & 0 & 1 & 0\\
0 & 0 & 0 & 0 & 0 & 0 & -1 & 0 & 0 & 0 & 0 & -1 & 0 & 0 & 1\\
0 & 0 & 0 & 0 & 0 & 0 & 0 & -1 & 0 & 0 & 1 & 0 & -1 & 0 & 0\\
0 & 0 & 0 & 0 & 0 & 0 & 0 & 0 & -1 & 0 & 0 & 1 & 0 & -1 & 0\\
0 & 0 & 0 & 0 & 0 & 0 & 0 & 0 & 0 & -1 & 0 & 0 & 1 & 0 & -1
\end{matrix}\right)\\
A_2&=\left(\begin{matrix}
1 & 0 & 0 & 0 & 1 & 0\\
1 & 1 & 0 & 0 & 0 & 0\\
0 & 1 & 1 & 0 & 0 & 0\\
0 & 0 & 1 & 1 & 0 & 0\\
0 & 0 & 0 & 1 & 1 & 0\\
1 & 0 & 0 & -1 & 0 & 0\\
0 & 1 & 0 & 0 & -1 & 0\\
-1 & 0 & 1 & 0 & 0 & 0\\
0 & -1 & 0 & 1 & 0 & 0\\
0 & 0 & -1 & 0 & 1 & 0\\
-1 & 0 & 0 & 0 & 0 & 1\\
0 & -1 & 0 & 0 & 0 & 1\\
0 & 0 & -1 & 0 & 0 & 1\\
0 & 0 & 0 & -1 & 0 & 1\\
0 & 0 & 0 & 0 & -1 & 1
\end{matrix}\right)
\end{align*}}

Hence for the matrix representation of $\slashed{\partial}$ we get

\begin{displaymath}
A:=\left(\begin{matrix}
0 & A_0 & 0 & 0\\
A_0^T & 0 & A_1 & 0\\
0 & A_1^T & 0 & A_2\\
0 & 0 & A_2^T & 0
\end{matrix}\right)
\end{displaymath}

where the zeros are zero matrices. It has the following eigenvalues

\begin{align*}
\begin{array}{l|l}
\text{Eigenvalues} & \text{multiplicity}\\
\hline
-\sqrt{10} & 1\\
-\sqrt{5+\sqrt{5}} & 3\\
-\sqrt{5} & 4\\
-\sqrt{5-\sqrt{5}} & 3\\
-\sqrt{2} & 5\\
\sqrt{2} & 5\\
\sqrt{5-\sqrt{5}} & 3\\
\sqrt{5} & 4\\
\sqrt{5+\sqrt{5}} & 3\\
\sqrt{10} & 1
\end{array}
\end{align*}

There are no non-zero monogenic functions.\\
The action of $\slashed{\partial}^2$ has eigenvalues $2,5-\sqrt{5},5,5+\sqrt{5},10$ with multiplicities 10,6,8,6,2 respectively. Hence there are no non-zero harmonic functions and the sound is equal to 140.

\end{example}


\chapter{Cartesian Products of Scripts}

There are two interesting ways to define the cartesian products of
scripts depending on whether or not one uses the accumulator $"1"$

\section{Cubic cartesian product}

To define the "cubic" cartesian product we start from two scripts:

$$
\mathcal{S}:\qquad 0\leftarrow \mathbb{Z}\leftarrow M_0(\mathcal{C}_0)\leftarrow^{\partial} \dots \leftarrow M_k(\mathcal{C}_k)
$$

$$
\mathcal{S'}:\qquad 0\leftarrow \mathbb{Z}\leftarrow M_0(\mathcal{C}_0')\leftarrow^{\partial} \dots \leftarrow M_k(\mathcal{C}_k'),
$$

and consider the truncated complexes:

$$
\mathcal{S}^{\bullet}:\qquad 0 \leftarrow  M_0(\mathcal{C}_0)\leftarrow^{\partial} \dots
$$

$$
\mathcal{S'}^{\bullet}:\qquad 0\leftarrow  M_0(\mathcal{C}_0')\leftarrow^{\partial} \dots
$$

Then the cubic cartesian product $\mathcal{S}^{\cdot}\times \mathcal{S'}^{\cdot}$ is defined as the complex

$$
\mathcal{S}^{\bullet}\times \mathcal{S'}^{\bullet}:\qquad 0\leftarrow M_0(\mathcal{C}_0'')\leftarrow^{\partial} \dots \leftarrow M_k(\mathcal{C}_k'')
$$
where

$$
\mathcal{C}_k''=\bigcup_{s=0}^k \mathcal{C}_s\times \mathcal{C}_{k-s}'
$$

and
$$
\mathcal{C}_s\times \mathcal{C}_{k-s}'=\{(C_j^s,C_l^{'k-s})\, C_j^s\in \mathcal{C}_s\, C_l^{'k-s}\in\mathcal{C}_{k-s}'\}\, .
$$

Hereby, we define

$$
\partial (C_j^s,C_l^{'k-s})=(\partial C_j^s,C_l^{'k-s})+(-1)^s(C_j^s,\partial C_l^{'k-s})
$$
and also assume linearity
$$
(\lambda C_i^s+\mu C_j^s, C^l)=\lambda (C_i^s, C^l)+\mu(C_j^s, C^l).
$$

Notice that
$$
\partial (\partial C_j^s,C_l^{'k-s})=(-1)^{s-1}(\partial C_j^s,\partial C_l^{'k-s})
$$
and
$$
\partial (C_j^s,\partial C_l^{'k-s})=(\partial C_j^s,\partial C_l^{'k-s})
$$

from which one can readily obtain that $\partial^2=0$ and, since $M_k(\mathcal{C}_k'')$ is the free $\mathbb{Z}-$module over $\mathcal{C}_k''$, we obtain a new complex.

In a similar way, one can also introduce longer cartesian products such as

$$
\mathcal{S}_1^{\bullet}\times \mathcal{S}_2^{\bullet}\times \cdots \times \mathcal{S}_l^{\bullet},
$$

where
$$
\mathcal{S}_j^{\bullet}:\qquad 0 \leftarrow  M_0(\mathcal{C}_{0,j})\leftarrow \dots \leftarrow M_k(\mathcal{C}_{k,j})\, .
$$

The product complex is given by:

$$
\prod_{j=1}^{l}\mathcal{S}_j^{\bullet}:= 0 \leftarrow  M_0(\mathcal{C}''_{0})\leftarrow \dots \leftarrow M_k(\mathcal{C}''_{k})\, ,
$$
whereby
$$
\mathcal{C}_k''=\bigcup_{k_1+\cdots+k_l=k} \mathcal{C}_{k_1,1}\times \cdots \times \mathcal{C}_{k_l,l}\, ,
$$
and
$$
\mathcal{C}_{k_1,1}\times \cdots \times \mathcal{C}_{k_l,l}\,
$$

is the set of $l-$tuples of the form
$$
(C_{j1}^{k_1, 1},\cdots, C_{jl}^{k_l, l})\, ,
$$
with $C_{js}^{k_s, s}\in \mathcal{C}_{k_s,s}$. The boundary operator is given by
$$
\partial (C_{j1}^{k_1, 1},\cdots, C_{jl}^{k_l, l})=(\partial C_{j1}^{k_1, 1},\cdots, C_{jl}^{k_l, l})+(-1)^{k_1}(C_{j1}^{k_1, 1},\partial C_{j2}^{k_2, 2}\cdots, C_{jl}^{k_l, l})+\cdots
+(-1)^{k_1+\cdots+k_{l-1}}(C_{j1}^{k_1, 1},\cdots,\partial C_{jl}^{k_l, l}).
$$

\section{Examples}
We will elaborate on the following examples:
\begin{itemize}
\item[i).] The $l-$dimensional cube $\mathcal{Q}_l^{\bullet}$
\item[ii).] The multicube $\mathcal{Q}_{l_1,\dots, l_s}^{\bullet}$
\item[iii).] The semigrid $\mathbb{N}^m$.
\item[iv).] The grid $\mathbb{Z}^m$
\item[v).] The torus $\mathbb{Z}_{k_1,\dots,k_s}$
\end{itemize}

{\bf i)} The one dimensional cube is the script:

$$
\mathcal{Q}_1:\qquad 0 \leftarrow \mathbb{Z} \leftarrow  M_0(\{p_0,p_1\})\leftarrow M_1(\{ l\})\, ,
$$

with $\partial p_0=\partial p_1=1$ and $\partial l=p_1-p_0$.

The $l-$ dimensional cube $\mathcal{Q}_l^{\bullet}$ is then defined to be:

$$
\mathcal{Q}_l^{\bullet}=\mathcal{Q}_1^{\bullet}\times\mathcal{Q}_1^{\bullet}\times \cdots \times\mathcal{Q}_1^{\bullet}\, ,
$$
for example $\mathcal{Q}_3^{\bullet}=\mathcal{Q}_1^{\bullet}\times\mathcal{Q}_1^{\bullet}\times\mathcal{Q}_1^{\bullet}\,$ consists of the following elements:

\begin{itemize}
\item 8 points:
$$
\mathcal{C}_0''=\{(p_{j_1}, p_{j_2}, p_{j_2})|\, j_1,\,j_2,\,j_3\in\{0,1\}\}\, ,
$$
\item 12 lines:
$$
\mathcal{C}_1''=\{(l,p_{j_1}, p_{j_2})|\, j_1,\,j_2\in\{0,1\}\}\bigcup \{(p_{j_1}, l, p_{j_2})|\, j_1,\,j_2\in\{0,1\}\}\bigcup \{(p_{j_1}, p_{j_2}, l)|\, j_1,\,j_2\in\{0,1\}\}\, ,
$$
\item 6 planes:
$$
\mathcal{C}_2''=\{(l,l,p_{j})|\, j\in\{0,1\}\}\bigcup \{(l,p_{j}, l)|\, j\in\{0,1\}\}\bigcup \{(l,l,p_{j})|\, j\in\{0,1\}\} \, ,
$$
\item 1 volume element:
$$
\mathcal{C}_3''=\{(l, l, l)\}\, .
$$
\end{itemize}

The boundary maps $\partial$ for the cube are evident. For example, we have:
$$
\partial (l, l, l)=(p_1-p_0,l,l)-(l,p_1-p_0,l)+(l,l,p_1-p_0)\, ,
$$

$$
\partial (l,l,p_j)=(p_1-p_0,l,p_j)-(l,p_1-p_0,p_j)\, ,
$$

$$
\partial (l,p_{j_1},p_{j_2})=(p_1-p_0,p_{j_1},p_{j_2})\, ,
$$
and $\partial (p_{j_1},p_{j_2},p_{j_3})=0$ inside $\mathcal{Q}_3^{\bullet}$  while $\partial (p_{j_1},p_{j_2},p_{j_3})=1$ inside the extended $\mathcal{Q}_3$.

{\bf ii)}  For the multicube $\mathcal{Q}_{l_1,\dots, l_s}^{\cdot}$, we start from the script $\mathbb{Z}_k^{\circ}$ with points: $\{ p_0,p_1,\cdots, p_k \}$
and lines $\{ l_1,l_2,\cdots, l_k \}$, and equations $\partial l_j=p_j-p_{j-1}$.

We then define the multicube as
$$
\mathcal{Q}_{l_1,\dots, l_s}^{\cdot}=\mathbb{Z}_{l_1}^{\circ}\times\cdots\times \mathbb{Z}_{l_s}^{\circ}.
$$

{\bf iii)} For the semigrid $\mathbb{N}^m$ we first denote $\mathbb{N}$ to be the script with points $\{ p_0,p_1,\cdots, p_k,\cdots \}$
and lines $\{ l_1,l_2,\cdots, l_k,\cdots \}$, and relations  $\partial l_j=p_j-p_{j-1}$.

Then we can define $\mathbb{N}^m=\mathbb{N}\times \mathbb{N}\times\cdots \mathbb{N}$, where the cartesian product is taken $m$ times.

{\bf iv)} For the grid $\mathbb{Z}^m$ we have as starting point the script $\mathbb{Z}$ with points $\{p_j|\, j\in\mathbb{Z}\}$, lines
$\{l_j|\, j\in\mathbb{Z}\}$ and relations $\partial l_j=p_j-p_{j-1}$.

We then define $\mathbb{Z}^m=\mathbb{Z}\times\mathbb{Z}\times\cdots\times \mathbb{Z}\, .$

One should also note that $\mathbb{N}^{m_1}\times \mathbb{Z}^{m_2}$ exists.

{\bf v)} For the torus  $\mathbb{Z}_{k_1,\dots,k_s}$ we first define $\mathbb{Z}_k$ to be the $k-$polygon with points $\{ p_0,p_1,\cdots, p_k \}$, lines $\{ l_1,l_2,\cdots, l_k, \}$, relations  $\partial l_j=p_j-p_{j-1}$, and glueing constraints $p_k=p_0$, i.e. $\partial l_k=p_0-p_{k-1}$.

The torus is then defined as
$$
\mathbb{Z}_{k_1,\dots,k_s}=\mathbb{Z}_{k_1}\times\mathbb{Z}_{k_2}\times\cdots \times\mathbb{Z}_{k_s}.
$$


\section{Tightness in Cubic Cartesian Products}

Consider a script:

$$
\mathcal{S}:\qquad 0\leftarrow \mathbb{Z}\leftarrow M_0(\mathcal{C}_0)\leftarrow^{\partial}  M_1(\mathcal{C}_1)\leftarrow \dots
$$

and its truncated version:

$$
\mathcal{S}^{\bullet}:\qquad 0\leftarrow^{\partial} M_0(\mathcal{C}_0)\leftarrow^{\partial} M_1(\mathcal{C}_1) \leftarrow^{\partial} \dots
$$

If script $\mathcal{S}$ is tight then every line $l\in\mathcal{C}_1$has two points $\partial l=q-p$. However, within the truncated complex $\mathcal{S}^{\bullet}$,
we have $\partial p=\partial q=0$ and line $l$ is no longer tight because $Z_0(\{p,q\})=\mathbb{Z}^2$.  We see then that the accumulator is quite essential in
script geometry, it implies that tight lines consist of two points.

Consider a second script:

$$
\mathcal{S}':\qquad 0\leftarrow \mathbb{Z}\leftarrow M_0(\mathcal{C}_0')\leftarrow^{\partial}  M_1(\mathcal{C}_1')\leftarrow \dots
$$

Even if $\mathcal{S}$ and $\mathcal{S}'$ are tight, $\mathcal{S}\times \mathcal{S}'$ won't be tight. We need to "attach" an accumulator.


\begin{lemma}
Consider the extension $S^{\prime\prime}$ of $S\times S^{\prime}$ given by
$$
S^{\prime\prime}: 0\leftarrow \mathbb{Z}\leftarrow M_0(\mathcal{C}_0^{\prime\prime})\leftarrow M_1(\mathcal{C}_1^{\prime\prime})\leftarrow \ldots
$$
with $\partial(p,q)=1$ for $(p,q)\in \mathcal{C}_0^{\prime\prime}$. Then $S^{\prime\prime}$ is still a complex.
\end{lemma}

\begin{proof}
We have to prove that for every line $L\in \mathcal{C}_1^{\prime\prime}$, $\partial^2 L=0$. But lines $L$ come in two-forms $L=(l,p^\prime)$ or $L=(p, l^\prime)$and $\partial L$ is given by $(\partial l, p^\prime)=(p-q, p^\prime)$, resp. $(p,\partial l^\prime)=(p,q^\prime-p^\prime)$. Clearly, $\partial(p-q.p^\prime)=\partial((p,p^\prime)-(q,p^\prime))=1-1=0$ and similarly $\partial(p,q^\prime-p^\prime)=0$.
\end{proof}

We now come to the  main theorem.
\begin{theorem}
Let $S$ and $S^\prime$ be tight scripts then the extended cubic Cartesian product $S^{\prime\prime}$ is also tight.
\end{theorem}

\begin{proof}
Let $C_j^k\in \mathcal{C}_k, C^{\prime,l}_k\in \mathcal{C}^{\prime}_l$. Then $\partial(C_j^k,C_i^{\prime, l})=(\partial C_j^k,C_i^{\prime,l})+(-1)^k(C_j^k,\partial C_i^{\prime,l})$ and so $\rb(C_j^k,C_i^{\prime, l})$ consists of two parts $\rb(C_j^k,C_i^{\prime, l})=(\rb C_j^k,C_i^{\prime, l})\cup ((C_j^k,\rb C_i^{\prime, l})$.

Now, consider a cycle $\mathcal{C}_{k+l-1}$ with support in $\rb(C_j^k,C_i^{\prime,l})$. Then
$$
\mathcal{C}_{k+l-1}=  \{ (\mathcal{C}_{k-1}, C_i^{\prime,l})+(C^{k}_j,\mathcal{C}_{l-1}^{\prime}) \}
$$
for some chains $\mathcal{C}_{k-1}\in M_{k-1}(\rb (C_j^k))$ and $\mathcal{C}^{\prime}_{l-1}\in M_{l-1}(\rb(C_i^{\prime, l}))$.

Now,
\begin{eqnarray*}
0=\partial \mathcal{C}_{k+l} & = & (\partial \mathcal{C}_{k-1}, C_i^{\prime,l})+(-1)^{k-1}(\mathcal{C}_{k-1},\partial C_i^{\prime, l})\\
&&+(\partial C_j^k, \mathcal{C}^{\prime}_{l-1})+(-1)^k(C_j^k,\partial \mathcal{C}^{\prime}_{l-1})
\end{eqnarray*}
implies that the parts of different dimensions in this sum are also zero.  In particular,
$$
(\partial \mathcal{C}_{k-1}, C_i^{\prime, l})=(C_j^k, \partial \mathcal{C}^{\prime}_{l-1})=0
$$
and so $\partial \mathcal{C}_{k-1}=0\; \& \; \partial \mathcal{C}^{\prime}_{l-1}=0$.

As $S$ and $S^\prime$ are tight, there are constants $\lambda,\mu$, such that $ \mathcal{C}_{k-1}=\lambda \partial C_j^k$, $ \mathcal{C}^{\prime}_{l-1}=\mu \partial C_i^l$. We also have that
\begin{eqnarray*}
0 & = & (-1)^{k-1}( \mathcal{C}_{k-1},\partial C_i^{\prime, l})+(\partial C_j^k, \mathcal{C}^{\prime}_{l-1})\\
& = & (-1)^{k-1}\lambda (\partial C_j^{k}, \partial C_i^{\prime, l})+\mu (\partial C_j^k,\partial C_i^{\prime, l})
\end{eqnarray*}
so that, in fact $\mu=(-1)^k\lambda$.

But that implies that
\begin{eqnarray*}
 \mathcal{C}_{k+l} & = & \lambda(\partial C_j^k,C_i^{\prime,l})+(-1)^k\lambda (C_j^k,\partial C_i^{\prime, l})\\
& = & \lambda \partial (C_j^k, C_i^{\prime, l})
\end{eqnarray*}
which means that $(C_j^k,C_i^{\prime, l})$ is tight.
\end{proof}

\section{The 3D Lie sphere $LS^3$}

Analysis:

The $3-$dimensional Lie sphere is the non-oriented $3-$dimensional manifold consisting of points in $\mathbb{C}^3$ of the form $e^{i\theta}\underline{\omega}, \theta\in [0,\pi[,\underline{\omega}\in S^2$.

In polar coordinates we put
$$
\underline{\omega}=\cos\psi(\cos\varphi e_1+\sin\varphi e_2)+\sin\psi e_3.
$$

We now create a script for $S^2$ putting
\begin{eqnarray*}
p_1 & \leftrightarrow & e_1, p_2\leftrightarrow -e_1\\
l_1 & \leftrightarrow & \cos\varphi e_1+\sin\varphi e_2, \varphi\in ]0,\pi[\\
l_2 & \leftrightarrow & \cos\varphi e_1-\sin\varphi e_2, \varphi\in ]0,\pi[\\
v_1 & \leftrightarrow & \cos\psi(\cos\varphi e_1+\sin\varphi e_2)+\sin\psi e_3, \varphi\in [0,2\pi [, \psi\in ]0, \frac{\pi}{2}]\\
v_1 & \leftrightarrow & \cos\psi(\cos\varphi e_1+\sin\varphi e_2)-\sin\psi e_3.
\end{eqnarray*}
Clearly, it makes sense to consider the script with $C_0=\{p_1,p_2\}, C_1=\{l_1,l_2\}, C_2=\{v_1,v_2\}$ and relations
\begin{eqnarray*}
& \partial l_1=p_2-p_1, &\partial l_2=p_2-p_1 \\
&\partial v_1=l_2-l_1, & \partial v_2=l_2-l_1
\end{eqnarray*}
and orientation $v_2-v_1$.

For the set of points $e^{i\theta}, \theta\in [0, \pi-\epsilon]$, we name ``1''$\leftrightarrow e^{i0}=1$, ``2''$\leftrightarrow e^{i(\pi-\epsilon)}$, $I\leftrightarrow \mbox{ set }e^{i\theta}, \theta\in ]0,\pi-\epsilon[$ and together  they form the script $J$ with
$$
C_0=\{1,2\}, C_1=\{I\}, \partial I =``2''-``1''.
$$
Next we form the cartesian product of scripts $S^2\times J$ and a script for $LS^3$ will emerge if we take the limit $\epsilon\rightarrow 0$ and make the necessary identifications on gluings for the end-cells.

In extenso for $S^2\in J$ we have 4 points ($(p_1,1),(p_1,2),(p_2,1),(p_2,2)$), 6 lines ($(l_1,1),(l_1,2),(l_2,1),(l_2,2),(p_1,I),(p_2,I)$), and 6 planes ($(v_1,1),(v_2,1),(v_1,2),(v_2,2),(l_1,I),(l_2,I)$), and 2 3D-cells ($(v_1,I),(v_2,I)$).

The script relations are given by
\begin{eqnarray*}
\partial (l_j,k) & = & (p_2,k)-(p_1,k), k=1,2,\\
\partial (v_j,k) & = & (l_2,k)-(l_1,k),\\
\partial (p_j, I) & = & (p_j,2)-(p_j,1),\\
\partial (l_j,I) & = & (p_2,I)-(p_1,I)-(l_j,2)+(l_j,1),\\
\partial (v_j, I) & = & (l_2,I)-(l_1,I)+(v_j,2)-(v_j,1).
\end{eqnarray*}

We taking limit $\epsilon\rightarrow 0$ it seems to we have to glue:

stage 1: $(p_1,2)=(p_2,1), (p_2,2)=(p_1,1)$

stage 2: $(l_1,2)= - (l_2,1), (l_2,2)=-(l_1,1)$

indeed $\partial (l_1,2)=(p_2,2)-(p_1,2)=(p_1,2)-(p_2,1)=-\partial (l_2,1)$ confirms this

stage 3: $(v_1,2)=(v_2,1)$, $(v_2,2)=(v_1,1)$

To simplify the notation we put
$$
(p_j,1)=p_j, (l_j,1)=l_j, (v_j,1)=v_j
$$
and we arrive at the following script for the Lie sphere with 2 points ($p_1,p_2$), 4 lines ($l_1,l_2,(p_1,I),(p_2,I)$), 4 planes ($v_1,v_2,(l_1,I),(l_2,I)$), and 2 volumes ($(v_1,I),(v_2,I)$).

The script equations are those we have for $S^2:\partial l_j=p_2-p_1,\partial v_j=l_2-l_1$ together with new equations
\begin{eqnarray*}
\partial(p_1,I) & = & (p_1,2)-(p_1,1)=p_2-p_1\\
\partial(p_2,I) & = & p_1-p_2,\\
\partial(l_1,I) & = & (p_2,I)-(p_1,I)-(l_1,2)+(l_1,1)\\
\partial(l_2,I) & = & \mathrm{ibid}\\
\partial(v_1,I) & = & (l_2,I)-(l_1,I)+v_2-v_1\\
\partial(v_2,I) & = & (l_2,I)-(l_1,I)+v_1-v_2\\
\partial(v_2,I)-\partial(v_1,I)=2(v_1-v_2)
\end{eqnarray*}

Synthesis: By this we mean the reconstruction of the geometry from the script.

In this case it won't be possible because the script is not tight. We have
$$
\partial(l_1,I)=(p_2,I)-(p_1,I)+l_2+l_1
$$
and $\partial(p_2,I)+\partial l_2=p_1-p_2+p_2-p_1=0$ so $(p_2,I)+l_2$ and $-(p_1,I)+l_1$ are linearly independent cycles inside $\rb (l_1,I)$ and $H_1(\rb(l_1,I))=\mathbb{Z}$. So the cells $(l_1,I), (l_2,I)$ are not tight and the canonical 2D-script is no CW complex. The geometrical offprint of the script does not determine the script and hence also not the geometry of the Lie sphere. For examples when comparing the Lie sphere $(l_1,I)$ has to be a rectangle of which the corners are glued together in opposite way (like the M\"obius band) but the sides are not glued together. It can be seen that $(l_1,I)$ and $(L_2,I)$ together form a 2D-torus which corresponds to the embedding of Lie sphere $LS^2$ inside $LS^3$. But the script itself does not lead to this interpretation; there is not enough information.

In fact we could introduce new tight $2-$cells $V_1,V_2,W_1,W_2$ with
\begin{eqnarray*}
& \partial V_1=(p_2,I)+l_2, & \partial V_2=(p_1,I)-l_1\\
& \partial W_1=(p_2,I)+l_1, & \partial W_2=(p_1,I)-l_2
\end{eqnarray*}
and then put $(l_1,I)=V_1-V_2$, $(l_2,I)=W_1-W_2$.

But that contradicts the Lie sphere because there $(p_2,I)+l_2$ is homologically non-trivial. In fact $(p_2,I)+l_2$ is homologous to $(p_1,I)-l_2$ and the double cycle $(p_2,I)+l_2+(p_1,I)-l_1$ is homologous to the circle $LS^1\leftrightarrow (p_1,I)+(p_2,I)$ ($l_2-l_1=\partial v_1$): the set of points $e^{i\theta},\theta\in[0,2\pi[$.

On the level of planes we see that $(l_1,I)-(l_2,I)$ is a cycle that is in fact $LS^2$ and its homologous to $v_2-v_1$, i.e. $S^2$ and also to $v_1-v_2$ for indeed $2(v_1-v_2)=\partial(v_2-v_1,I)$ is a boundary. But the sphere $v_2-v_1$ itself is not a boundary and hence $LS^2\leftrightarrow (l_1-l_2,I)$ is also not a boundary. It seems that the homological information included in the script is in full agreement with the Lie sphere $LS^3$ in spite of the fact that the script is not tight.

To arrive at a tight script for $LS^3$ we can replace interval $J$ by the ``double interval'' $J_2$ with $C_0=\{"1","2","3"\}, C_1=\{I_1,I_2\}$, $\partial I_1="2"-"1"$, and $\partial I_2="3"-"2"$..

First we consider the cartesian product $S^2\times J_2$, followed by the list of identifications (and new notations):
\begin{eqnarray*}
(p_j,1)=p_j, & (l_j,1)=l_j, & (v_j,1)=v_j\\
(p_j,2)=ip_j, & (l_j,2)=il_j, & (v_j,2)=iv_j\\
(p_1,3)=p_2, & (p_2,3)=p_1 & \\
(l_1,3)=-l_2, & (l_2,3)=-l_1 & \\
(v_1,3)=v_2, & (v_2,3)=v_1. &
\end{eqnarray*}

Apart from the spheres $S^2=\{p_j,l_j,v_j\}$, $iS^2=\{ip_j,il_j,iv_j\}$ we have:

extra lines $(p_j,I_1),(p_j,I_2)$

extra spheres $(l_j,I_1), (l_j, I_2)$

extra volumes $(v_j,I_1), (v_j,I_2)$

Moreover, apart from the relations for the scripts $S^2$ and $iS^2$ we have
\begin{eqnarray*}
& \partial(p_j,I_1)=ip_j-p_j, & \partial(p_j,I_2)=p_{3-j}-ip_j\\
& \partial(l_j,I_1)=(p_2-p_1, I_1)-il_j+l_j, & \\
& \partial(l_j,I_2)=(p_2-p_1,I_2)+l_{3-j}+il_j & \\
& \partial(v_j,I_1)=(l_2-l_1,I_1)+iv_j-v_j, & \\
& \partial(v_j,I_2)=(l_2-l_1,I_2)+v_{3-j}-iv_j. & \\
\end{eqnarray*}

Now the geometry of $LS^3$ is fully determined by the script which is also tight. As an exercise one can study homology of Lie sphere, but that gives no new information compared to the previous non-tight scripts. This examples shows that it is not a good idea to systematically demand tightness; non-tight scripts may be much simpler and still relevant. So the idea of script geometry goes beyond CW-complexes and it is more than just a generalization.

\section{A discrete curvature model}
  \begin{figure}[H]
    \label{fig:curvature}
    \begin{center}
      \includegraphics[scale=0.4]{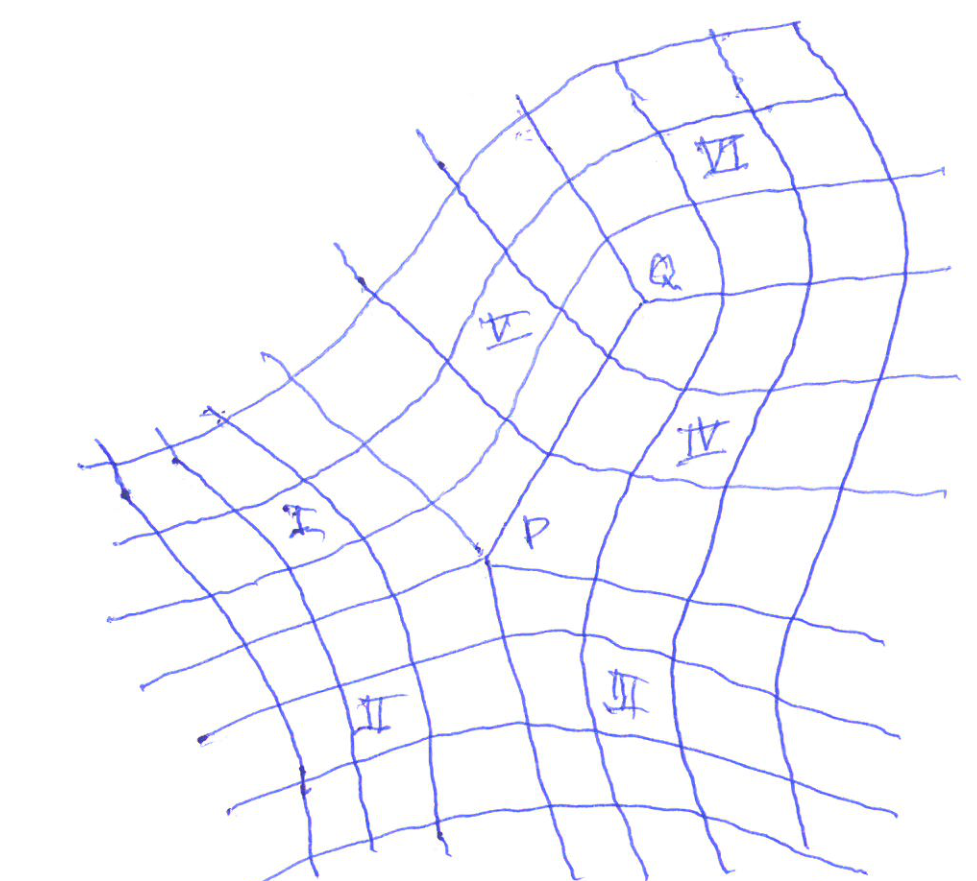}
      \caption{Discrete curvature in 2D}
    \end{center}
  \end{figure}

What is represented in this figure is the geometric offprint of a 2D-script which is unitary and tight. Hence the above figure determines the script up to equivalence. It appears to be a curved or bent 2D-surface with
\begin{itemize}
\item a curvature point $P$ with negative curvature $-1$
\item a curvature point $Q$ with positive curvature $+1$
\end{itemize}
All the other points, lines, and spheres are like in the standard grid $\mathbb{Z}^2$ that is flat.

In Einstein's theory, curvature is linked to gravity and in script geometry, curvature is a property of the geometric offprint alone, not of the actual script itself, that seems closer related to electromagnetism.

But of course we want to have at least one script for which this model is the geometric offprint we will realize it as 2D-surface inside $\mathbb{Z}^3$ where $\mathbb{Z}$ consists of points $a\in\{\ldots,-2,-1,0,1,2,\ldots\}$ and lines $I_j, j\in \mathbb{Z}$ with $\partial I_j="j+1"-"j"$. Hence, in $\mathbb{Z}^3$ itself we have points $(a_1,a_2,a_3)\in\mathbb{Z}^3$, lines $(I_j,a_2,a_3), (a_1,I_j,a_3), (a_1,a_2,I_j)$ with e.g. $\partial(I_j,a_2,a_3)=(j+1,a_2,a_3)-(j,a_2,a_3)$ and planes $(I_j,I_k,a_3)$, $(I_j,a_2,I_k)$, $(a_1,I_j,I_k)$ with e.g. $\partial(I_j,I_k,a_3)=(j+1,I_k,a_3)-(j,I_k,a_3)-(I_j,k+1,a_3)+(I_j,k,a_3)$ and volumes $(I_j,I_k,I_l)$ but we won't be needing theses.

Hence, the script equations follow from the cartesian product $\mathbb{Z}^3$ and all we have to do is to determine how the curvature model fits into $\mathbb{Z}^3$.

We have decomposed it into 6 overlapping zones consisting of points
\begin{itemize}
\item Zone 1: $(a_1,a_2,0) \; \& \; a_1\leq 0, a_2\geq 0$
\item Zone 2: $(a_1,a_2,0)\; \& \; a_1\leq 0, a_2\leq 0$
\item Zone 3: $(a_1,a_2,0) \; \& \; a_1\geq 0, a_2\leq 0$
\item Zone 4: $(a_1,0,a_3) \; \& \; a_1\geq 0, a_3=0,1,2,3$
\item Zone 5: $(0,a_2,a_3) \; \& \; a_2\geq 0, a_3=0,1,2,3$
\item Zone 6: $(a_1,a_2,3) \; \& \; a_1\geq 0, a_2\ge 0$
\end{itemize}
Also the lines can be computed zone by zone, of course these will be overlapping.

This leads to the list:
\begin{itemize}
\item Lines in Zone 1:
\begin{eqnarray*}
(I_j,a_2,0) & \& & j<0,a_2\geq 0\\
(a_1,I_j,0) & \& & j\geq 0, a_1\leq 0
\end{eqnarray*}
\item Lines in Zone 2:
\begin{eqnarray*}
(I_j,a_2,0) & \& & j<0,a_2\leq 0\\
(a_1,I_j,0) & \& & j< 0, a_1\leq 0
\end{eqnarray*}
\item Lines in Zone 3:
\begin{eqnarray*}
(I_j,a_2,0) & \& & j\geq 0,a_2\leq 0\\
(a_1,I_j,0) & \& & j< 0, a_1\geq 0
\end{eqnarray*}
\item Lines in Zone 4:
\begin{eqnarray*}
(I_j,0,a_3) & \& & j\geq 0,a_3=0,1,2,3\\
(a_1,0,I_j) & \& & a_1\geq 0, j=0,1,2
\end{eqnarray*}
\item Lines in Zone 5:
\begin{eqnarray*}
(0,I_j,a_3) & \& & j\geq 0,a_3=0,1,2,3\\
(0,a_2,I_j) & \& & a_2\geq 0, j=0,1,2
\end{eqnarray*}
\item Lines in Zone 6:
\begin{eqnarray*}
(I_j,a_2,3) & \& & j\geq 0,a_2\geq 0\\
(a_1,I_j,3) & \& & j\geq 0, a_1\geq 0
\end{eqnarray*}
\end{itemize}
Similarly, we have the plane elements (no overlapping)
\begin{itemize}
\item Planes in Zone 1:
$$
(I_j,I_k,0)\; ; \;  j<0 \; \& \; k\geq 0
$$
\item Planes in Zone 2:
$$
(I_j,I_k,0)\; ; \; j<0 \; \& \;  k< 0
$$
\item Planes in Zone 3:
$$
(I_j,I_k,0) \; ; \;   j\geq 0\; \& \;  k<0
$$
\item Planes in Zone 4:
$$
(I_j,0,I_k)\; ; \;   j\geq 0 \; \& \;  k= 0,1,2
$$
\item Planes in Zone 5:
$$
(0,I_j,I_k)\; ; \;  j\geq 0 \; \& \;  k= 01,2,3
$$
\item Planes in Zone 6:
$$
(I_j,I_k,3)\; ; \;  j\geq 0 \; \& \;  k\geq 0
$$
\end{itemize}
This fixes the whole script because the script relations follow from the cartesian product. In script geometry much of the creativity lies in finding the best algorithms to describe something. Scripts also involve gravity and electromagnetism combined whereby everything is expressed in terms of chains and their supports.

\section{The ``simplicial'' cartesian product}

We again start from two scripts
\begin{eqnarray*}
S & : & 0\leftarrow \mathbb{Z}\leftarrow M_0(\mathcal{C}_0)\overset{\partial}{\longleftarrow} \cdots\\
S^\prime & : & 0\leftarrow \mathbb{Z}\leftarrow M_0(\mathcal{C}_0^\prime)\overset{\partial}{\longleftarrow}\cdots\\
\end{eqnarray*}
Then the ``simplicial'' cartesian product is defined as the complex
$$
S\times S^\prime : 0\leftarrow\mathbb{Z}=M_{-2}({1})\overset{\partial}{\longleftarrow}M_{-1}(\mathcal{C}_{-1}^{\prime\prime})\overset{\partial}{\longleftarrow}M_{0}(\mathcal{C}_{0}^{\prime\prime})\overset{\partial}{\longleftarrow}\cdots
$$
whereby this time for $k\geq -2$
$$
\mathcal{C}_k^{\prime\prime}=\bigcup_{s=-1}^{k+1} \mathcal{C}_s\times\mathcal{C}_{k-s}^\prime
$$
and as before
$$
\mathcal{C}_s\times \mathcal{C}_{k-s}^\prime=\{(C_j^s,C_l^{\prime, k-s}:\ldots)\}
$$
so for example
\begin{eqnarray*}
\mathcal{C}_{-2}^{\prime\prime} & = & \{(1,1)\}=\{1\}\\
\mathcal{C}_{-1}^{\prime\prime} & = & \{(p_j,1),(1,p_j^\prime):p_j\in \mathcal{C}_0,p_j^\prime\in \mathcal{C}_0^\prime\}\\
\mathcal{C}_{0}^{\prime\prime} & = & \{(p_j,p_l^\prime):p_j\in \mathcal{C}_0,p_j^\prime\in \mathcal{C}_0^\prime\}\cup \{(1,l_j^\prime):l_j^\prime\in \mathcal{C}_1^\prime\}\cup \{(l_j,1):l_j\in C_j\}
\end{eqnarray*}
Just like before the $\partial$-operator is defined by
$$
\partial (C_j^s,C_l^{\prime, k-s})=(\partial C_j^s,C_l^{\prime, k-s})+(-1)^s(C_j^s,\partial C_l^{\prime, k-s})
$$
whereby this time $s\geq -1, k\geq -2$.

So for example
\begin{eqnarray*}
\partial (p_j,1) & = & (1,1)=1\\
\partial (1,p_j^\prime) & = & -(1,\partial p_j^\prime)=-(1,1)=-1\\
\partial (p_j,p_l^\prime) & = & (1,p_l^\prime)+(p_j,1)\\
\partial (1,l_j^\prime) & = & -1(1,\partial l_j^\prime)=-(1,p_\alpha^\prime-p_\beta^\prime)
\end{eqnarray*}
and so on.

So the main difference with the cubic case is that we make explicit use of the accumulations of $S$ and $S^\prime$ within the $\partial$-operator. This means that we have extra cells
$$
(C_j^k,1), (1,C_j^{\prime, k}).
$$
Also the elements $(p_j,p_l^\prime)$ behave like lines rather than points while the sets
$$
\mathcal{C}_0\times\{1\}=\{(p_j,s):\cdots\}, \{ 1 \}\times \mathcal{C}_0^\prime=\{(1,p_j^\prime):\cdots\}
$$
are like two sets of points for which
$$
\partial (p_j,1)=1, \partial (1,p_j^\prime)=-1.
$$
This may seem questionable because normally the boundary of a point is +1. But one can always introduce $-(1,p_j^\prime)$ as ``new points''. Also the dimensions of cells seem to have a shift $-1$, it is
$$
\mathrm{dim}1=-2, \qquad \mathrm{dim}(p_j,1)=-1,\qquad \mathrm{dim}(p_j,p_l^\prime)=0.
$$
while one would rather expect
$$
\mathrm{dim}1=-1, \qquad \mathrm{dim}(p_j,1)=0,\qquad \mathrm{dim}(p_j,p_l^\prime)=1.
$$
One can of course redefine the dimensions of the cells in this way. But that gives problems when defining longer symplicial cartesian products like $S_1\times S_2\times\ldots\times S_l$,
\begin{eqnarray*}
S_j & : & 0\leftarrow \mathbb{Z}\leftarrow M_0(\mathcal{C}_{0,j})\overset{\partial}{\longleftarrow}\ldots\\
\times_{j=1}^l S_j & : & 0\leftarrow \mathbb{Z}=M_{-l}(\mathcal{C}_{-l}^{\prime\prime})\overset{\partial}{\longleftarrow} M_{-l+1}(\mathcal{C}_{-l+1}^{\prime\prime})\leftarrow\ldots \overset{\partial}{\longleftarrow} M_{0}(\mathcal{C}_{0}^{\prime\prime})\leftarrow\ldots
\end{eqnarray*}
whereby for $k\geq -l$
$$
\mathcal{C}_k^{\prime\prime}=\bigcup_{k_1+\ldots+k_l=k}\mathcal{C}_{k_1,1}\times \ldots \times \mathcal{C}_{k_l,l}
$$
and the boundary operators is still given by
\begin{eqnarray*}
&& \partial (C_{j_1}^{k_1,1},\ldots,C_{j_l}^{k_l,l})\\
& = & (\partial C_{j_1}^{k_1,1},\ldots,C_{j_l}^{k_l,l})+\ldots +(-1)^{k_1+\ldots+k_l}(C_{j_1}^{k_1,1},\ldots,\partial C_{j_l}^{k_l,l}),\quad k_1,\ldots,k_l=-1,0,\ldots.
\end{eqnarray*}
Again,  one can consider the elements
$$
(p_j^1,1,\ldots,1),(1,p_j^2,\ldots,1), \ldots, (1, \ldots,p_j^l)
$$
as a partition of the whole set of points and one could redefine the dimension of objects
$$
\mathrm{dim} (\mbox{ object })\rightarrow \mathrm{dim} (\mbox{ object })+l-1
$$
to be in agreement with the general theory of scripts and to renormalize the points, lines, etc. so that $\partial$ point$=+1$. But these are rather cosmetic changes one does not need to make.

\begin{example}
Let
$$
S:0\leftarrow \mathbb{Z}\leftarrow M_0(\{p\})
$$
the script of a single point. Then $S\times S\times S\times S$ has cells
\begin{eqnarray*}
k=-4 & : & (1,1,1,1)=1\\
k=-3 & : & (p,1,1,1),(1,p,1,1),(1,1,p,1),(1,1,1,p)\\
k=-2 & : & (p,p,1,1), (p,1,p,1),(p,1,1,p),(1,p,p,1),(1,p,1,p),(1,1,p,p)\\
k=-1 & : & (p,p,p,1),(p,p,1,p),(p,1,p,p),(1,p,p,p)\\
k=0 & : & (p,p,p,p)
\end{eqnarray*}
\end{example}

The boundary map in this script follows from the general theory. For example,
\begin{eqnarray*}
\partial (1,p,1,1) & = &-1\\
\partial (p,1,p,1) & = & (1,1,p,1)-(p,1,1,1)\\
\partial (p,p,1,p) & = & (1,p,1,p)+(p,1,1,p)-(p,p,1,1), \mbox{ etc. }
\end{eqnarray*}
The geometric offprint of this script is the same as that of a symplex (up to dimensional shift). Hence, the script is tight and hence the script is equivalent to a symplex.

\begin{exercise}
Let $S: 0\leftarrow \mathbb{Z}\leftarrow M_0(\{ p,q\})\leftarrow M_1(\{l\}), \partial l=p-q.$. Prove that $S\times S$ is also a 3D symplex.
\end{exercise}

Concerning tightness we have:
\begin{theorem}
Let $S_1,S_2$ be tight scripts then the symplicial cartesian product $S_1\times S_2$ is also tight.
\end{theorem}

\begin{proof}
Adapt the cubic case (exercise).
\end{proof}
A similar result holds for $S_1 \times S_2 \times \ldots \times S_l$; in fact one can use associativity $(S_1\times S_2)\times S_3=S_1\times (S_2\times S_3)=S_1\times S_2\times S_3$.

\section{The simplicial refinement}
In this subsection we start from a script
$$
S:0\leftarrow \mathbb{Z}\leftarrow M_0\overset{\partial}{\longleftarrow} \ldots
$$
and first consider the  $1$-point script $P$
$$
C_0=\{p\},\partial p=1.
$$
The idea is to construct a canonical symplicial complex such that to every cell $C_j^k$ corresponds a unique chain of symplexes $\sigma(C_j^k)$ so that $\rb(\partial C_j^k)=supp \partial^2C_j^k=\{ 0\}$ which implies $\partial\sigma (C_j^k)=\sigma (\partial C_j^k)$, i.e. the script can be replaced via $\sigma$ by a symplicial complex.

The idea is based on the idea that, within $S\times P$,
$$
\partial (C_j^k,p)=(\partial C_j^k,p)+(-1)^k (C_j^k,1)
$$
so that $(C_j^k,1)$ is cobordant with $(-1)^{k-1}(C_j^k,p)$.

The algorithm is recursive and goes in stages.

\begin{itemize}
\item Stage 0: Identify $S$ with $(S,1)$ consisting of cells $(C_j^k,1)$ of which the dimension is shifted to $k-1$. In particular $(1,1)$ has dimension $-2$. Next for every point $(p_j,1)\in (S,1)$ we put $\sigma(p_j,1)=-(1,p_j)$.
\item Stage $k$: Suppose that we have completed stage $k-1$ and let $(C_j^k,1)\in (C^k,1)$. For each such element we create a new point $p_j^k$ and define
$$
\sigma (C_j^k)=(-1)^{k-1}(\partial C_j^k,p_j^k)\sim (C_j^k,1)
$$
where $\sim$ denotes the cobordism.

The dimensions are the same, it is a chain of symplexes because $\partial C_j^k=\sum$ symplexes and (symplex, point) is a symplex.

Also $\partial\sigma (C_j^k)=(\partial C_j^k,\partial p_j^k)=(\partial C_j^k,1)=\sigma (\partial C_j^k)$ whatever. Finally for every $C_j^{k+1}$, if $C_l^k$ occurs in $\partial C_j^{k+1}$, replace $C_j^k$ by $\sigma(C_l^k)$ and raise the dimension of $\sigma(C_l^k)$ by $+1$.
\item Last stage:  Cancel all elements $(C_j^k,1)$.
\end{itemize}

\begin{example}
Consider the disc $C_0=\{p_1,p_2\}, C_1=\{l_1,l_2\}, C_2=\{v\}, \partial l_1=\partial l_2=p_2-p_1, \partial v=l_2-l_1$.
First we are to replace $p_j\rightarrow -(1,p_j)$ then $\partial l_1=\partial l_2=-(1,p_2)+(1,p_1)$. Next we introduce points $q_1,q_2$ then we replace
$$
\sigma : l_j\rightarrow +(-(1,p_2)+(1,p_1),q_j)=-((1,p_2),q_j)+((1,p_1),q_j)
$$
now we replace
$$
\partial v=l_2-l_1\rightarrow -((1,p_2),q_2)+((1+p_1),q_2)+((1,p_2),q_1)-((1,p_1),q_1).
$$
So taking new point $q$, we obtain
$$
\sigma: v\rightarrow -((1,p_2),q_2)+((1,p_1),q_2)+((1,p_2),q_1)-((1,p_1,q_1),q);
$$
it is the sum of 4 triangles

  \begin{figure}[H]
    \label{fig:curvature1}
    \begin{center}
      \includegraphics[scale=1]{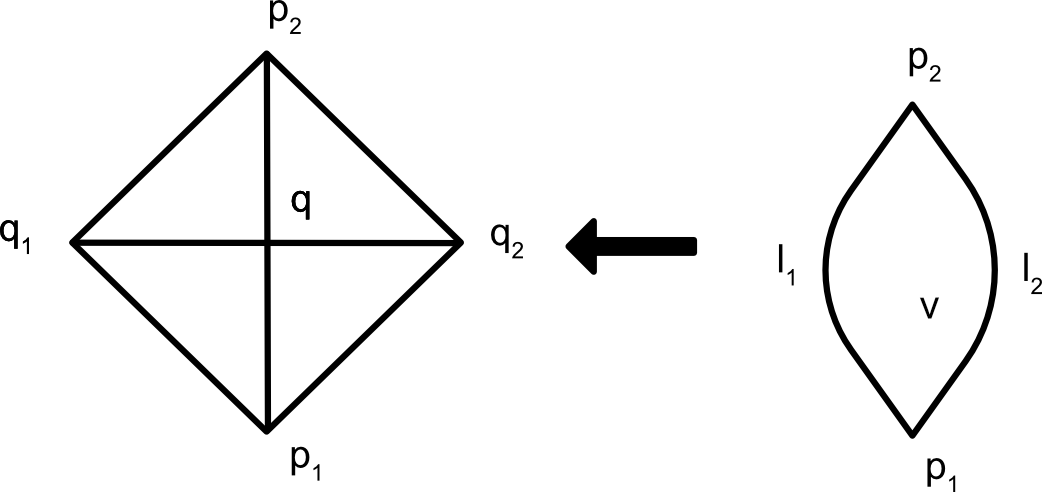}
      \caption{Simplicial refinement of two lines}
    \end{center}
  \end{figure}
\end{example}

It is important to know that every script can be refined to a simplicial complex. Note that the refinement of the sphere $S^2$ is an octahedron. For higher dimensions: similar story.

\vfill\eject

\section*{Affilations}

~

Paula Cerejeiras\\
CIDMA - Center for Research and Development in Mathematics and
Applications, Department of Mathematics,University of Aveiro, Campus Universit\'ario de Santiago, 3810-193 Aveiro, Portugal\\
pceres@ua.pt\\

~

Uwe K\"ahler\\
CIDMA - Center for Research and Development in Mathematics and
Applications, Department of Mathematics,University of Aveiro, Campus Universit\'ario de Santiago, 3810-193 Aveiro, Portugal\\
ukaehler@ua.pt\\

~

Teppo Mertens\\
Clifford Research Group, Department of Mathematical Analysis,
  Ghent University, Galglaan 2, B-9000 Ghent, Belgium\\
teppo.mertens@ugent.be\\

~

Frank Sommen\\
Clifford Research Group, Department of Mathematical Analysis,
  Ghent University, Galglaan 2, B-9000 Ghent, Belgium\\
frank.sommen@ugent.be\\

~

Adrian Vajiac\\
CECHA - Center of Excellence in Complex and Hypercomplex Analysis, Chapman University, One University Drive, Orange CA 92866\\
avajiac@chapman.edu\\

~

MihaelaVajiac\\
CECHA - Center of Excellence in Complex and Hypercomplex Analysis, Chapman University, One University Drive, Orange CA 92866\\
mbvajiac@chapman.edu\\

%
%
%
%


\end{document}